\documentclass{article}
\usepackage[utf8]{inputenc}
\usepackage{setspace,mathtools,amsmath,amssymb,amsthm,fullpage,outline,centernot,slashed}
\usepackage{bbm}
\usepackage{xcolor}
\newcommand{\Z}{\mathbb{Z}}

\newcommand{\R}{\mathbb{R}}
\newcommand{\N}{\mathbb{N}}

\newcommand{\inv}{{^{-1}}}
\newcommand{\half}{\frac 12}

\def\bal#1\eal{\begin{align*}#1\end{align*}}
\newtheorem{theorem}{Theorem}
\newtheorem{condition}{Condition}
\newtheorem{conjecture}{Conjecture}
\newtheorem{lemma}[theorem]{Lemma}
\newtheorem{corollary}[theorem]{Corollary}
\newtheorem{proposition}[theorem]{Proposition}
\newtheorem{claim}[theorem]{Claim}
\title{Stability and instability of traveling wave solutions \\ to nonlinear wave equations}
\author{John Anderson\footnote{Princeton University, jranders@math.princeton.edu}\ \,and Samuel Zbarsky\footnote{Princeton University, szbarsky@math.princeton.edu}}
\date{\today}

\begin{document}

\maketitle

\begin{abstract}
In this paper, we study the stability and instability of plane wave solutions to semilinear systems of wave equations satisfying the null condition. We identify a condition which allows us to prove the global nonlinear asymptotic stability of the plane wave. The proof of global stability requires us to analyze the geometry of the interaction between the background plane wave and the perturbation. When this condition is not met, we are able to prove linear instability assuming an additional genericity condition. The linear instability is shown using a geometric optics ansatz.
\end{abstract}

\tableofcontents
%Comments (I (Sam) am planning to deal with these)

%Make sure that there is stuff about intuition in the section where we recover the energies--which terms you need to use $E_1$ to bound.

%things to recheck

%Section \ref{sec:Geometry}

%Proposition \ref{prop:KlaiSob}

%Proposition \ref{prop:ConeSob}

%Lemma~\ref{lem:RecoverIntegratedKlaiSob}

%A note on the meaning of genericity: Suppose we have the system
%\bal
%\box \phi_1=\partial_y\phi_1\partial_{u'}\phi_3-\partial_y\phi_3\partial_{u'}\phi_1\\
%\box \phi_2=\partial_y\phi_2\partial_{u'}\phi_3-\partial_y\phi_3\partial_{u'}\phi_2\\
%\box \phi_3=0\\
%\box \phi_4=0
%\eal
%and suppose we are linearizing realtive the solution $\phi_i=f_i(u')$.
%Then we will have that the linearization gives us
%\[
%B_y=\begin{bmatrix}
%   f_3' & 0 & -f_1' & 0\\
%   0 & f_3' & -f_2' & 0\\
%   0 & 0 & 0 & 0\\
%   0 & 0 & 0 & 0
% \end{bmatrix}
%\]
%which has has eigenvalues $f_3',f_3',0,0$. Thus there is no traveling wave for this system which will make the linearization satisfy the genericity condition. Therefore, when we say "generic", we mean with respect to joint perturbations of the traveling wave and the system of equations (WHY IS IT GENERIC IN THAT SENSE? BECAUSE WE CAN PERFORM LINEAR TRANSFORMATION/CHANGE OF BASIS TO GET $f_1'(u_0)\ne 0$ AND ALL OTHER COMPONENTS have derivatives 0 at that point. THEN IT IS EASY TO SET $B_y$ TO WHATEVER WE WANT IN AN OPEN BALL).

\section{Introduction}
In this paper, we study the stability and instability of plane wave solutions to semilinear systems of wave equations satisfying the null condition in $\R^{3 + 1}$. More precisely, we say that a bilinear form $m:\R^{3+1}\times\R^{3+1}\to\R$ is a \emph{semilinear null form} if $m(\upsilon,\upsilon)=0$ for any null vector $\upsilon$. We say that $m$ is the \emph{standard null form} if $m(\upsilon,\upsilon)=\upsilon_\alpha\upsilon^\alpha$ (where we are using Einstein summation notation and the Minkowski metric on $\R^{3+1}$). We consider systems of semilinear wave equations of the form
\begin{equation}\label{eq:zerothversion}
\square \phi_i=\sum  m_{ij\ell}(\nabla \phi_j,\nabla\phi_\ell) = \sum m_{i j \ell} (d \phi_j,d \phi_\ell)
\end{equation}
where each $m_{ij\ell}$ is a null form and where $\nabla$ is the Minkowski gradient. There is an isomorphism between null forms acting on vectors and null forms acting on covectors in Lorentzian manifolds induced by the metric, and we are using $m_{i j \ell}$ to represent both. To avoid ambiguity between $m_{ij\ell}$ and $m_{i\ell j}$, we will without loss of generality take $m_{ij\ell}(\nabla \phi_j,\nabla\phi_\ell)=m_{i\ell j}(\nabla \phi_\ell,\nabla\phi_j)$. We can thus think of the solution as being a vector valued function $\phi$ taking values in $\R^k$ where $k$ is the number of equations. We consider special traveling wave solutions of the form $f(t - x)$ in $(t,x,y,z)$ coordinates where $f$ is a smooth, compactly supported vector valued function, and we ask under what condition this solution is stable under smooth and small perturbations supported in the unit ball. Depending on the structure of the linearization around the background traveling wave, we are able to show either global nonlinear stability or linear instability, where the linear instability result requires a genericity assumption as well.

The global nonlinear stability of hyperbolic equations has been extensively studied in recent years. Here and later, by ``stability'' we mean stability with respect to smooth, localized perturbations. For simplicity, we will only consider perturbations which are in fact supported in the unit ball. The equations we discuss will sometimes be quasilinear, but will usually be semilinear, having nonlinearities depending on the gradient. We will not discuss equations where the nonlinearity depends on the function.

The stability mechanism that is often taken advantage of for hyperbolic equations is decay. Using, say, the fundamental solution of the wave equation on Minkowski space $\R^{n + 1}$, we can see that solutions to the homogeneous linear wave equation with compactly supported initial data will decay like $t^{-{n - 1 \over 2}}$. Starting with \cite{Kl80}, Klainerman was able to establish the global stability of nonlinear perturbations of the homogeneous wave equation on Minkowski space. He was only able to treat quadratic nonlinearities in dimensions $n + 1$ with $n \ge 6$. The decay of the linear wave equation was very important in that work, which only used the dispersive estimate for the wave equation. Then, in \cite{Kl85}, Klainerman developed a way of proving pointwise decay for the wave equation on Minkowski space that is very well adapted to studying nonlinear equations. This involved commuting the equations with the vector fields generating the Lorentz group along with the scaling vector field, with is a conformal Killing vector field on Minkowski space. This use of weighted commutation vector fields has been very successful in several contexts. This can already be seen in the Klainerman-Sobolev inequality, which first appeared in \cite{Kl85}. In the physical case of $\R^{3 + 1}$, this states that
\[
|f|(t,r,\omega) \le {C \over (1 + t + r) (1 + |t - r|^{{1 \over 2}})} \sum_{|\alpha| \le 2} \Vert \Gamma^\alpha f \Vert_{L^2 (\Sigma_t)}
\]
where the $\Gamma^\alpha$ are strings of vector fields consisting of translations, Lorentz vector fields, and the scaling vector field. Using this inequality, Klainerman established the global stability of the trivial solution for wave equations with quadratic nonlinearities in $\R^{n + 1}$ for $n \ge 4$.

For $n = 3$, general quadratic nonlinearities may result in finite time singularity formation (see \cite{John79}). However, for certain quadratic nonlinearities satisfying the \textit{null condition}, the trivial solution is still globally stable. The null condition was first described by Klainerman in \cite{Kl82}. Then, in \cite{Kl86}, Klainerman was able to establish the global stability of the trivial solution to nonlinear wave equations in $\R^{3 + 1}$ satisfying the null condition using techniques based on the use of commutation vector fields. Christodoulou was also able to establish the global stability of the trivial solution to nonlinear equations satisfying the null condition in \cite{Chr86} using a conformal compactification. Then, in the monumental work \cite{ChrKl93}, Christodoulou-Klainerman were able to prove the global nonlinear stability of the trivial solution (Minkowski space) of the Einstein vacuum equations under suitable perturbations. In this work, Christodoulou-Klainerman had to find a form of the null condition present in the Einstein equations.

While the conditions under which the trivial solution to semilinear and quasilinear wave equations are globally stable are fairly well understood, much less is known about the stability of other solutions. This has become an active area of research. For example, there has been much work over the last few years in understanding the stability of black hole solutions to the Einstein vacuum equations.\footnote{This has been an extremely active area of research, with a huge number of results. The interested reader can look at, for example, \cite{DRSR16}, \cite{DHR19}, and \cite{KlSz17} and the references therein.} There has also been some interest recently in understanding when plane symmetric solutions are stable. In particular, for the \textit{relativistic membrane equation} (also known as \textit{timelike or Lorentzian minimal/maximal surface equation} or \textit{Lorentzian vanishing mean curvature flow}), $\cite{AbbresciaWong19}$ show stability of compactly supported planar waves in dimension $3$ and higher while $\cite{LiuZhou19}$ show it in dimension 2 and higher (they also consider a more general class of traveling wave solution than the class considered in this paper which is schematically of the form $(a + b y) f(t - x)$, see Section~\ref{sec:RelatedDirections} for further discussion). For the wave map equation, $\cite{AbbresciaChen19}$ prove stability for certain planar waves.

Here we work in $3 + 1$ dimensions and consider semilinear systems of wave equations satisfying the 
%classical
null condition. We shall take smooth traveling wave solutions to the system of equations, and we shall consider the behavior under smooth perturbations which are small and supported in the unit ball. The linearized equations will be the wave equation on Minkowski space plus first order terms which depend on the background traveling wave.

To be more precise, we assume that we have a system of semilinear wave equations of the form~\eqref{eq:zerothversion} and that the smooth vector valued function $f(t - x)$ supported when $|t - x| \le 1$ is a solution of the system. We then ask whether the solution $f$ is stable. In order to study this problem, we linearize around the plane wave $f$. We shall now give two simple examples of such systems. The two examples will be different structurally, and they will show the structural difference which leads to either global nonlinear stability or linear instability. The first example is
\begin{equation}\label{eq:firstexample}
    \begin{aligned}
    \Box \phi_1 &= m(d \phi_1,d \phi_2)
    \\ \Box \phi_2 &= m(d \phi_1,d \phi_1),
    \end{aligned}
\end{equation}
while the second example is
\begin{equation}
    \begin{aligned}
    \Box \phi_1 &= \partial_t \phi_1 \partial_y \phi_2 - \partial_y \phi_1 \partial_t \phi_2
    \\ \Box \phi_2 &= m(d \phi_1,d \phi_1)
    \end{aligned}
\end{equation}
where $m$ is the standard null form. In both cases, we note that $\phi_1 = 0$ and $\phi_2=f(t - x)$ is a solution where $f(t - x)$ is smooth and supported when $|t - x| \le 1$. The linearization of the first system around the background traveling wave is
\begin{equation}\label{eq:firstexamplelinearized}
    \begin{aligned}
    \Box \psi_1 &= 2 (\partial_t + \partial_x) \psi_1 f'(t - x)
    \\ \Box \psi_2 &= 0,
    \end{aligned}
\end{equation}
while the nonlinear equation for the perturbation is
\begin{equation}
    \begin{aligned}
    \Box \psi_1 &= - 2 (\partial_t + \partial_x) \psi_1 f'(t - x) + m(d \psi_1,d \psi_2)
    \\ \Box \psi_2 &= m(d \psi_1,d \psi_1).
    \end{aligned}
\end{equation}
Meanwhile, the linearization of the second system around the background traveling wave is
\begin{equation}\label{eq:secondexamplelinearized}
    \begin{aligned}
    \Box \psi_1 &= -2 \partial_y \psi_1 f'(t - x)
    \\ \Box \psi_2 &= 0,
    \end{aligned}
\end{equation}
and the nonlinear equation for the perturbation is
\begin{equation}
    \begin{aligned}
    \Box \psi_1 &= -2 \partial_y \psi_1 f'(t - x) + \partial_t \psi_1 \partial_y \psi_2 - \partial_y \psi_1 \partial_t \psi_2
    \\ \Box \psi_2 &= m(d \psi_1,d \psi_1).
    \end{aligned}
\end{equation}
The first example is representative of the structure present in systems where the plane wave solution is globally nonlinearly stable, while the second example is representative of the structure present where the traveling wave is linearly unstable. The distinguishing feature is that the traveling wave only excites the standard null form in the first example, while it excites one of the antisymmetric null forms in the second. We thus have the following rough version of the main Theorems.

\begin{theorem} [Rough Version of Main Results]
Let $f(t - x)$ be a smooth solution supported where $|t - x| \le 1$ to a semilinear system of wave equations satisfying the null condition as in \eqref{eq:zerothversion}. If the nonzero components of the vector valued function $f$ only excite the standard null form, then the solution $f$ is globally nonlinearly stable under small perturbations supported in the unit ball. If $f$ excites other null forms and if an additional genericity condition holds, then the solution $f$ is linearly unstable.
\end{theorem}
The precise statements of the main results are given in Section \ref{sec:maintheorems}, as is the genericity condition. We note that the plane wave is actually also allowed to excite some antisymmetric null forms for stability to be true, see Condition~\ref{cond:unexciting} for the precise condition.

In general, the form of the linearized equations determine the behavior of the perturbation. When the traveling wave only excites the standard null form, we shall be able to prove global nonlinear stability (see Condition~\ref{cond:unexciting} for the precise statement of the condition). In a generic subset of the cases when this condition is not met, we shall be able to show linear instability (see Condition $2$ and Theorem~\ref{thm:generalinstab}). We shall also be able to show blow up in very specific cases (see Proposition~\ref{prop:Bessel}). We now briefly discuss the strategies we follow in proving these results as well as the main difficulties we encounter along the way.

For proving stability, the first part of the proof involves removing the first order terms in the equation, reducing the problem to a nonlinear perturbation of the linear wave equation on Minkowski space. These first order terms are in general very dangerous and can lead to exponential growth (see Theorem~\ref{thm:generalinstab}). Condition $1$ allows us to apply a certain transformation to the system which makes this reduction. This transformation seems to be similar in spirit to something done in \cite{LiuZhou19}. Indeed, in that paper, certain exponential weights are used when doing energy estimates. In this paper, we shall renormalize the quantities using a certain exponential transformation that will remove the first order terms from the beginning (see Section~\ref{sec:transformation}). This transformation is easier to identify in the example \eqref{eq:firstexample} above.
We simply take $\gamma_1 = e^{ -f(t - x)} \psi_1$ and $\gamma_2 = \psi_2$. In terms of $\gamma_1$ and $\gamma_2$, the linearized equations \eqref{eq:firstexamplelinearized} take the form
\begin{equation}
    \begin{aligned}
    \Box \gamma_1 &= 0
    \\ \Box \gamma_2 &= 0.
    \end{aligned}
\end{equation}

The strategy we follow to prove global stability of the renormalized problem is similar in spirit to the strategy used in \cite{AbbresciaWong19} because both strategies follow the commuting vector field method, where decay is established by commuting with weighted vector fields.
The main difficulty we encounter arises from the fact that the nonlinearities in the renormalized equations do not have constant coefficients, but rather have coefficients that are functions of the plane wave solution. For example, in the example case \eqref{eq:firstexample} above, the nonlinear equations for the transformed quantities $\gamma_1$ and $\gamma_2$ become
\begin{equation}
    \begin{aligned}
    \Box \gamma_1 &= e^{ -f(t - x)} m(d(e^{ f(t - x)} \gamma_1),d \gamma_2)
    \\ \Box \gamma_2 &= m(d(e^{ f(t - x)} \gamma_1),d(e^{ f(t - x)} \gamma_1)).
    \end{aligned}
\end{equation}
This creates difficulties because a large number of the weighted vector fields are not well behaved when they hit the traveling wave. Indeed, these vector fields are specifically those vector fields arising from symmetries of Minkowski space, and the traveling wave solution does not respect all of the symmetries that Minkowski space has (for example, the traveling wave is far from being spherically symmetric). We thus have to be careful about terms where the weighted commutation fields hit the coefficients in the nonlinearity (which come from the background traveling wave). Controlling these terms involves two main observations, both of them geometric. The first is that weighted commutation fields will not introduce the a priori worst weights when hitting the traveling wave. Indeed, the weights will be of size at most $\sqrt{t}$, which is much better than the a priori worst case of $t$. This observation was already used in \cite{AbbresciaWong19}.

The second observation is that the volume of interaction between a compactly supported plane wave and the perturbation is of size $t$. Solutions of the wave equation arising from smooth data supported in the unit ball are roughly evenly distributed in an annular shell of thickness $2$ and inner radius $t - 1$ at time $t$. This means that the volume of self interaction in a nonlinearity for such a perturbation is of size $t^2$. Thus, the plane wave reduces the volume of interaction by a factor of $t$. This gain in the volume is very important in the proof of stability. Both of these geometric observations are made precise in Section~\ref{sec:Geometry}.

When the traveling wave does excite a null form other than the one coming from the Minkowski metric, we will not have a transformation to reduce the equation to a nonlinear perturbation of the linear wave equation on Minkowski space, though we will still use the transformation to get rid of terms $(\partial_t+\partial_x)\psi$ terms. In this case, the remaining first order terms will result in growth of the linearized equation like $e^{\sqrt{t}}$. This will be shown using a geometric optics argument (see Section~\ref{sec:unstable}). This approach can be motivated by considering the equation
\begin{equation} \label{eq:ExpGrowth}
\Box \psi = f (t - x) (\partial_t - \partial_x) \psi,
\end{equation}
where $f$ is some compactly supported function. Geometric optics can be used to show that solutions to this linear equation can experience exponential growth. However, the linearized equations actually look like
\[
\Box \psi = f (t - x) \partial_y \psi
\]
instead (see \eqref{eq:secondexamplelinearized}), and the argument must be modified. In fact, we are able to show that given any $T$ sufficiently large, we can construct solutions which grow exponentially with a rate comparable to ${1 \over \sqrt{T}}$, and they grow at this rate for time $T$. This gives us a solution which is of size $e^{C_1 \sqrt{T}}$ for all $T$ sufficiently large, giving us the desired growth. Along the way, we also prove that solutions grow no faster than $e^{C_2 \sqrt{t}}$, meaning that the solutions we construct saturate the growth rate (although our values of $C_1$ and $C_2$ may not be optimal).

In this paper, we specifically study the stability and instability of plane wave solutions to semilinear systems of wave equations satisfying the null condition. However, we believe that there are several other related questions and areas of study. We believe that some of the tools used in this paper can be used to study the stability of plane wave solutions to other equations, such as the Lorentzian minimal surface equation and semilinear systems of wave equations satisfying the weak null condition. We also note that the decay rates we obtain are not sharp and can be improved, and we describe how analogous results for the stability and instability of traveling wave solutions propagating in different directions is a corollary of the main results in this paper. We shall further discuss these and other topics for further study in Section~\ref{sec:RelatedDirections}.

In Section~\ref{sec:notation}, we list notations and conventions used throughout the paper. In Section~\ref{sec:maintheorems}, we state the main theorems. In Section~\ref{sec:RelatedDirections}, we list related problems for which our techniques should work. In Sections~\ref{sec:transformation} through \ref{sec:ClosingEnergy}, we prove the stability result in $\R^{3+1}$. In Section~\ref{sec:unstable}, we prove the instability results. In Section~\ref{sec:RelatedDirections}, we discuss problems related to the ones studied in this paper, and we discuss if and how the strategies used in this paper may be applicable to these other problems.

\section{Acknowledgements}
The authors wish to express their gratitude to their advisors, Sergiu Klainerman and Igor Rodnianski, for many helpful discussions and suggestions. They wish to thank Igor Rodnianski for suggesting this problem to them. They also wish to thank Yakov Shlapentokh-Rothman for his useful suggestions. Samuel Zbarsky's work was supported by the National Science Foundation Graduate Research Fellowship Program under Grant No. DGE-1656466. Any opinions, findings, and conclusions or recommendations expressed in this material are those of the authors and do not necessarily reflect the views of the National Science Foundation.

\section{Notation and Coordinate Systems}\label{sec:notation}
In this section, we shall list the notation and coordinate systems used throughout the paper.

We will use Einstein summation notation by default. We also introduce the following notation for various quantities where $r^2 = x^2 + y^2 + z^2$.
\begin{align}
\Box&=\partial_t^2-\partial_x^2-\partial_y^2-\partial_z^2\\
u&=t-r, \qquad v=t+r.\\
\partial_u&=\frac 12 (\partial_t-\partial_r), \qquad \partial_v=\frac 12 (\partial_t+\partial_r)\\
u'&= t - x, \qquad v'= t + x.\\
\partial_{u'}&=\frac 12 (\partial_t - \partial_x), \qquad \partial_{v'}=\frac 12 (\partial_t+\partial_x)\\
\Sigma_t&\text{ is a level set of $t$}.\\
S_t&=\Sigma_t\cap\{u\ge -1\}\cap |u'|\le 1\text{ is the intersection of the travelling wave support and the}\nonumber\\
&\qquad\qquad\qquad\qquad\qquad\qquad\qquad\qquad\qquad \text{domain of influence of the perturbation}.\\
\mathcal C_u&\text{ is an outgoing cone (a level set of $u$)}.\\
\mathcal C_u (v_1,v_2)&=\mathcal C_u\cap \{v_1 \le v \le v_2\}\text{ is a truncated outgoing cone}\label{notation:truncatedcone}\\
\Gamma&\text{ denotes a translation, Lorentz, or scaling vector field, and } Z \text{ denotes the set of such vector fields}\\
\Gamma^{\alpha}&\text{ is shorthand for a string of such vector fields}\\
\Gamma^{\alpha,i}&\text{ is shorthand a string of such vector fields of which at most $i$ are weighted}
\end{align}
We use the following coordinate systems where $\omega\in S^2$:
\[
(t,x,y,z),\qquad(t,r,\omega),\qquad(u,v,\omega),\qquad(u',v',y,z)
\]
\[
L_x^p\text{ is the }L^p\text{ norm over all of }\R^3\text{, not just over the $x$-axis}
\]
$\delta\in (0,1/100)$ is a constant fixed for the duration of the proof. The constant $C$ may depend on the nonlinearity in the equation, on the wave profile $f$, and on the small constant $\delta$ used in various estimates. $C$ may be used to denote different constants on different lines.

The expression $\overline{\partial} f$ is meant to denote derivatives which are good in the sense that they decay faster. We shall refer to such quantities as \emph{good derivatives}. It is well known (see, for example, \cite{LukNotes}) that derivatives tangent to the light cone decay faster, so it is precisely these quantities which we refer to as good derivatives. Natural quantities to call good derivatives are $\partial_v f$ and $\nabla_{\Sigma_t} f - (\partial_r f) \partial_r = \slashed{\nabla}_r f$ where $\nabla_{\Sigma_t} f = \partial_x f \partial_x + \partial_y f \partial_y + \partial_z f \partial_z$ is the spatial gradient induced on the $\Sigma_t$ hypersurfaces, and where $\slashed{\nabla}_r$ denotes the gradient induced on the two dimensional sphere of radius $r$. It is useful to instead work with $\partial_v$ and a collection of vector fields tangent to the spheres. A very useful set of vector fields to consider on the spheres is the set of rescaled rotation vector fields $e_{x y}$, $e_{x z}$, and $e_{y z}$ where $e_{x y} = {1 \over r} \Omega_{x y} = {1 \over r} (x \partial_y - y \partial_x)$, and similarly for $e_{x z}$ and $e_{y z}$. We note that
\[
|\slashed{\nabla}_r f| \le C_1 \left (|e_{x y} f| + |e_{x z} f| + |e_{y z} f| \right ) \le C_2 |\slashed{\nabla}_r f|.
\]
Thus, when we refer to good derivatives or write the expression $\overline{\partial}$, we shall be referring to one of $\partial_v$, $e_{x y}$, $e_{x z}$, or $e_{y z}$.

We also recall that the Lorentz vector fields are given by the rotations and boosts
\[
x^\alpha \partial_\beta + x^\beta \partial_\alpha,
\]
while the scaling vector field is given by
\[
t \partial_t + r \partial_r.
\]

\section{Main Theorems}\label{sec:maintheorems}

We work with the equation
\begin{equation}\label{eq:firstversion}
\square \phi_i=m_i(\nabla \phi,\nabla \phi) = m_i(d \phi,d \phi)
\end{equation}
where $\phi:\R^{k}\times[0,T)\to\R^N$,
where $m_i=\sum_{j,\ell} m_{ij\ell}(d \phi_j,d \phi_\ell) $ and where each $m_{ij\ell}$ is a constant-coefficient null form (we recall that we are using $m_i$ to denote the null form acting on both vectors and covectors). We now take a smooth function $f : \R\rightarrow \R^N$ supported on $[-1,1]$ and note that $f(t - x)$ solves the equation \eqref{eq:firstversion}, that is
\[
\Box f_i = m_i (d f,d f).
\]
The mechanism that drives instability is, intuitively, the interaction of derivatives of the plane wave with $\partial_y$ or $\partial_z$ derivatives of the perturbation. It is absent as long as we have
\begin{condition}\label{cond:unexciting}
For every $j$ and $\ell$ such that $f_j \ne 0$ or $f_\ell \ne 0$, we have that
\[
m_{i j \ell} (d t - d x,d y) = m_{i j \ell} (d t - d x,d z) = 0.
\]
\end{condition}
This condition is stated for $k=3$, and henceforth we will always assume $k=3$ unless stated otherwise. In higher dimension, we can take the analogous condition. When Condition~\ref{cond:unexciting} holds, there is still a dispersive mechanism for the linearized equations just like for the usual wave equation, and this dispersive mechanism gives us global stability.
Condition~\ref{cond:unexciting} is satisfied for any $f$ for certain null forms $m_{ij\ell}$, including the standard null form. However, the condition is phrased so as to allow other null forms, as long as certain components of $f$ are 0.
Under this condition, we have the following stability result.

\begin{theorem}\label{thm:main}
Suppose $f : \R\rightarrow \R^3$ is a smooth function supported on $[-1,1]$. Moreover, suppose that the $m_{i j \ell}$ and $f$ satisfy Condition~\ref{cond:unexciting}. Then $f(t - x)$ is a globally asymptotically stable solution of \eqref{eq:firstversion}. More precisely, there exists some $\epsilon>0$ such that, whenever $g,h$ are compactly supported in $B(0,1)$ with $||g||_{H^3}+||h||_{H^2} < \epsilon$ and we are given the initial-value problem $\eqref{eq:firstversion}$ with initial data
\begin{align*}
\phi\mid_{t=0}&=f(-x)+g\\
\partial_t\phi\mid_{t=0}&=f'(-x)+h\\
\end{align*}
then there is some constant $C$ and a global in time solution $\phi=f(t-x)+\psi$ with
\begin{align*}
||\psi(t,\cdot)||_{H^3}&<C(||g||_{H^3}+||h||_{H^{2}})\\
|\partial \psi(t,r,\omega)| &\le {C \epsilon \over (1 + t + r)^{1 - \delta} (1 + |t - r|)^{{1 \over 2}}},
\end{align*}
where the constants are allowed to depend on $f$. Moreover, the bootstrap assumptions in Section~\ref{BtstrpA} are shown to propagate for the quantity $\gamma$, where $\gamma$ is an appropriately renormalized quantity that is described in Section~\ref{sec:transformation}.
\end{theorem}
We also have a generic linear instability result when the equation does not satisfy Condition~\ref{cond:unexciting}. We introduce the following condition for a real matrix $M$
\begin{condition}\label{cond:M}
$M$ has at least one eigenvalue with nonzero real part.
\end{condition}
\[
\{A\in \R^{n\times n}\mid A\text{ does not satisfy Condition~\ref{cond:M}}\}
\]
is lower-dimensional, so the condition is in a sense generic. At the beginning of Section~\ref{sec:unstable}, we will discuss in more detail the sense in which this condition is generic as applied to the following theorem.
\begin{theorem}\label{thm:generalinstab}
Suppose that we have an equation of the form
\begin{equation}\label{eq:exprootnoforcing}
\square\eta=B_y(t-x)\partial_y\eta+B_z(t-x)\partial_z\eta
\end{equation}
and assume that there is some value $u_0$ so that some real linear combination of $B_y(u_0),B_z(u_0)$ satisfies Condition~\ref{cond:M}. Then there is some $K>0$ so that for all $m>0$, there is some $c>0$ so that for all $T>0$, there is some solution $\eta$ with
\[
||\eta(T,\cdot)||_\infty\ge c\exp(K\sqrt{T})||\eta(0,\cdot)||_{H^m}.
\]
Furthermore, in the case, that $N=1,B_z=0,B_y=f'$ for some wave profile $f$, we have that we can take $K=\frac{1}{\sqrt{2}}|f|_{1/2}-\epsilon$ for any $\epsilon>0$.
\end{theorem}
We also have a more precise instability and blowup result for the nonlinear problem when the right hand side is of a very specific form, see Proposition~\ref{prop:Bessel}.

As a result of the transformation described in Section~\ref{sec:transformation}, the linearized equations can always be reduced to the form in the statement of Theorem~\ref{thm:generalinstab}. At least one of the matrices $B_y$ or $B_z$ will be nonzero whenever Condition~\ref{cond:unexciting} is not satisfied. Condition~\ref{cond:M} is the structural condition we need to be true in order to show linear instability, see Section~\ref{sec:unstable}.

%We note that $\eqref{cond:wavestable}$ are necessary to ensure that $f$ is a solution (and will also be used in the proof of stability).
%To show stability, we let $\psi=\eta-f(t-x)$ and note that we need to solve the equation (here and later we use the Einstein summation convention)
%\[
%\square \psi_i=2a_{ij\ell}f'_j(t-x)(\partial_t+\partial_x)\psi_\ell+m_i(\partial\psi,\partial\psi)
%\]
%where $a_{ij\ell}=m_{ij\ell}(\partial_t-\partial_x,\partial_t+\partial_x)$. We want to prove that if $\psi$ is small initially, then it remains small.

\section{Transformation}\label{sec:transformation}
We shall now calculate the linearization of the equations around the solution $f$ in order to understand the behavior of the perturbation and the stability and instability properties of this solution.
When Condition~\ref{cond:unexciting} is satisfied, the first step of proving the global stability described in Theorem~\ref{thm:main} will be to apply a transformation to the equation which gives us an appropriately renormalized quantity $\gamma$. We shall study the equation satisfied by $\gamma$ instead of $\phi$. The instability is further studied in Section~\ref{sec:unstable}.

%For this calculation, we will assume that $m_i$ is symmetric in the two arguments, in the sense that
%\[
%m_{ij\ell}(\vec x_1,\vec x_2)=m_{i\ell j}(\vec x_2,\vec x_1)
%\]
%which can always be achieved by symmetrizing $m_i$ without changing the equation.
We let $\psi = \phi - f(t - x)$. We calculate now the equation satisfied by $\psi$. Note that $m_i (d f,d f) = 0$ because $m_i$ is a null form and note that $\Box f = 0$. We have that
\begin{equation} \label{eq:PerturbationEquation}
    \begin{aligned}
    \Box \psi_i &= \Box \phi_i - \Box f_i (t - x) = m_i (d \phi,d \phi) = m_i (d \psi + d f,d \psi + d f)
    \\ &= 2 m_{i j \ell} (d t - d x,d t + d x) f_j' (t - x) (\partial_t + \partial_x) \psi_\ell +2 m_{i j \ell} (d t + d x,d t - d x) f_\ell' (t - x) (\partial_t + \partial_x) \psi_j \\
    &\qquad + 2 m_{i j \ell} (d t - d x,d y) f_j' (t - x) \partial_y \psi_\ell+2 m_{i j \ell} (d y,d t - d x) f_\ell' (t - x) \partial_y \psi_j \\
    &\qquad + 2 m_{i j \ell} (d t - d x,d z) f_j' (t - x) \partial_z \psi_\ell+2 m_{i j \ell} (d z,d t - d x) f_\ell' (t - x) \partial_z \psi_j\\
    &\qquad+m_i(d \psi,d \psi).
    \end{aligned}
\end{equation}
We let 
\begin{align*}
    a_{i j \ell} &= 2 m_{i j \ell} (d t - d x,d t + d x)+ 2 m_{i \ell j} (d t + d x,d t - d x),\\
    b_{i j \ell} &= 2 m_{i j \ell} (d t - d x,d y)+2 m_{i \ell j} (d y,d t - d x),\\
    c_{i j \ell} &= 2 m_{i j \ell} (d t - d x,d z)+2 m_{i j \ell} (d z,d t - d x).
\end{align*}
and the equation for $\psi_i$ becomes
\begin{align*}
\Box \psi_i=a_{i j \ell}f_j' (t - x) (\partial_t + \partial_x) \psi_\ell+b_{i j \ell}f_j' (t - x) \partial_y \psi_\ell+ c_{i j \ell}f_\ell' (t - x) \partial_z \psi_j+m_i(d \psi,d \psi)
\end{align*}
Now, let $A:\R\to M^{N\times N}$ be the matrix solution of the ODE
\begin{align*}
A'_{iq} + {1 \over 2} A_{ij}a_{j\ell q}f'_\ell&=0\\
A_{ij}(10)&=Id.
\end{align*}
Since this equation is of the form $A'=AB$, $A$ is invertible everywhere. Also, note that $A$ is constant both on $[1,\infty)$ and on $(-\infty,-1]$. Let
\begin{equation}\label{eq:bigtransformation}
\gamma=A(t-x)\psi.
\end{equation}
Then
\begin{align}
\square \gamma_i &= 2 A'_{iq}(\partial_t+\partial_x)\psi_q+A_{ij}\square\psi_j\nonumber\\
&= 2 A'_{iq}(\partial_t+\partial_x)\psi_q + A_{ij}a_{j\ell q}f'_\ell(t-x)(\partial_t+\partial_x)\psi_q +
A_{i j} b_{j \ell q} f_\ell' (t - x) \partial_y \psi_q +\nonumber\\
&\qquad + A_{i j} c_{j \ell q} f_\ell' (t - x) \partial_z \psi_q +(Am(d \psi,d \psi))_i,\nonumber
\end{align}
where we treat $A$ as a function of $t - x$. Thus, we have that
\begin{align}
\square \gamma_i = A_{i j} b_{j \ell q} f_\ell'(t - x) A^{-1}_{q p} \partial_y \gamma_p &+ A_{i j} c_{j \ell q} f_\ell'(t - x) A^{-1}_{q p} \partial_z \gamma_p  + (Am(d (A\inv\gamma),d (A\inv\gamma)))_i\label{eq:transformed}
\end{align}
Henceforth, we will study this equation. The bounds we get on $\gamma$ will easily convert into bounds on $\psi$. When Condition~\ref{cond:unexciting} is satisfied, we note that $b_{ij\ell}=c_{ij\ell}=0$, so the equation is
\begin{equation}
    \square \gamma_i=(Am(d (A\inv\gamma),d (A\inv\gamma)))_i
\end{equation}
or, written in vector form,
\begin{equation} \label{eq:TransformedEq}
    \square \gamma=Am(d (A\inv\gamma),d (A\inv\gamma)).
\end{equation}
Thus, in this case, the equation reduces to a nonlinear perturbation with variable coefficients of the wave equation on Minkowski space after the transformation~\eqref{eq:bigtransformation} from $\psi$ to $\gamma$.

\section{Geometry of the Interaction} \label{sec:Geometry}
We now analyze the geometry of the interaction between the fixed solution $f$ and the transformed perturbation $\gamma$. The traveling wave $f$ is supported in the set of points where $|t - x| \le 1$. Looking at equation~\eqref{eq:TransformedEq} above, we have a global stability problem where the nonlinearity involves additional terms in $f$ (namely $A$ depends on $f$). We know from existing theory (see, for example, \cite{LukNotes}) that the decay properties of wave equations arising from localized data such as that of $\gamma$ can be well described in terms of outgoing and incoming cones. It is thus natural to examine the interaction between the outgoing cones of $\gamma$, along which radiation propagates, and the family of hyperplanes given by $t - x = c$, as this is where $f$ is supported. The Lemmas in this section study two main geometric aspects of the interaction between the plane wave and the perturbation. The first is the size of the set of interaction, which is small. The second is the relationship between null frames adapted to the plane wave and null frames adapted to the perturbation.

We will define $u=t-r$ and $v=t+r$, as well as $u'=x-r$ and $v'=x+r$. We begin by bounding the intersection in $\Sigma_t$ of $u \ge -1$ and $|u'| \le 1$. We recall that this set is denoted by $S_t$ (see Section \ref{sec:notation}). The volume of $S_t$ then gives the volume of the set where $A$ and $\gamma$ interact. We note that $|r - x| \le 2$ in $S_t$.

\begin{lemma}
For $t\ge 1$, we have that $\mu(S_t) \le 100t$, where $\mu$ is the Lebesgue measure on $\R^3$.
\end{lemma}
\begin{proof}
We have that in $S_t$, $x\ge t-1$ and $r\le t+1$, so
\[
y^2+z^2=r^2-x^2\le (t+1)^2-(t-1)^2=4t
\]
This gives us that $S_t$ is a subset of a cylinder whose axis lies along the $x$ axis of radius $2\sqrt{t}$ and of thickness $2$, so $|S_t|\le 100t$. Also note that we obtain
\begin{equation}\label{ineq:yzbound}
|y|,|z|\le 2\sqrt{t}
\end{equation}
\end{proof}

We shall also require an estimate on ${1 \over r^2} |S(r) \cap S_t|$, where $S(r)$ is the sphere of radius $r$. This corresponds to the value of the integral
\[
\int_{S^2} \chi_{S_t} d \omega
\]
when taking the $r$ coordinate fixed. We have the following lemma.

\begin{lemma} \label{lem:VolumeEst}
For $t\ge 2$, we have that $\sigma(S(r) \cap S_t) \le {C \over t}$, where $\sigma$ is the measure on the unit sphere $S^2$. Moreover, we also have that $\Vert (1 + t)^{{1 \over q}} \chi_{S_t} \Vert_{L_t^\infty L_r^l L_\omega^q} \le C$ for $1 \le l, q \le \infty$.
\end{lemma}
\begin{proof}
The first part follows from the fact that the diameter of the set on the sphere is controlled by ${1 \over \sqrt{t}}$.
Indeed, by \eqref{ineq:yzbound} we have that $|y| \le 2 \sqrt{t}$ and that $|z| \le 2 \sqrt{t}$.
On a sphere of radius comparable to $t$, this means that the angle $\alpha$ made between the line connecting the center and the point with coordinates $(x,y,z)$ and the $x$ axis is comparable to ${1 \over \sqrt{t}}$.
This is because $y^2 + z^2$ will be comparable to $t^2 \sin^2 (\alpha)$.
This means that $\sin(\alpha)$ is comparable to ${1 \over \sqrt{t}}$, meaning that $\alpha$ is comparable to ${1 \over \sqrt{t}}$.
The desired result then follows because the region is contained in a disc on the sphere of radius ${C \over \sqrt{t}}$, giving us a volume of $C \over t$, as desired.

The second part follows directly from the first part after noting that $\chi_{S_t}$ is bounded and has width at most $2$ in the $r$ direction at every $t$.
\end{proof}

The frame $\partial_v$, $\partial_u$, and $e_A$ is well adapted to the geometry of compactly supported perturbations in Minkowski space. Meanwhile, the vector fields $\partial_{v'} = \half(\partial_t + \partial_x)$, $\partial_{u'} = \half(\partial_t - \partial_x)$, $\partial_y$, and $\partial_z$ are well adapted to the geometry of the right moving plane wave $f$. Because $A$ is defined in terms of $f$ and has nonzero derivative only in the region where $f$ is supported, these vector fields are well adapted to the behavior of $A$. Moreover, the commutation vector fields $\Gamma$ are used for showing decay for solutions to the wave equation arising from localized data (see \cite{LukNotes}). We now describe the relationships between these three sets of vector fields in the set $S_t$ where $\gamma$ and $A$ interact.

\begin{lemma}\label{lem:frameconversion}
We have the following decompositions of the commutation vector fields in terms of the frame adapted to the plane wave.
\begin{equation}
    \begin{aligned}
    S &= t \partial_t + r \partial_r = v' \partial_{v'} + u' \partial_{u'} + y \partial_y + z \partial_z,
    \\ \Omega_{x y} &= x \partial_y - y \partial_x = {v' - u' \over 2} \partial_y - y \partial_{v'} + y \partial_{u'},
    \\ \Omega_{x z} &= x \partial_z - z \partial_x = {v' - u' \over 2} \partial_z - z \partial_{v'} + z \partial_{u'},
    \\ \Omega_{y z}& = y \partial_z - z \partial_y,
    \\ \Omega_{t x} &= t \partial_x + x \partial_t = v' \partial_{v'} - u' \partial_{u'},
    \\ \Omega_{t y} &= t \partial_y + y \partial_t = t \partial_y + y \partial_{v'} + y \partial_{u'},
    \\ \Omega_{t z} &= t \partial_z + z \partial_t = t \partial_z + z \partial_{v'} + z \partial_{u'}.
    \end{aligned}
\end{equation}
Moreover, we have the following decomposition of the null frame adapted to the perturbation in terms of the frame adapted to the plane wave.
\begin{equation}
    \begin{aligned}
    \partial_v &= \partial_{v'} + {x - r \over 2 r} \partial_{v'} - {x - r \over 2 r} \partial_{u'} - {y \over 2 r} \partial_y - {z \over 2 r} \partial_z,
    \\ {1 \over r} \Omega_{x y} &= {v' - u' \over 2 r} \partial_y - {y \over r} \partial_{v'} + {y \over r} \partial_{u'},
    \\ {1 \over r} \Omega_{x z} &= {v' - u' \over 2 r} \partial_z - {z \over r} \partial_{v'} + {z \over r} \partial_{u'}.
    \end{aligned}
\end{equation}
\end{lemma}
\begin{proof}
These follow from direct computation.
\end{proof}

Let $Z$ be the family of geometric vector fields consisting of rotations, scaling and Lorentz boosts, and translations. We note the following consequence of this lemma.

\begin{lemma}\label{lem:roottfactors}
Let $F$ be a smooth function of $u' = t - x$ with $F'$ supported in $[-1,1]$. Then, for $k\ge 1$, we have that $|\Gamma_1\cdots\Gamma_k F| \le C (1+t)^{k/2} \sum_{j=1}^k|F^{(j)}|$ in $S_t$ for $\Gamma_1,\ldots,\Gamma_k\in Z$. This gives us that
\begin{equation} \label{eq:CommutationBound}
    \begin{aligned}
    \Gamma_1\cdots\Gamma_kA_{i j} \le C_k (1 + t)^{k/2} \Vert A \Vert_{C^k}.
    \end{aligned}
\end{equation}
in $S_t$.
\end{lemma}
\begin{proof}
We begin by noting that Lemma \ref{lem:frameconversion} implies that $\Gamma^\alpha F$ can be written in the form
\begin{equation} \label{eq:FrameExpansion}
    \begin{aligned}
    \sum_{|\beta| \le |\alpha|} a_\beta (v',u',y,z) \partial^\beta F,
    \end{aligned}
\end{equation}
where $\partial^\beta$ denotes strings of the coordinate vector fields $\partial_{v'}, \partial_{u'}, \partial_y, \partial_z$ determined by the multiindex $\beta$ in the usual way, and where the functions $a_\beta (v',u',y,z)$ are sums of monomials in the variables $v', u', y, z$ of degree at most $|\beta|$. Now, in order to prove the result, we proceed by induction. With $k$ the induction parameter, the inductive hypothesis shall be that, for all $|\alpha| \le k$, every monomial $b_\beta$ that appears in some $a_\beta$ in the expansion \eqref{eq:FrameExpansion} is of size at most $C_{|\beta|} (1 + t)^{{|\beta| \over 2}}$ in $S_t$.
This means that if $b_\beta$ is of the form $C (v')^{\kappa_1} (u')^{\kappa_2} y^{\kappa_3} z^{\kappa_4}$ with $\kappa_1 + \kappa_2 + \kappa_3 + \kappa_4 \le |\beta|$, we have that $\kappa_1 + {1 \over 2} (\kappa_3 + \kappa_4) \le {|\beta| \over 2}$. We note establishing this for the monomials will imply \eqref{eq:CommutationBound} by the triangle inequality.

The base case $k = 1$ follows immediately from Lemma \ref{lem:frameconversion}. We shall now show that this is true for $|\alpha| = k + 1$.

We note that $\Gamma^\alpha F = \Gamma \Gamma^\omega F$ for some $|\omega| = k$. We can expand $\Gamma^\omega$ to get
\[
\Gamma^\alpha F=\sum_{|\beta| \le |\omega|} \Gamma (a_\beta (v',u',y,z) \partial^\beta F).
\]
If $\Gamma$ falls on $\partial^\beta F$ in this expansion, the result follows by the inductive hypothesis along with Lemma \ref{lem:frameconversion}. Now, if $\Gamma$ falls on $a_\beta$, the result follows from the fact that $|\Gamma b| \le C (1 + \sqrt{t}) |b|$ whenever $b$ is a monomial in the variables $v', u', y, z$. Indeed, the only way for $\Gamma$ to increase the size of $b$ in $S_t$ is for it to contain the derivatives $\partial_y$, $\partial_z$, or $\partial_{u'}$ with a weight in front. The only such terms that can be present in $\Gamma$ are of the form $u' \partial_{u'}$, $y \partial_{u'}$, $z \partial_{u'}$, $v' \partial_y$, and $v' \partial_z$. Every such term adds a weight comparable to $\sqrt{1+t}$ in $S_t$, proving the claim for all monomials $b$. This completes the proof.
\end{proof}

\section{Linear Estimates}
We shall use two estimates for the linear wave equation in conjunction with the pointwise estimates given by commuting. The first of these estimates is just the basic energy estimate applied in a spacetime slab bounded by two $\Sigma_t$ hypersurfaces. The second is an energy estimate applied to truncated outgoing cones. The energy on the time slices $\Sigma_t$ is the standard one which controls all derivatives in $L^2$. The characteristic energy on outgoing cones controls only the good derivatives.

We can now state the energy estimates precisely.

\begin{proposition} \label{prop:EnEst}
Let $h : \R^{n + 1} \rightarrow \R$ be a smooth function decaying sufficiently rapidly at infinity. Then, we have that
\[
\int_{\Sigma_s} \sum_\alpha (\partial_\alpha h)^2 d x = \int_{\Sigma_0} \sum_\alpha (\partial_\alpha h)^2 d x + \int_0^s \int_{\Sigma_t} (\Box h) (\partial_t h) d x d t.
\]
Moreover, we have that
\[
\int_{\Sigma_s - B_{s - u}} (\partial h)^2 d x + \int_{\mathcal C_{u} (|u|,2 s - u)} (\overline{\partial} h)^2 d vol(\mathcal C_u) = \int_{\Sigma_0} (\partial h)^2 d x + \int_0^s \int_{\Sigma_t - B_{t - u}} (\Box h) (\partial_t h) d x d t.
\]
\end{proposition}
\begin{proof}
These follow from just multiplying $\Box h$ by $\partial_t h$ and integrating by parts in the appropriate spacetime region.
\end{proof}

The second estimate we shall use was used by Lindblad and Rodnianski in to control nonlinear errors when proving the global stability of Minkowski space for the Einstein equations in wave coordinates (see \cite{LindRod03}, \cite{LindRod05}, and \cite{LindRod10}).
\begin{proposition}
Let $h : \R^{n + 1} \rightarrow \R$ be a smooth function decaying sufficiently rapidly at infinity. Then, we have that
\[
\int_0^s \int_{\Sigma_t} {1 \over (1 + |u|)^{1 + \delta}} (\overline{\partial} h)^2 d x d t + \int_{\Sigma_t} \sum_\alpha (\partial_\alpha h)^2 d x \le C_\delta \int_{\Sigma_0} \sum_\alpha (\partial_\alpha h)^2 d x + C_\delta \left |\int_0^s \int_{\Sigma_t} (\Box h) (\partial_t h) d x d t \right |.
\]
\end{proposition}
\begin{proof}
This is simply an averaged characteristic energy estimate. It follows by using the energy estimate on truncated cones from Proposition~\ref{prop:EnEst}, multiplying by the function $(1 + |u|)^{-1 - \delta}$, and integrating in $u$.
\end{proof}

\section{Klainerman-Sobolev Inequalities} \label{sec:KlaiSobInequalities}
We get decay by commuting with weighted vector fields and using various slightly modified versions of the Klainerman-Sobolev Inequality. Because we do not want to commute too many times with weighted vector fields, we shall use $L^p$ versions of the inequalities. We now turn to providing proofs of these inequalities.

We recall the following results. These rescaled Sobolev inequalities are the blueprint for proving the Klainerman-Sobolev inequalities that we need. The rescalings must, in general, be anisotropic because we rescale by different amounts in the $u$ and angular directions when near the light cone. Thus, we let $L$ be any vector in $\R^n$ with $L^i > 0$. For any open set $\Omega \subset \R^n$, we shall denote by $L \Omega$ the set of all points $x^i \in \R^n$ with $x^i = L^i y^i$ for some $y^i \in \Omega$. This is an anisotropic rescaling of $\Omega$. We note that the case of $L^i = L^j$ for all $i$ and $j$ corresponds to a standard rescaling of $\Omega$. We shall also denote by $\partial_L^\alpha$ a string of rescaled translation vector fields. Each rescaled translation vector field is of the form $L^i \partial_{x^i}$.
\begin{lemma} \label{lem:SobolevEst}
Let $h : \R^n \rightarrow \R$ be a smooth, compactly supported function. Moreover, let $\Omega$ be an open subset of $\R^n$ with a Lipschitz boundary, and let $L$ and $\partial_L^\alpha$ be as defined above. Then, for $k < {n \over 2}$, we have that

\[
\Vert h \Vert_{L^{1 / ((1 / 2) - k / n)} (\Omega)} \le C_\Omega\Vert h \Vert_{H^k(\Omega)},
\]

%and that

%\[
%L^k \Vert h \Vert_{L^{1 / ((1 / 2) - k / n)}}(\Omega) \le C \sum_{|\alpha| \le k} \Vert \partial_L^\alpha h \Vert_{L^2 (\Omega)}.
%\]

and that

\[
\prod_{i = 1}^n (L^i)^{k / n} \Vert h \Vert_{L^{1 / ((1 / 2) - k / n)} (L \Omega)} \le C_\Omega \sum_{|\alpha| \le k} \Vert \partial_L^\alpha h \Vert_{L^2 (L\Omega)}.
\]

Similarly, for $k > {n \over 2}$, we have that
\[
\Vert h \Vert_{L^\infty (\Omega)} \le C_\Omega \Vert h \Vert_{H^k (\Omega)},
\]
and that
\[
\prod_{i = 1}^n (L^i)^{1 \over 2} \Vert h \Vert_{L^\infty (L\Omega)} \le C_\Omega \sum_{|\alpha| \le k} \Vert \partial_L^\alpha h \Vert_{L^2 (L\Omega)}.
\]

Similar statements hold if the domain is an $n$-dimensional compact manifold with boundary rather that a subset of $\R^n$.
\end{lemma}
\begin{proof}
The unscaled results follow from usual Sobolev embedding theorems (see \cite{AdamsFournier03}). The rescaled results follow from applying the unscaled Sobolev embedding theorem in rescaled coordinates $\overline{x}^i = {1 \over L^i} x^i$ (see also \cite{Hormander97}). The rescalings are described in more detail in the proof of Proposition \ref{prop:KlaiSob} below, where they are carried through in a related setting.
\end{proof}

We shall also need an anisotropic version of these estimates in which we rescale by different amounts in different directions. Let $L$ now be a vector in $\R^n$. For any open set $\Omega$ in $\R^n$, we shall denote by $L \Omega$ the set of all points $x^i \in \R^n$ where $x^i = L^i y^i$ for some point $y^i \in \Omega$.

We can now state the Klainerman-Sobolev Inequalities we will use. The ones embedding into $L^\infty$ are standard. We have not seen the $L^p$ norms stated anywhere, so we shall prove them here.

\begin{proposition}\label{prop:KlaiSob}
Let $h : \R^{n + 1} \rightarrow \R$ be a smooth, compactly supported function. We have the following estimates when $t \ge 0$:
\begin{enumerate}
    \item For $k < n/2$, we have that
    \[
    \|(1+t + r)^{k (n - 1)/n} (1+|u|)^{k/n} h\|_{L^{1/((1 / 2) - k/n)} (\Sigma_t)} \le C \sum_{|\alpha| \le k} ||\Gamma^\alpha h||_{L^2 (\Sigma_t)}.
    \]
    \item For $k > n/2$, we have that
    \[
    \|(1 + t + r)^{{n - 1} \over 2}(1+|u|)^{{1 \over 2}}  h\|_{L^\infty (\Sigma_t)} \le C \sum_{|\alpha| \le k} ||\Gamma^\alpha h||_{L^2 (\Sigma_t)}.
    \]
\end{enumerate}
\end{proposition}
We note that there are generalizations to the case of $k = {n \over 2}$ as well (such as a statement analogous to the case of $k = {n - 1 \over 2}$ in Proposition~\ref{prop:LowRegKS}), but since they are not necessary for this paper, we shall omit them.
\begin{proof}
We shall only consider the case of $k < n / 2$. The case of $k > n / 2$ is essentially the original Klainerman-Sobolev inequality, and can be found in \cite{Kl85} (see also \cite{Sogge08} and \cite{LukNotes}). The proof of the $k < n / 2$ case is an adaptation of the proof of the usual Klainerman-Sobolev Inequality that is found in \cite{Sogge08}.

The proof of this proposition will be broken up into three regions. The first region is the interior of the light cone defined by $t \ge {21 \over 20} r$, the second region is along the light cone defined by ${9 \over 10} r \le t \le {10 \over 9} r$, and the third region is outside the light cone $r \ge {21 \over 20} t$. Proving the desired estimate in each of these regions will establish the result. Moreover, we note that we can assume that $t + r \ge 1$, as the estimate in the region $0 \le t + r \le 1$ follows easily from the usual Sobolev embedding by using translation vector fields.

We introduce a modified system of polar coordinates $(u,\omega)$ on $\Sigma_t$ where $u = t - r$ and the $\omega$ are coordinates on the sphere. In the following, we shall denote by $\partial^\alpha_R$ strings of the vector field $u \partial_r$ and the rotation vector fields $\Omega$. We note that $u\partial_r$ is a linear combination of the scaling and Lorentz boost vector fields with coefficients that depend smoothly on $x/t$ (see Lemma~\ref{lem:OutgoingConeVFs}). From this, we obtain that in the region defined by ${9 \over 10} r \le t \le {10 \over 9} r$, we have
\begin{equation} \label{eq:WeightedVFs}
    \begin{aligned}
    |\partial^\alpha_R h| \le C \sum_{|\beta| \le |\alpha|} |\Gamma^\beta h|.
    \end{aligned}
\end{equation}

We begin with the region where $t \ge {21 \over 20} r$. We recall (see, for example, \cite{Sogge08} or \cite{LukNotes}) that
\[
|\partial^\alpha h| \le {C \over |t - r|^{|\alpha|}} \sum_{|\beta| \le \alpha} |\Gamma^\beta h|.
\]
Away from the light cone, the weights are comparable to $(t + r)^{|\alpha|}$. Let $\chi : \R \rightarrow \R$ be a smooth cutoff function equal to $1$ in the ball of radius ${20 \over 21}$ and equal to $0$ outside the ball of radius ${21 \over 22}$. Then, we consider the function $\chi \left ({r \over t} \right ) h$. We note that this is supported where ${r \over t} \le {21 \over 22}$, and moreover, this is equal to $1$ where ${r \over t} \le {20 \over 21}$. Now, we take coordinates $\overline{x}^i$ on $\Sigma_t$ that come from rescaling by ${1 \over t}$. These coordinates are given by $\overline{x}^i = {1 \over t} x^i$ where the $x^i$ are the standard Cartesian coordinates on $\Sigma_t$, and we denote by $\partial_{\overline{x}}^\alpha$ strings of translation vector fields with respect to the $\overline{x}^i$ coordinate system. We now use the Sobolev embedding theorem on the function $\chi \left ({r \over t} \right ) h$ on $\Sigma_t$ in the $\overline{x}^i$ coordinate system with the volume form $d \overline{x}^1 \dots d \overline{x}^n$. This gives us
\[
\Vert \chi h \Vert_{L^{1 / ((1 / 2) - k / n)} (\Sigma_t,d \overline{x}^1 \dots d \overline{x}^n)} \le C \sum_{|\alpha| \le k} \Vert \partial_{\overline{x}}^\alpha (\chi h) \Vert_{L^2 (\Sigma_t,d \overline{x}^1 \dots d \overline{x}^n)}.
\]
In the support of $\chi$, we note that $t$ and $t + r$ are comparable, and we note that
\[
|\partial_{\overline{x}}^\alpha (\chi h)| \le C \sum_{|\beta| \le |\alpha|} |\Gamma^\alpha (\chi h)|.
\]
Moreover, we note that $d x^1 \dots d x^n = t^n d \overline{x}^1 \dots d \overline{x}^n$. Thus, rescaling these norms to make them in terms of the volume form $d x^1 \dots d x^n$ gives us that
\begin{equation} \label{InteriorKSEst1}
    \begin{aligned}
    ||(1 + t + r)^k h||_{L^{1/((1 / 2) - k/n)} (I_t)} \le C \sum_{|\alpha| \le k} ||\Gamma^\alpha h||_{L^2 (\Sigma_t)},
    \end{aligned}
\end{equation}
where $I_t \subset \Sigma_t$ is the region where $t \ge {21 \over 20} r$. A similar argument can be used in the third region where $r \ge {21 \over 20} t$, which we denote by $E_t$. This gives us the estimate
\begin{equation} \label{AwayFromLCKSEst1}
    \begin{aligned}
    ||(1 + t + r)^k h||_{L^{1 / ((1 / 2) - k / n)} (E_t \cup I_t)} \le C \sum_{|\alpha| \le k} ||\Gamma^\alpha h||_{L^2 (\Sigma_t)}.
    \end{aligned}
\end{equation}
Thus, we must now only consider the second region along the light cone.

We proceed in a slightly different fashion, rescaling in a different fashion in the angular and radial directions. We should always rescale in the angular directions by something comparable to $r$. Meanwhile, if we take points where $|u| \approx |u_0|$, we see that we should rescale in the radial direction by something comparable to $|u_0|$. Thus, we shall localize along a dyadic sequence of points $p_l$ in $u$ using cutoff functions $\chi_l$ (see below for the precise definition) adapted to scale $p_l$. We shall then prove an estimate of the form
\begin{equation} \label{eq:LocalizedKSEst}
    \begin{aligned}
    \Vert (1 + |u|)^{k / n} (1 + t + r)^{k (n - 1) / n} \chi_l h \Vert_{L^{1 / ((1 / 2) - k / n)} (\Sigma_t)} \le C \sum_{|\alpha| \le k} \Vert \partial_R^\alpha (\chi_l h) \Vert_{L^2 (\Sigma_t)}.
    \end{aligned}
\end{equation}
By \eqref{eq:WeightedVFs}, this will imply that
\[
\Vert (1 + |u|)^{k / n} (1 + t + r)^{k (n - 1) / n} \chi_l h \Vert_{L^{1 / ((1 / 2) - k / n)} (\Sigma_t)} \le C \sum_{|\alpha| \le k} \Vert \Gamma^\alpha (\chi_l h) \Vert_{L^2 (\Sigma_t)}.
\]
We denote by $c_l$ the sequence of numbers on the left hand side of this inequality, and by $d_l$ the sequence of numbers on the right hand side of this inequality. The cutoff functions $\chi_l$ will be chosen such that
\begin{equation} \label{eq:SumKSEst}
    \begin{aligned}
    &\Vert (1 + t + r)^{k (n - 1) / n} (1 + |u|)^{k / n} h \Vert_{L^{1 / ((1 / 2) - k / n)} (\Sigma_t)}
    \\ &\qquad\qquad\le C \Vert c_l \Vert_{\ell^{1 / ((1 / 2) - k / n)} (\Z)} + \Vert (1 + t + r)^{k (n - 1) / n} (1 + |u|)^{k / n} h \Vert_{L^{1 / ((1 / 2) - k / n)} (I_t \cup E_t)},
    \end{aligned}
\end{equation}
where the $\ell^{1 / ((1 / 2) - k / n)}$ norm of $c_l$ is with respect to the index $l$.

We have already shown the desired estimate for this final term. Using the estimate \eqref{eq:LocalizedKSEst} along with the fact that $\Z$ with the counting measure is a discrete measure space and that $2 \le 1 / ((1 / 2) - k / n)$, we will have that
\[
||c_l||_{\ell^{1/((1/2) - k/n)}} \le ||d_l||_{\ell^{1/((1/2) - k/n)}} \le ||d_l||_{\ell^2} \le C \sum_{|\alpha| \le k} ||\Gamma^\alpha h||_{L^2 (\Sigma_t)},
\]
giving us that
\[
\Vert (1 + t + r)^{k (n - 1) / n} (1 + |u|)^{k / n} h \Vert_{L^{1 / ((1 / 2) - k / n)} (\Sigma_t)} \le C \sum_{|\alpha| \le k} \Vert \Gamma^\alpha h \Vert_{L^2 (\Sigma_t)},
\]
as desired.

It thus suffices to make \eqref{eq:LocalizedKSEst} precise. We take a sequence of points $p_l = 2^l$ with $1 \le l \le \big\lceil \log \left ({10 \over 9} t \right ) \big \rceil$. In order to localize in $u$ at scale $p_l$ around the points where $u \approx p_l$, we also take a smooth cutoff function $\chi : \R \rightarrow \R$ supported in $\left [-{7 \over 4},{7 \over 4} \right ]$ equal to $1$ in the interval $\left [-{3 \over 4},{3 \over 4} \right ]$. We then define the rescaled cutoff functions
\[
\chi_l (u) = \chi \left ({u - p_l} \over p_{l - 1} \right ) + \chi \left ({u + p_l} \over p_{l - 1} \right )
\]
for $2 \le l \le \big\lceil \log \left ({10 \over 9} t \right ) \big \rceil$. We are taking a sum of two functions because we have to consider points with both positive and negative $u$ coordinates. We also take the cutoff function $\chi_1 (u) = \chi \left ({u \over 10} \right )$. The key properties of these functions are that they are adapted to scale $p_l$, and that $|u|$ is comparable to $p_l$ in their support. With this choice of $\chi_l$, we have \eqref{eq:SumKSEst}.

We shall first consider $\chi_1 (u) h$, which only requires angular rescaling. We stay within one time slice and change variables into $(r,\theta)$. We shall use a Sobolev embedding theorem in $(r,\theta)$ coordinates with $d r d \theta$ as the volume form. This involves decreasing the volume by a factor of $t^{n-1}$, as the unit sphere has volume comparable to $1$. We then use Sobolev embedding to get, for $k<n/2$
\[
||\chi_1 h||_{L^{1/(1/2-k/n)}(\Sigma_t, d\theta dr)}\lesssim \sum_{|\alpha| \le k} ||\partial_\theta^\alpha (\chi_1 h)||_{L^2(\Sigma_t, d\theta dr)}
\]
and, rescaling, we get
\begin{equation} \label{AlongLCKSEst1}
    \begin{aligned}
    ||(1 + t + r)^{{k (n - 1) \over n}} \chi_1 h||_{L^{1/(1/2-k/n)} (\Sigma_t)} \lesssim \sum_{j=1}^k ||\Gamma^k (\chi_1 h)||_{L^2 (\Sigma_t)},
    \end{aligned}
\end{equation}
as desired.

We shall now consider the remaining $\chi_l h$. Going in to the $(u,\omega)$ coordinates on $\Sigma_t$, we then want to apply one of the estimates in Lemma~\ref{lem:SobolevEst} to the function $\chi_k (u) h(u,\omega)$. Moreover, we note that $u \partial_r \chi_l (u) \le C$ where $C$ is a universal constant depending only on $\chi$. Indeed, we have that
\[
u \partial_r \chi_l (u) = u \chi_l' (u) = {u \over p_{l - 1}} \chi' \left ({u - p_l \over p_{l - 1}} \right ).
\]
We note that ${u \over p_{l - 1}} \le 4$ in the support of $\chi' \left ({u - p_k \over p_{l - 1}} \right )$, giving us the desired result.

We now consider the rescaled coordinate $\overline{u}_l = {1 \over p_l} u$, giving us the coordinate system $(\overline{u}_l,\omega)$. We shall use $\partial_l^\alpha$ to denote strings of the coordinate vector fields $\partial_{\overline{u}_l}, \partial_\omega$. Using the Sobolev embedding theorem in these coordinates with respect to the volume form $d \overline{u}_l d \omega$, we have that
\[
\Vert \chi_l h \Vert_{L^{1 / ((1 / 2) - k / n)} (\Sigma_t, d \overline{u}_l d \omega)} \le C \sum_{|\alpha| \le k} \Vert \partial_l^\alpha (\chi_l h) \Vert_{L^2 (\Sigma_t,d \overline{u} d \omega)}.
\]
Now, we have that $|\partial_{\overline{u}_l} (\chi_l h)| \le C|u \partial_r (\chi_l h)|$. Rescaling by $r$ in the angular directions and $p_l$ in the $u$ direction then gives us that
\[
\Vert (1 + |u|)^{k / n} (1 + t + r)^{k (n - 1) / n} \chi_l h \Vert_{L^{(1 / ((1 / 2) - k / n)} (\Sigma_t)} \le C \sum_{|\alpha| \le k} \Vert \partial_R^\alpha (\chi_l h) \Vert_{L^2 (\Sigma_t)},
\]
as desired.
\end{proof}

We shall also need the following inequalities, which are a simple consequence of Sobolev inequalities on the sphere. These are used when we can only commute with a single weighted derivative, and we want to take advantage of the small measure of $\chi_{S_t}$ on the sphere $S^2$ (see Lemma~\ref{lem:VolumeEst}). Thus, the case that will be useful in this paper is $k = (n-1) / 2$ in the following proposition, although we state the general estimates for completeness.

\begin{proposition} \label{prop:LowRegKS}
Let $h : \R^n \rightarrow \R$ be a smooth, compactly supported function. We take polar coordinates $(r,\omega)$ on $\R^n$ where $\omega \in S^{n - 1}$. We have that
\begin{enumerate}
    \item If $k > {n - 1 \over 2}$, we have that
    \[
    \Vert r^{{n - 1 \over 2}} h \Vert_{L^\infty (S^{n - 1})} \le C \Vert r^{{n - 1 \over 2}} h \Vert_{H^k (S^{n - 1})}.
    \]
    As a result of this, we have that
    \[
    \Vert r^{{n - 1} \over 2} h \Vert_{L_r^2 L_\omega^\infty} \le C \Vert r^{{n - 1} \over 2} h \Vert_{L_r^2 H_\omega^k}.
    \]
    \item If $k < {n - 1 \over 2}$, we have that
    \[
    \Vert r^{{n - 1 \over 2}} h \Vert_{L^p (S^{n - 1})} \le C \Vert r^{{n - 1 \over 2}} h \Vert_{H^k (S^{n - 1})},
    \]
    where $p = 1 / ((1 / 2) - k / n)$. As a result of this, we have that
    \[
    \Vert r^{{n - 1} \over 2} h \Vert_{L_r^2 L_\omega^p} \le C \Vert r^{{n - 1} \over 2} h \Vert_{L_r^2 H_\omega^k}.
    \]
    \item If $k = {n - 1 \over 2}$, we have that
    \[
    \Vert r^{{n - 1 \over 2}} h \Vert_{L^p (S^{n - 1})} \le C_p \Vert r^{{n - 1 \over 2}} h \Vert_{H^k (S^{n - 1})},
    \]
    where $2 \le p < \infty$, and where the constant $C_p$ blows up as $p \rightarrow \infty$. As a result of this, we have that
    \[
    \Vert r^{{n - 1} \over 2} h \Vert_{L_r^2 L_\omega^p} \le C_p \Vert r^{{n - 1} \over 2} h \Vert_{L_r^2 H_\omega^k}.
    \]
\end{enumerate}
\end{proposition}
\begin{proof}
These all follow from using the appropriate Sobolev inequality on the sphere on the function $r^{{n - 1 \over 2}} h$, and from then simply taking an $L^2$ norm in $r$.
\end{proof}

We shall also sometimes need some mixed space inequalities which involve the above rescaled Sobolev estimates on the spheres along with a Hardy inequality in the $r$ direction. These are used when controlling terms in the nonlinearity in Section~\ref{sec:ClosingEnergy} that appear without any derivatives.

\begin{proposition} \label{prop:Hardy}
Let $h : \R^n \rightarrow \R$ be a smooth, compactly supported function. We take polar coordinates $(r,\omega)$ on $\R^n$. Then, we have that
\[
\Vert h(r,\cdot) \Vert_{L_r^\infty L_\omega^p} \le \Vert \partial_r h \Vert_{L_r^1 L_\omega^p}.
\]
In the case where $h$ is supported in the region where $r \le b$, we also have that
\[
\Vert h \Vert_{L_r^2 [a,b] L_\omega^p} \le (b - a) \Vert \partial_r h \Vert_{L_r^2 [a,b] L_\omega^p}.
\]
\end{proposition}
\begin{proof}
Let $p'$ be the Holder dual to $p$ (i.e., ${1 \over p} + {1 \over p'} = 1$). Let $g$ be an arbitrary smooth function on $S^{n - 1}$ with $\Vert g \Vert_{L_\omega^{p'}} \le 1$. Then, we have that
\begin{equation}\label{ineq:Hardypropproof}
    \begin{aligned}
    \int_{S^{n - 1}} h(r,\omega) g(\omega) d \omega = -\int_r^\infty \int_{S^{n - 1}} \partial_r h(s,\omega) g(\omega) d \omega d s \le \int_r^\infty \Vert \partial_r h(s,\cdot) \Vert_{L_\omega^p} d s \le \Vert \partial_r h \Vert_{L_r^1[r,\infty] L_\omega^p}.
    \end{aligned}
\end{equation}
Taking the supremum over all such $g$ gives us the first inequality (specifically, a version of it where the $L_r^\infty$ and $L_r^1$ norms are over $[r_0,\infty]$ for arbitrary $r_0\ge 0$). The second inequality follows from applying Cauchy-Schwarz to this version of the first inequality with $r_0=a$.
\end{proof}
We shall also need Klainerman-Sobolev Inequalities on outgoing null cones. This allows us to show that good derivatives decay like $\frac{1}{(1+t+r)^{3/2}}$. Indeed, in practice, we take $h=\overline\partial\gamma$ in the following proposition. See the truncated cone notation defined in \eqref{notation:truncatedcone}.
\begin{proposition} \label{prop:ConeSob}
Let $h : \R^{3 + 1} \rightarrow \R$ be a smooth function decaying sufficiently rapidly at infinity. We recall that $v = t + r$. Then, in the region where $|u_0| = |t_0 - r_0| \le {t_0 \over 2}$ and where $r_0 \ge 10$ and $t_0 \ge 10$, we have that
\[
|h| (t_0,r_0,\omega_0) \le {C \over (t_0 + r_0)^{{3 \over 2}}} \sum_{|\alpha| \le 2} \Vert \overline{\partial}_R^\alpha h \Vert_{L^2 \left (\mathcal C_{u_0} \left ({4 v_0 \over 5},v_0 \right ) \right )},
\]
where $\overline{\partial}_R$ denotes either $v \partial_v$ or $\Omega_i$ for any rotation vector field $\Omega_i$.
\end{proposition}
\begin{proof}
Let $\chi$ be a smooth function equal to $1$ for $x \ge {9 \over 10}$ and equal to $0$ for $x \le {4 \over 5}$. Then, we consider the function $\chi \left ({v \over v_0} \right )$. This function is equal to $1$ for $v \ge {9 \over 10} v_0$ and it is equal to $0$ for $v \le {4 \over 5} v_0$. Now, we note that $v$, $v_0$, and $r$ are all comparable on the truncated cone $\mathcal C_{u_0} \left ({4 v_0 \over 5},v_0 \right )$ when $|t_0 - r_0| \le {t_0 \over 2}$. Moreover, we note that $r \ge {t \over 100}$ on the truncated cone $\mathcal C_{u_0} \left ({4 v_0 \over 5},v_0 \right )$. We now simply apply a rescaled Sobolev inequality (see Lemma~\ref{lem:SobolevEst}) on the truncated cone $\mathcal C_{u_0} \left ({4 v_0 \over 5},v_0 \right )$ to the function $\chi h$.

There are multiple ways to do this rescaling correctly. One way is to cover the truncated cone by finitely many $(v,\omega)$ coordinate charges (these are polar coordinates on the cone). We localize to each coordinate chart using appropriate cutoffs. Then, we use a rescaled Sobolev embedding theorem on the quantity $r_0 h$ where we rescale by $v_0$ in the $v$ direction (which gives us the $v_0 \partial_v$ vector field). The angular derivatives $\partial_\omega$ are controlled by linear combinations of the rotation vector fields, and the $r_0$ in $r_0 h$ becomes $r_0^2$. The desired result then follows from noting once again that $v$, $r$, $v_0$, and $r_0$ are all comparable in this region (for example, $r_0^2$ is comparable to the real volume form $r^2$).
\end{proof}

We now translate this proposition in terms of the commutation vector fields and the characteristic energy.

We begin with a result on the coefficients we encounter for the vector fields $u \partial_r$ and $v \partial_v$. The first part comes from \cite{Sogge08}.

\begin{lemma} \label{lem:OutgoingConeVFs}
We have that
\[
u \partial_r = a_0 (t,x) S + \sum_{i = 1}^3 a_i (t,x) \Omega_{0 i}
\]
with $a_0$ and $a_i$ smooth functions away from $r = 0$, homogeneous of degree $0$, and with $|\partial^\alpha a_0| \le C_{\alpha,\delta'} (t + r)^{-|\alpha|}$ and $|\partial^\alpha a_i| \le C_{\alpha,\delta'} (t + r)^{-|\alpha|}$ for $r \ge \delta' t$ for any $\delta' > 0$. Similarly, we have that
\[
v \partial_v = \half\left(S + \sum_{i = 1}^3 b_i (x) \Omega_{0,i}\right)
\]
with $b_0$ and $b_i$ smooth functions away from $r = 0$, homogeneous of degree $0$, with $|\partial^\alpha b_i| \le C_{\alpha,\delta'} (1 + r)^{-|\alpha|}$ for $2t\ge r \ge \delta'$ for any $\delta' > 0$.
\end{lemma}

\begin{proof}
The first part can be found in Chapter $2$ of \cite{Sogge08}. For the second part, we simply note that
\[
v \partial_v = \half\left(S + \sum_{i = 1}^3 {x^i \over r} (t \partial_i + x^i \partial_t)\right),
\]
giving us the desired result.
\end{proof}

We now use this result to find the commutators of these vector fields with various commutation fields.

\begin{lemma} \label{lem:ConeCommutation}
Let $\overline{\partial}_R$ denote either $v \partial_v$ or a rotation vector field. Let $h$ be any smooth function. In the region where $2t\ge r \ge \delta' t$ for $\delta' > 0$ and where $t \ge 2$, we have that
\[
\overline{\partial}_R \overline{\partial} h = \sum_{|\alpha| \le 1} A_{\alpha,1} (t,x) \overline{\partial} \Gamma^\alpha h + \sum_{|\alpha| \le 1} A_{\alpha,2} \Gamma^\alpha h,
\]
where the $A_{\alpha,1}$ and $A_{\alpha,2}$ are smooth functions with $|\partial^\beta A_{\alpha,1}| \le {C_{\beta,\delta'} \over (1 + t + r)^{|\beta|}}$ and $|\partial^\beta A_{\alpha,2}| \le {C_{\beta,\delta'} \over (1 + t + r)^{|\beta| + 1}}$.
\end{lemma}
\begin{proof}
These results follow directly from computing with the expressions in Lemma~\ref{lem:OutgoingConeVFs} for $\overline{\partial}_R$ in terms of the commutation fields, along with the expressions for $\overline{\partial}$ from Section~\ref{sec:notation} and checking the commutators of $\overline{\partial}_R$ and $\overline{\partial}$ for the various cases.
\end{proof}

In the following proposition, the assumptions on the support are such that the result will be applicable in our case.
\begin{proposition} \label{prop:ConeKS}
Let $h : \R^{3 + 1} \rightarrow \R$ be a smooth function decaying sufficiently rapidly at infinity and supported in the spacetime ball of radius $10$ centered at the origin along with the region where $t \ge {r \over 2}$. Then, in the region where $|u_0| = |t_0 - r_0| \le {t_0 \over 10}$ and $t_0\ge 2$, we have that
\[
|\overline{\partial} h| (t_0,r_0,\omega_0) \le {C \over (t_0 + r_0)^{{3 \over 2}}} \sum_{|\alpha| \le 2} \left [ \Vert \overline{\partial} \Gamma^\alpha h \Vert_{L^2 \left (\mathcal C_{u_0} \left ({4 t_0 \over 5},v_0 \right ) \right )} + \Vert \partial \Gamma^\alpha h \Vert_{L^2 (\Sigma_{t_0})} \right ].
\]
\end{proposition}

\begin{proof}
We must only consider the region where $r_0 \ge 10$ and $t_0 \ge 10$. Indeed, in the region near the spacetime origin, the result follows from the usual Sobolev inequality because $t_0$ and $r_0$ are then comparable to $1$. Now, using Proposition~\ref{prop:ConeSob}, we have that

\[
\left |\overline{\partial} h \right | (t_0,r_0,\omega_0) \le {C \over (t_0 + r_0)^{{3 \over 2}}} \sum_{|\alpha| \le 2} \Vert \overline{\partial}_R^\alpha \overline{\partial} h \Vert_{L^2 \left (\mathcal C_{u_0} \left ({4 v_0 \over 5},v_0 \right ) \right )}.
\]

We must now commute $\overline{\partial}_R^\alpha$ and $\overline{\partial}$ and write $\overline{\partial}_R^\alpha$ in terms of the commutation fields.

Using Lemma~\ref{lem:ConeCommutation} twice, we have that
\begin{equation}\label{Prop17someintermediatestep}
    \begin{aligned}
    \sum_{|\alpha| \le 2} &\Vert \overline{\partial}_R^\alpha \overline{\partial} h \Vert_{L^2 \left (\mathcal C_{u_0} \left ({4 v_0 \over 5},v_0 \right ) \right )}\\ &\le C \sum_{|\alpha| \le 2} \left [ \Vert \overline{\partial} \Gamma^\alpha h \Vert_{L^2 \left (\mathcal C_{u_0} \left ({4 t_0 \over 5},v_0 \right ) \right )} + \Vert (1 + v)^{-1} \Gamma^\alpha h \Vert_{L^2 \left (\mathcal C_{u_0} \left ({4 v_0 \over 5},v_0 \right ) \right )} \right ].
    \end{aligned}
\end{equation}

We must now control the term
\begin{equation} \label{eq:HardyTerm}
    \begin{aligned}
    \sum_{|\alpha| \le 2} \Vert (1 + v)^{-1} \Gamma^\alpha h \Vert_{L^2 \left (\mathcal C_{u_0} \left ({4 v_0 \over 5},v_0 \right ) \right )}
    \end{aligned}
\end{equation}
in terms of the energy. We first note that $v$ and $r$ are comparable on the truncated cone $C_{u_0} \left ({4 v_0 \over 5},v_0 \right )$, meaning that we can control this by
\[
C \sum_{|\alpha| \le 2} \Vert r^{-1} \Gamma^\alpha h \Vert_{L^2 \left (\mathcal C_{{t_0 - r_0}} \left ({4 v_0 \over 5},v_0 \right ) \right )}\le \frac{C}{r}\left (\int_{\mathcal C_{u_0} \left ({4 v_0 \over 5},v_0 \right )} (\Gamma^\alpha h)^2 d \omega d v \right)^{1 \over 2}.
\]
We now use a Hardy inequality. Integrating by parts with $v - {4 v_0 \over 5}$, we have that
\begin{equation}
    \begin{aligned}
    \int_{\mathcal C_{u_0} \left ({4 v_0 \over 5},v_0 \right )} (\Gamma^\alpha h)^2 d \omega d v = {v_0 \over 5} \int_{S^2} (\Gamma^\alpha h)^2 (u_0,v_0,\omega) d \omega - 2 \int_{\mathcal C_{u_0} \left ({4 v_0 \over 5},v_0 \right )} \left (v - {4 v_0 \over 5} \right ) (\Gamma^\alpha h) \partial_v (\Gamma^\alpha h) d \omega d v.
    \end{aligned}
\end{equation}
Thus, we have that
\begin{equation}
    \begin{aligned}
    \int_{\mathcal C_{u_0} \left ({4 v_0 \over 5},v_0 \right )} (\Gamma^\alpha h)^2 d \omega d v \le {v_0 \over 5} \int_{S^2} (\Gamma^\alpha h)^2 (u_0,v_0,\omega) d \omega + 2 \int_{\mathcal C_{u_0} \left ({4 v_0 \over 5},v_0 \right )} \left (v - {4 v_0 \over 5} \right ) |\Gamma^\alpha h| |\partial_v (\Gamma^\alpha h)| d \omega d v.
    \end{aligned}
\end{equation}
We note that
\begin{equation}
    \begin{aligned}
    \int_{\mathcal C_{u_0} \left ({4 v_0 \over 5},v_0 \right )} \left (v - {4 v_0 \over 5} \right ) |\Gamma^\alpha h| |\partial_v (\Gamma^\alpha h)| d \omega d v \le C\Vert r^{-1} \Gamma^\alpha h \Vert_{L^2 \left (\mathcal C_{{t_0 - r_0}} \left ({4 v_0 \over 5},v_0 \right ) \right )} \Vert \overline{\partial} (\Gamma^\alpha h) \Vert_{L^2 \left (\mathcal C_{{t_0 - r_0}} \left ({4 v_0 \over 5},v_0 \right ) \right )},
    \end{aligned}
\end{equation}
where we have used the fact that $v$, $v_0$, and $r$ are all comparable on the truncated cone $\mathcal C_{u_0} \left ({4 v_0 \over 5},v_0 \right )$. We shall now control the term
\[
{v_0 \over 5} \int_{S^2} (\Gamma^\alpha h)^2 (u_0,v_0,\omega) d \omega
\]
by the energy, from which the desired result will follow.

We have that
\begin{equation}
    \begin{aligned}
    v_0 \int_{S^2} (\Gamma^\alpha h)^2 (u_0,v_0,\omega) d \omega = -2 v_0 \int_{r_0}^{2 t_0} \int(\Gamma^\alpha h) \partial_r (\Gamma^\alpha h) (t_0,r,\omega) d \omega d r,
    \end{aligned}
\end{equation}
where we have changed into $(t,r,\omega)$ coordinates, and where we have used the fact that $h$ is supported in the region where $t \ge {r \over 2}$ for $t$ large. Thus, we have that
\begin{equation}\label{middlehHardyproductRHS}
    \begin{aligned}
    v_0 \int_{S^2} (\Gamma^\alpha h)^2 (u_0,v_0,\omega) d \omega \le 2 v_0 \left(\sup_{r_0 \le r \le 2 t_0} \int_{S^2} |\Gamma^\alpha h| (t_0,r,\omega) d \omega \right)\int_{r_0}^{2 t_0} \int |\partial_r \Gamma^{\alpha} h| d \omega d r.
    \end{aligned}
\end{equation}
We now note that in $\eqref{middlehHardyproductRHS}$, the factor in parentheses on the right-hand side can be bounded by the last factor by integrating in $r$. Thus we obtain
\begin{equation}\label{middlehHardy}
    \begin{aligned}
    v_0 \int_{S^2} (\Gamma^\alpha h)^2 (u_0,v_0,\omega) d \omega \le 2 v_0 \left(\int_{r_0}^{2 t_0} \int |\partial_r \Gamma^{\alpha} h| d \omega d r\right)^2.
    \end{aligned}
\end{equation}
Then, because $S^2$ is a finite measure space and because the length of integration in $r$ is comparable to $v_0$ by the support of $h$, we use Cauchy-Schwarz with respect to the $d\omega dr$ measure to obtain
\begin{equation}\label{Hardynextingredient}
\int_{r_0}^{2 t_0} \int |\partial_r \Gamma^{\alpha} h| d \omega d r  \le C v_0^{{1 \over 2}} \Vert r^{-1} \partial \Gamma^\alpha h \Vert_{L^2 (\Sigma_{t_0})}.
\end{equation}
Plugging this into \eqref{middlehHardy} and taking the square root, we obtain
\begin{equation}
    \begin{aligned}
    \sup_{r_0 \le r \le 2 t_0} v_0^{{1 \over 2}} \left (\int_{S^2} (\Gamma^\alpha h)^2 (u_0,v_0,\omega) d \omega \right )^{1 \over 2} \le C v_0 \Vert r^{-1} \partial \Gamma^\alpha h \Vert_{L^2 (\Sigma_t)}.
    \end{aligned}
\end{equation}
Because $r$ is controlled by $C v_0$ in the support of $h$, we have that this last expression is controlled by $\Vert \partial \Gamma^\alpha h \Vert_{L^2 (\Sigma_t)}$. This finishes bounding the second term on the right hand side of $\eqref{Prop17someintermediatestep}$, completing the proof of the proposition.
\end{proof}

The Klainerman-Sobolev inequality along with these inequalities using the characteristic energy on outgoing cones will allow us to establish improved pointwise decay for the good derivatives everywhere.

\section{Commutators}
We shall need to compute the commutators between all of the commutation fields, and also the commutators between the commutation fields and the good derivatives $\overline{\partial}$. We have the following result.
\begin{lemma} \label{lem:Commutators}
Let $h : \R^{3 + 1} \rightarrow \R$ be a smooth function. Moreover, let $\Gamma$, $\Gamma_1$, and $\Gamma_2$ denote arbitrary commutation fields, let $\partial$ denote an arbitrary translation field, let $\Omega$ denote an arbitrary rotation field, and let $\overline{\partial}$ denote an arbitrary good derivative. We have that
\begin{enumerate}
    \item $|[\Gamma,\partial] h| \le C |\partial h|$,
    \item $|[\Gamma_1,\Gamma_2] h| \le C \sum_{|\alpha| = 1} |\Gamma^\alpha h|$,
    \item $|[\Omega,\overline{\partial}] h| \le C |\overline{\partial} h|$,
    \item $|[\Gamma,\overline{\partial}] h| \le C |\overline{\partial} h| + {C \over r} \sum_{|\alpha| = 1} |\Gamma^\alpha h|$.
\end{enumerate}
\end{lemma}

\begin{proof}
The first two identities can be found in Chapter $2$ of \cite{Sogge08}. The third follows immediately from taking $\overline{\partial}$ as ${1 \over r} \Omega$ and $\partial_v$ as in Section~\ref{sec:notation}. For the fourth identity, we must only consider the case of $\overline{\partial} = \partial_v$, as the case of $\overline{\partial} = {1 \over r} \Omega$ follows from the other three identities. Moreover, we must only consider the case of $\Gamma$ being the scaling vector field or a boost, as the other case is the third identity. We now check these remaining two cases.

When $\Gamma = S = t \partial_t + r \partial_r$, we have that
\[
[\partial_v,\Gamma] = \partial_v,
\]
giving us the desired result. When $\Gamma$ is a boost, we may assume without loss of generality that $\Gamma = x \partial_t + t \partial_x$. Then, we note that
\[
[\partial_v,\Gamma] = \partial_x + {x \over r} \partial_t.
\]
We shall first write this as ${x \over r} \partial_v$ plus an error. We have that
\begin{equation}
    \begin{aligned}
    \partial_x + {x \over r} \partial_t = {r^2 + x^2 - x^2 \over r^2} \partial_x + {x y \over r^2} \partial_y - {x y \over r^2} \partial_y + {x z \over r^2} \partial_z - {x z \over r^2} \partial_z + {x \over r} \partial_t
    \\ = {x \over r} \partial_v + {r^2 - x^2 \over r^2} \partial_x - {x y \over r^2} \partial_y - {x z \over r^2} \partial_z.
    \end{aligned}
\end{equation}
Now, we note that $x \partial_y - y \partial_x = \Omega_{x y}$. Thus, we have that
\[
{x y \over r^2} \partial_y = {y \over r^2} \Omega_{x y} + {y^2 \over r^2} \partial_x.
\]
Similarly, we have that
\[
{x z \over r^2} \partial_z = {z \over r^2} \Omega_{x z} + {z^2 \over r^2} \partial_x.
\]
Thus, we have that
\begin{equation}
    \begin{aligned}
    \partial_x + {x \over r} \partial_t = {x \over r} \partial_v + {r^2 - x^2 \over r^2} \partial_x - {y \over r^2} \Omega_{x y} - {y^2 \over r^2} \partial_x - {z \over r^2} \Omega_{x z} - {z^2 \over r^2} \partial_x = {x \over r} \partial_v - {y \over r^2} \Omega_{x y} - {z \over r^2} \Omega_{x z},
    \end{aligned}
\end{equation}
giving us the desired result.

\end{proof}

\section{Bootstrap Assumptions} \label{sec:BootstrapAssumptions}
In the following, we shall use $\Gamma^{\alpha,i}$ to denote a string of commutation vector fields where at most $i$ of them are weighted. Because we only commute with two derivatives, we have that $0 \le i \le 2$. The energy associated with $\Gamma^{\alpha,2}$ is allowed to grow like $(1 + t)^\delta$.

We shall consider two separate energies, $E_1 [\gamma] (s)^2$ and $E_2 [\gamma] (s)^2$. They shall both measure a supremum of energy on the time slices $\Sigma_t$ and outgoing cones truncated above by $\Sigma_t$. We take
\[
E_1 (s) = \sup_{\substack{0 \le t \le s\\ -1\le u\le t}} \sum_{|\alpha| \le 2} \Vert \partial \Gamma^{\alpha,1} \gamma \Vert_{L^2 (\Sigma_t)} + \Vert \overline{\partial} \Gamma^{\alpha,1} \gamma \Vert_{L^2 (\mathcal C_u (|u|,2 t - u))},
\]
and we take
\[
E_2 (s) = \sup_{\substack{0 \le t \le s\\ -1\le u\le t}} \sum_{|\alpha| \le 2} (1 + t)^{-\delta} \Vert \partial \Gamma^{\alpha,2} \gamma \Vert_{L^2 (\Sigma_t)} + (1 + t)^{-\delta} \Vert \overline{\partial} \Gamma^{\alpha,2} \gamma \Vert_{L^2 (\mathcal C_u (|u|,2 t - u))}.
\]
Note that we will usually work with the quantities $E_1,E_2$ which are the square roots of the energies $E_1^2, E_2^2$. This is purely a notational convenience.

We shall now take various bootstrap assumptions on $\gamma$. The remainder of the proof of the global stability of the plane waves will involve recovering the bootstrap assumptions. Most of these bootstrap assumptions will be recovered easily from the various embedding theorems along with some minor arguments after recovering the bootstrap assumptions for the energy. We have listed all of them to record all of the estimates we shall use in recovering the bootstrap assumptions for the energies $E_1 (T)^2$ and $E_2 (T)^2$, which are the only steps that require controlling nonlinear terms.

In the following, we fix some large, positive $p$ in terms of $\delta$. More precisely, we pick
\begin{equation}\label{pdef}
    p \ge {2 \over \delta}.
\end{equation}
This $p$ will be used for angular Sobolev embeddings. Indeed, we shall use Proposition~\ref{prop:LowRegKS}, which gives us control of the an appropriate mixed Lebesgue space norm in terms of commuting with a single weighted commutation field. The choice of $p$ must be large enough in order to take advantage of the volume of $\chi_{S_t}$ (see Lemma~\ref{lem:VolumeEst}).

With $\Gamma$ an arbitrary commutation field and $\Gamma^{\alpha,i}$ an arbitrary string of commutation fields where at most $i$ of them are weighted, we let $T$ be the maximal time such that the following bootstrap assumptions are true:
\subsection{Bootstrap assumption list}\label{BtstrpA}
\begin{align}
    E_1 (T) \le \epsilon^{{3 \over 4}},\label{Btstrp1}
    \\ E_2 (T) \le \epsilon^{{3 \over 4}},\label{Btstrp2}
    \\ \sup_{0 \le t \le T} (1 + t + r)^{1 - \delta} (1 + |u|)^{{1 \over 2}} |\partial \gamma| (t,r,\omega) \le \epsilon^{{3 \over 4}},\label{Btstrp3}
    \\ \sup_{0 \le t \le T} (1 + t + r)^{{3 \over 2} - \delta} |\overline{\partial} \gamma| (t,r,\omega) \le \epsilon^{{3 \over 4}},\label{Btstrp4}
    \\ \Vert (1 + |u|)^{-{1 \over 2} - {\delta \over 2}} \overline{\partial} \gamma \Vert_{L_t^2 [0,T] L_x^2} \le \epsilon^{{3 \over 4}},\label{Btstrp5}
    \\ \Vert (1 + |u|)^{-{1 \over 2} - {\delta \over 2}} \overline{\partial} \Gamma^{\alpha,1} \gamma \Vert_{L_t^2 [0,T] L_x^2} \le \epsilon^{{3 \over 4}},\label{Btstrp6}
    \\ \Vert (1 + t)^{-2 \delta} (1 + |u|)^{-{1 \over 2} - {\delta \over 2}} \overline{\partial} \Gamma^{\alpha,2} \gamma \Vert_{L_t^2 [0,T] L_x^2} \le \epsilon^{{3 \over 4}},\label{Btstrp7}
    \\ \sup_{0 \le t \le T}\Vert (1 + |u|)^{{1 \over 4}} (1 + t + r)^{{1 \over 2} - {3 \delta \over 4}} \partial \Gamma \gamma \Vert_{L^4 (\Sigma_t)} \le \epsilon^{{3 \over 4}},\label{Btstrp8}
    \\ \Vert (1 + |u|)^{-{1 \over 4} - {\delta \over 2}} (1 + t + r)^{{1 \over 2} - {3 \delta \over 2}} \overline{\partial} \Gamma \gamma \Vert_{L_t^2 [0,T] L_x^4} \le \epsilon^{{3 \over 4}},\label{Btstrp9}
    \\ \Vert \chi_{S_t} r \gamma \Vert_{L_t^\infty [0,T] L_r^\infty L_\omega^p} \le \epsilon^{{3 \over 4}},\label{Btstrp10}
    \\ \Vert \chi_{S_t} r \partial \gamma \Vert_{L_t^\infty [0,T] L_r^\infty L_\omega^p} \le \epsilon^{{3 \over 4}},\label{Btstrp11}
    \\ \Vert (1 + |u|)^{-{1 \over 2} - {\delta \over 2}} r \overline{\partial} \gamma \Vert_{L_t^2 [0,T] L_r^2 L_\omega^p} \le \epsilon^{{3 \over 4}},\label{Btstrp12}
    \\ \Vert (1 + t)^{-\delta} r \partial \Gamma \gamma \Vert_{L_t^\infty [0,T] L_r^2 L_\omega^p} \le \epsilon^{{3 \over 4}},\label{Btstrp14}
    \\ \Vert (1 + t)^{-2 \delta} (1 + |u|)^{-{1 \over 2} - {\delta \over 2}} r \overline{\partial} \Gamma \gamma \Vert_{L_t^2 [0,T] L_r^2 L_\omega^p} \le \epsilon^{{3\over 4}},\label{Btstrp15}
    \\ \sup_{0 \le t \le T} (1 + t + r)^{1 - \delta} \chi_{S_t} |\gamma| (t,r,\omega) \le \epsilon^{{3 \over 4}},\label{Btstrp16}
    \\ \sup_{0 \le t \le T} \Vert (1 + t)^{-\delta} \chi_{S_t} \Gamma^{\alpha,2} \gamma \Vert_{L^2 (\Sigma_t)} \le \epsilon^{{3 \over 4}}.\label{Btstrp17}
    \\ \Vert (1 + t)^{-\delta} \chi_{S_t} r \Gamma \gamma \Vert_{L_t^\infty [0,T] L_r^\infty L_\omega^p} \le \epsilon^{{3 \over 4}},\label{Btstrp13}
\end{align}

\subsection{Continuation of discussion regarding bootstrap assumptions}
We shall improve the bounds from $\epsilon^{{3 \over 4}}$ to being $C \epsilon$. This will recover the bootstrap assumptions when $\epsilon$ is sufficiently small, giving us the desired result.

The main remaining difficulty in establishing Theorem~\ref{thm:main} is recovering the bootstrap assumptions on the energy. We have the following proposition.

\begin{proposition} \label{prop:EnergyMain}
The bootstrap assumptions in Section~\ref{BtstrpA} imply that $E_1 (T) \le C \epsilon$ and $E_2 (T) \le C \epsilon$.
\end{proposition}

We shall now assume that we have shown Proposition~\ref{prop:EnergyMain}. These estimates are established in Section~\ref{sec:ClosingEnergy}. We shall show that, as a result of this, we can recover all of the other bootstrap assumptions. For each of these, the portion of the estimates where $t<10$ follow from the usual Sobolev embedding on a spacetime cube of side length 100, so we will only worry about the parts of these norms that have $t>2$. 

We shall first use Proposition~$\ref{prop:EnergyMain}$ to recover the pointwise bootstrap assumptions on $\partial \gamma$ and $\overline{\partial} \gamma$. They are direct consequences of the Klainerman-Sobolev inequalities we have established.

\begin{lemma}\label{lem:bootstrappartialgammabound}
Assuming Proposition~\ref{prop:EnergyMain}, we recover bootstrap assumptions~\eqref{Btstrp3} and \eqref{Btstrp4} and improve the bound to $C \epsilon$. Namely, we derive the estimates
\begin{enumerate}
    \item $(1 + t + r)^{1 - \delta} (1 + |u|)^{{1 \over 2}} |\partial \gamma| (t,r,\omega) \le C E_2(t)$,
    \item $(1 + t + r)^{{3 \over 2} - \delta} |\overline{\partial} \gamma| (t,r,\omega) \le C E_2(t)$.
\end{enumerate}
\end{lemma}

\begin{proof}
The first estimate follows from applying Proposition~\ref{prop:KlaiSob} and noting that the right hand side is controlled by $E_2$ recovered in Proposition~\ref{prop:EnergyMain}. The second estimate similarly follows similarly using Proposition~\ref{prop:ConeKS} instead when $|u|<\frac{t}{100}$ and follows directly from the first estimate when $|u|\ge\frac{t}{100}$.
\end{proof}

We shall now recover the spacetime $L^2$ estimates on the good derivatives.

\begin{lemma}\label{lem:bootstrap567}
Assuming Proposition~\ref{prop:EnergyMain}, we recover bootstrap assumptions~\eqref{Btstrp5},\eqref{Btstrp6} ,\eqref{Btstrp7}, namely we have that
\begin{enumerate}
    \item $\Vert (1 + |u|)^{-{1 \over 2} - {\delta \over 2}} \overline{\partial} \gamma \Vert_{L_t^2 [2,T] L_x^2} \le C \epsilon$,
    \item $\Vert (1 + |u|)^{-{1 \over 2} - {\delta \over 2}} \overline{\partial} \Gamma^{\alpha,1} \gamma \Vert_{L_t^2 [2,T] L_x^2} \le C \epsilon$,
    \item $\Vert (1 + t)^{-2 \delta} (1 + |u|)^{-{1 \over 2} - {\delta \over 2}} \overline{\partial} \Gamma^{\alpha,2} \gamma \Vert_{L_t^2 [2,T] L_x^2} \le C \epsilon$.
\end{enumerate}
\end{lemma}
\begin{proof}
The first two parts follow immediately from the fact that $E_1 (T)^2$ controls the characteristic energy along with the fact that $(1 + |u|)^{-1 - \delta}$ is integrable in $u$. We now turn to the third part.

We begin by noting that, as a consequence of Proposition~\ref{prop:EnergyMain}, we have that
\begin{equation} \label{CharGrowth}
    \begin{aligned}
    \int_{|u|}^{2 s - u} \int_{S^2} (\overline{\partial} \Gamma^{\alpha,2} \gamma)^2 r^2 d \omega d v \le C \epsilon^2 (1 + s)^{2 \delta}
    \end{aligned}
\end{equation}
for every $s \le T$, where we are integrating on the cone $\mathcal C_u$. Thus, with $h(u,v) = \int_{S^2} (\overline{\partial} \Gamma^{\alpha,2} \gamma)^2 r^2 d \omega$ and integrating by parts in $v$, we have that
\begin{equation}\label{eq:CharIntByParts}
\int_{|u|}^{2 T - u} (1 + t)^{-4 \delta} h d v = (1 + T)^{-4 \delta} \int_{|u|}^{2T-u} h(u,v')dv' + 2 \delta \int_{|u|}^{2 T - u} \left (1 + {1 \over 2} (v + u) \right )^{-1 - 4 \delta} \int_{|u|}^v h(u,v') d v' d v.
\end{equation}
Now, we note that
\[
\int_{|u|}^v h(u,v') d v' \le C \epsilon^2 (1 + t)^{2 \delta} = C \epsilon^2 \left (1 + {1 \over 2} (v + u) \right )^{2 \delta}.
\]
Thus, we have that
\[
\delta \int_{|u|}^{2 T - u} \left (1 + {1 \over 2} (v + u) \right )^{-1 - 4 \delta} \int_{|u|}^v h(u,v') d v' d v \le C \delta \epsilon^2 \int_{|u|}^{2 T - u} \left (1 + {1 \over 2} (v + u) \right )^{-1 - 2 \delta} d v \le C \epsilon^2
\]
and similarly for the first term in \eqref{eq:CharIntByParts}. Multiplying~\eqref{eq:CharIntByParts} by $(1 + |u|)^{-1 - \delta}$, integrating in $u$, we get that
\begin{equation*}
    \begin{aligned}
    \Vert (1 + |u|)^{-{1 \over 2} - {\delta \over 2}} (1 + t)^{-2 \delta} \overline{\partial} \Gamma^{\alpha,2} \gamma \Vert_{L_t^2 L_x^2}^2&\le 4\int_{-1}^T (1 + |u|)^{-1 - \delta} \int_{|u|}^{2 t - u} (1 + t)^{-4 \delta} \int_{S^2}  (\overline{\partial} \Gamma^{\alpha,2} \gamma)^2 r^2 d \omega d v d u\\
    & \le 4\int_{-1}^T (1 + |u|)^{-1 - \delta} \int_{|u|}^{2 t - u} (1 + t)^{-4 \delta}h d v d u
    \\ &\le C \epsilon^2,
    \end{aligned}
\end{equation*}
where we have used the fact that $(1 + |u|)^{-1 - \delta}$ is integrable in $u$. This gives us the desired result.
\end{proof}

We now recover the $L_t^\infty L_x^4$ estimates on all derivatives and the $L_t^2 L_x^4$ estimates on good derivatives. They are both a consequence of the $L_x^6$ Klainerman-Sobolev inequality in Proposition~\ref{prop:KlaiSob} and an interpolation argument.

\begin{lemma} \label{lem:RecoverIntegratedKlaiSob}
Assuming Proposition~\ref{prop:EnergyMain}, we recover bootstrap assumptions~\eqref{Btstrp8} and \eqref{Btstrp9} and improve the bound to $C \epsilon$. Namely, we derive the estimates
\begin{enumerate}
    \item $\Vert (1 + |u|)^{{1 \over 4}} (1 + t + r)^{{1 \over 2} - {3 \delta \over 4}} \partial \Gamma \gamma \Vert_{L^4 (\Sigma_t)} \le CE_2(T)$,
    \item $\Vert (1 + |u|)^{-{1 \over 4} - {\delta \over 2}} (1 + t + r)^{{1 \over 2} - {3 \delta \over 4}} \overline{\partial} \Gamma \gamma \Vert_{L_t^2 L_x^4} \le CE_2(T)$.
\end{enumerate}
\end{lemma}
\begin{proof}
Both of these estimates are a consequence of interpolating between unweighted bounds and weighted bounds. We shall begin by proving the following two estimates which are $L^6$ in $x$:
\begin{align}
\Vert (1 + |u|)^{{1 \over 3}} (1 + t + r)^{{2 \over 3} - \delta} \partial \Gamma \gamma \Vert_{L_t^\infty [0,T] L_x^6} \le C \epsilon\label{LinftyL6first}\\
\Vert (1 + |u|)^{-{1 \over 6} - {\delta \over 2}} (1 + t + r)^{{2 \over 3} - 2 \delta} \overline{\partial} \Gamma \gamma \Vert_{L_t^2 L_x^6} \le C \epsilon\label{LinftyL6second}.
\end{align}
The first of these estimates follows immediately by applying Proposition~\ref{prop:KlaiSob} and Proposition \ref{prop:EnergyMain}. We now turn to proving the second estimate.

We must only consider the region where $t \ge 10$. Indeed, when $t < 10$, the following argument works where instead of using the Klainerman-Sobolev inequalities we have proved, we use the regular Sobolev inequality. This can be done because the functions $t$, $r$, and $u$ are all comparable to $1$ in the support of $\gamma$ when $t \le 10$.

With this restriction on $t$, we further consider two regions, the region $A$ where $|u| \le {t \over 10}$, and the region $B$ where $t \ge {21 r \over 20}$. We begin with the region away from the light cone where $t \ge {21 r \over 20}$.

In this region, we apply the Klainerman-Sobolev inequality in Proposition~\ref{prop:KlaiSob} to a suitable cutoff $\chi$ times $\overline{\partial} \Gamma \gamma$ (see the proof of Proposition~\ref{prop:KlaiSob}, where such a cutoff is also used). This, along with Proposition~\ref{prop:EnergyMain}, gives us that
\[
\Vert (1 + |u|)^{{1 \over 3}} (1 + t + r)^{{2 \over 3} - \delta} \overline{\partial} \Gamma \gamma \Vert_{L^6 (\Sigma_t \cap A)} \le C \epsilon.
\]
Moreover, in this region, we note that $|u| \ge {t \over 100}$. Using this, we have that
\[
\Vert (1 + |u|)^{-{1 \over 6} - {\delta \over 2}} (1 + t + r)^{{2 \over 3} - 2 \delta} \overline{\partial} \Gamma \gamma \Vert_{L^6 (\Sigma_t \cap A)} \le C \epsilon (1 + t)^{-{1 \over 2} - {3 \delta \over 2}}.
\]
Integrating in $t$ then gives us the desired result in this region.

We now consider the region where along the light cone where $|u| \le {t \over 10}$. Let $\chi$ be a suitable cutoff localizing along the light cone as in Proposition~\ref{prop:ConeSob}. We can also take $\chi$ to localize in the region where $t \ge 10$. We now observe that $|\Gamma u| \le C |u|$ where $\Gamma$ is any commutation field and where $u = t - r$. We can now use one of the Klainerman-Sobolev inequalities. With $g(t,r,\omega) = \chi (1 + |u|)^{-{1 \over 2} - {\delta \over 2}} (1 + t + r)^{-2 \delta} \overline{\partial} \Gamma \gamma (t,r,\omega)$ and using one of the Klainerman-Sobolev inequalities in Proposition~\ref{prop:KlaiSob}, we have that
\begin{equation}
    \begin{aligned}
    \Vert \chi (1 + |u|)^{-{1 \over 6} - {\delta \over 2}} (1 + t + r)^{{2 \over 3} - 2 \delta} \overline{\partial} \Gamma \gamma \Vert_{L^6 (\Sigma_t)} \le C \sum_{|\mu| \le 1} \Vert \Gamma^\mu g \Vert_{L^2 (\Sigma_t)}.
    \end{aligned}
\end{equation}
Thus, taking the $L_t^2$ norm in the $t$ interval $[10,s]$ of both sides, we get that
\begin{equation}
    \begin{aligned}
    \Vert \chi (1 + |u|)^{-{1 \over 6} - {\delta \over 2}} (1 + t + r)^{{2 \over 3} - 2 \delta} \overline{\partial} \Gamma \gamma \Vert_{L_t^2 ([10,s]) L_x^6} \le C \sum_{|\mu| \le 1} \Vert \Gamma^\mu g \Vert_{L_t^2 [10,s] L_x^2}.
    \end{aligned}
\end{equation}

Now, we have that
\begin{align}
    \sum_{|\mu| \le 1} \Vert \Gamma^\mu g \Vert_{L_t^2 [10,s] L_x^2} &\le C \sum_{|\mu| \le 1} \Vert \chi (1 + t + r)^{-2 \delta} (1 + |u|)^{-{1 \over 2} - {\delta \over 2}} \Gamma^\mu \overline{\partial} \Gamma \gamma \Vert_{L_t^2 [10,s] L_x^2}
    \\ &+ C \Vert (1 + t + r)^{-2 \delta} (1 + |u|)^{-{1 \over 2} - {\delta \over 2}} \overline{\partial} \Gamma \gamma \Vert_{L_t^2 [10,s] L_x^2}.
\end{align}

Using part 3 of Lemma~\ref{lem:bootstrap567}, the second of these terms is controlled. Moreover, using Lemma~\ref{lem:Commutators}, we have that
\begin{equation}
    \begin{aligned}
    \Vert \chi (1 + t + r)^{-2 \delta} (1 + |u|)^{-{1 \over 2} - \delta} \Gamma^\mu \overline{\partial} \Gamma \gamma \Vert_{L_t^2 L_x^2} &\le C \Vert \chi (1 + t)^{-2 \delta} (1 + |u|)^{-{1 \over 2} - {\delta \over 2}} \overline{\partial} \Gamma^\mu \Gamma \gamma \Vert_{L_t^2 L_x^2}
    \\ + C \sum_{|\beta|\le|\mu|}\Vert \chi (1 + t)^{-2 \delta} (1 + &|u|)^{-{1 \over 2} - {\delta \over 2}} r^{-1} \Gamma^\beta \Gamma \gamma \Vert_{L_t^2 L_x^2}.
    \end{aligned}
\end{equation}

The first of these terms is controlled by part 3 of Lemma~\ref{lem:bootstrap567}. Thus, we only need to control
\[
\Vert \chi (1 + t)^{-2 \delta} (1 + |u|)^{-{1 \over 2} - \delta} r^{-1} \Gamma^\beta \Gamma \gamma \Vert_{L_t^2 L_x^2}.
\]
This term is controlled as long as we control
\[
\sup_{|u| \le {t \over 10}} \Vert \chi (1 + t)^{-2 \delta} r^{-1} \Gamma^\beta \Gamma \gamma \Vert_{L^2 \left (\mathcal C_u \left (u + 2,2 s - u \right ) \right )},
\]
where we note that $2 s - u$ is the $v$ coordinate where the cone $\mathcal C_u$ intersects $\Sigma_s$. This can be controlled in terms of
\[
C \sup_{|u| \le {s \over 10}} \Vert (1 + t)^{-2 \delta} \overline{\partial} \Gamma^\mu \Gamma \gamma \Vert_{L^2 \left (\mathcal C_u \left (10u/9 + 2,2 s - u \right ) \right )} + C \Vert (1 + t)^{-\delta} \partial \Gamma^\mu \Gamma \gamma \Vert_{L_t^\infty L_x^2}
\]
using a similar argument as is used to control the term \eqref{eq:HardyTerm} in Proposition~\ref{prop:ConeKS}. This gives us the desired result.

We now turn to using these $L^6$ in $x$ estimates along with the $L^2$ in $x$ estimates given in Proposition \ref{prop:EnergyMain} and in Lemma~\ref{lem:bootstrap567} in order to establish the result. We have that
\begin{equation}
    \begin{aligned}
    \Vert (1 + |u|)^{{1 \over 4}} (1 + t + r)^{{1 \over 2} - {3 \delta \over 4}} \partial \Gamma \gamma \Vert_{L^4 (\Sigma_t)}^4 = \int_{\Sigma_t} (1 + |u|) (1 + t + r)^{2 - 3 \delta} |\partial \Gamma \gamma|^4 d x
    \\ \le \Vert \partial \Gamma \gamma \Vert_{L^2 (\Sigma_t)} \left (\int_{\Sigma_t} (1 + |u|)^2 (1 + t + r)^{4 - 6 \delta} |\partial \Gamma \gamma|^6 d x \right )^{{1 \over 2}}
    \\ = \Vert \partial \Gamma \gamma \Vert_{L^2 (\Sigma_t)} \Vert (1 + |u|)^{{1 \over 3}} (1 + t + r)^{{2 \over 3} - \delta} \partial \Gamma \gamma \Vert_{L^6 (\Sigma_t)}^3.
    \end{aligned}
\end{equation}
Thus, we have that
\begin{equation}
    \begin{aligned}
    \Vert (1 + |u|)^{{1 \over 4}} (1 + t + r)^{{1 \over 2} - {3 \delta \over 4}} \partial \Gamma \gamma \Vert_{L^4 (\Sigma_t)} \le \Vert \partial \Gamma \gamma \Vert_{L^2 (\Sigma_t)}^{{1 \over 4}} \Vert (1 + |u|)^{{1 \over 3}} (1 + t + r)^{{2 \over 3} - \delta} \partial \Gamma \gamma \Vert_{L^6 (\Sigma_t)}^{{3 \over 4}}\le C\epsilon,
    \end{aligned}
\end{equation}
giving us the first inequality in the lemma statement.

For the second inequality, we proceed in a similar way. We have that
\begin{equation}
    \begin{aligned}
    \Vert (1 + |u|)^{-{1 \over 4} - {\delta \over 2}} (1 + t + r)^{{1 \over 2} - {\delta \over 2}} \overline{\partial} \Gamma \gamma \Vert_{L^4 (\Sigma_t)}^4 = \int_{\Sigma_t} (1 + |u|)^{-1 - 2 \delta} (1 + t + r)^{2 - 3 \delta} |\overline{\partial} \Gamma \gamma|^4 d x
    \\ \le \Vert (1 + |u|)^{-{1 \over 2} - {\delta \over 2}} \overline{\partial} \Gamma \gamma \Vert_{L^2 (\Sigma_t)} \left (\int_{\Sigma_t} (1 + |u|)^{-1 - 3 \delta} (1 + t + r)^{4 - 6 \delta} |\overline{\partial} \Gamma \gamma|^6 d x \right )^{{1 \over 2}}
    \\ = \Vert (1 + |u|)^{-{1 \over 2} - {\delta \over 2}} \overline{\partial} \Gamma \gamma \Vert_{L^2 (\Sigma_t)} \Vert (1 + |u|)^{-{1 \over 6} - {\delta \over 2}} (1 + s + r)^{{2 \over 3} - \delta} \overline{\partial} \Gamma \gamma \Vert_{L^6 (\Sigma_t)}^3.
    \end{aligned}
\end{equation}
Thus, we have that
\begin{equation}
    \begin{aligned}
    \Vert (1 + |u|)^{-{1 \over 4} - {\delta \over 2}} &(1 + s + r)^{{1 \over 2} - {\delta \over 2}} \overline{\partial} \Gamma \gamma \Vert_{L^4 (\Sigma_t)}
    \\ &\le \Vert (1 + |u|)^{-{1 \over 2} - {\delta \over 2}} \overline{\partial} \Gamma \Vert_{L^2 (\Sigma_t)}^{{1 \over 4}} \Vert (1 + |u|)^{-{1 \over 6} - {\delta \over 2}} (1 + s + r)^{{2 \over 3} - \delta} \overline{\partial} \Gamma \gamma \Vert_{L^6 (\Sigma_t)}^{{3 \over 4}}
    \\ &\le {1 \over 4} \Vert (1 + |u|)^{-{1 \over 2} - {\delta \over 2}} \overline{\partial} \Gamma \Vert_{L^2 (\Sigma_t)} + {3 \over 4} \Vert (1 + |u|)^{-{1 \over 6} - {\delta \over 2}} (1 + s + r)^{{2 \over 3} - \delta} \overline{\partial} \Gamma \gamma \Vert_{L^6 (\Sigma_t)}.
    \end{aligned}
\end{equation}
Taking $L^2$ in $t$ gives us the second inequality in the lemma statement.
\end{proof}

We finally recover the remaining bootstrap assumptions assuming Proposition~\ref{prop:EnergyMain}. The key in the following result is that the quantities are controlled in terms of $E_1 (T)$ and not $E_2 (T)$. If the quantities were controlled in terms of $E_2 (T)$ instead, the bootstrap assumptions would not close, as the growth rate of the norms in $E_2 (T)$ are not consistent with recovering this same growth rate.

\begin{lemma} \label{lem:AngularEst}
Assuming Proposition \ref{prop:EnergyMain}, we recover bootstrap assumptions~\eqref{Btstrp10},~\eqref{Btstrp11},~\eqref{Btstrp12},~\eqref{Btstrp13} and improve the bound to $C \epsilon$. Namely, we derive the estimates
\begin{enumerate}
    \item $\Vert \chi_{S_t} r \gamma \Vert_{L_t^\infty [2,T] L_r^\infty L_\omega^p} \le C E_1 (T)$,
    \item $\Vert \chi_{S_t} r \partial \gamma \Vert_{L_t^\infty [2,T] L_\infty^2 L_\omega^p} \le C E_1 (T)$,
    \item $\Vert (1 + |u|)^{-{1 \over 2} - \delta} r \overline{\partial} \gamma \Vert_{L_t^2 [2,s] L_r^2 L_\omega^p} \le C E_1 (T)$.
\end{enumerate}
\end{lemma}

\begin{proof}
The first two estimates are of similar form, so we handle them together by letting $f=\gamma$ for the first estimate and $f=\partial\gamma$ for the second estimate.

We use the fact that $t$ and $r$ are comparable on $S_t$ to get that for every $t\in [s,\infty]$, we have
\[
\Vert \chi_{S_t} r f \Vert_{L_r^\infty L_\omega^p}\le Ct \Vert f \Vert_{L_r^\infty[t-1,t+1] L_\omega^p}
\]
Using $\Omega_{ij}$ to denote a rotation vector field, we then use Proposition~\ref{prop:Hardy}, Sobolev embedding on the sphere, and the fact that $t$ is comparable to $r$ in the relevant region to get
\begin{align*}
t \Vert f \Vert_{L_r^\infty[t-1,t+1] L_\omega^p}&\le t \Vert \partial_rf \Vert_{L_r^2[t-1,t+1] L_\omega^p}\\
&\le t \Vert \partial_rf \Vert_{L_r^2[t-1,t+1] L_\omega^2}+\sum_{i,j} t \Vert \partial_r\Omega_{ij}f \Vert_{L_r^2[t-1,t+1] L_\omega^2}\\
&\le C\Vert r\partial_rf \Vert_{L_r^2[t-1,t+1] L_\omega^2}+C\Vert r\partial_r\Omega_{ij}f \Vert_{L_r^2[t-1,t+1] L_\omega^2}\\
&\le CE_1(T)
\end{align*}

For the third part, we first note the following commutation:
\begin{equation}\label{rotgoodcommute}
|\Omega_{ij} r\overline{\partial}\gamma|\le C|r\overline{\partial}\gamma|+|r\overline{\partial}\Omega_{ij}\gamma|.
\end{equation}
This is true when $\overline{\partial}=\partial_v$ because we can just commute with the rotation, and it is true when $\overline{\partial}=\frac{1}{r}\Omega_{k\ell}$ because we combine with the factor $r$ to obtain a rotation, commute the two rotations, then factor $r$ out again. We now use Sobolev embedding on spheres and $\eqref{rotgoodcommute}$ to get that
\[
\Vert r \overline{\partial} \gamma \Vert_{L_v^2[|u|,2T-u] L_\omega^p} \le C\Vert r \overline{\partial} \gamma \Vert_{L_v^2[|u|,2T-u] L_\omega^2}+C\Vert r \overline{\partial} \Omega_{ij}\gamma \Vert_{L_v^2[|u|,2T-u] L_\omega^2}\le C E_1 (T)
\]
Now taking the $L^2$ norm in $u$ and noting that $(1+|u|)^{-1-\delta}$ is integrable in $u$, we obtain the third part of the lemma statement.
\end{proof}

\begin{lemma}
Assuming Proposition \ref{prop:EnergyMain}, we recover bootstrap assumptions~\eqref{Btstrp14},~\eqref{Btstrp15},~\eqref{Btstrp16},~\eqref{Btstrp17},\eqref{Btstrp13} and improve the bound to $C \epsilon$. Namely, we derive the estimates
\begin{enumerate}
    \item $\Vert (1 + t)^{-\delta} r \partial \Gamma \gamma \Vert_{L_t^\infty [2,T] L_r^2 L_\omega^p} \le C E_2 (T)$,
    \item $\Vert (1 + t)^{-2 \delta} (1 + |u|)^{-{1 \over 2} - {\delta \over 2}} r \overline{\partial} \Gamma \gamma \Vert_{L_t^2 [2,T] L_r^2 L_\omega^p} \le C E_2 (T)$,
    \item $\sup_{2 \le t \le T} (1 + t + r)^{1 - \delta} \chi_{S_t} |\gamma| (t,r,\omega) \le C E_2 (T)$,
    \item $\sup_{2 \le t \le T} \Vert (1 + t)^{-\delta} \chi_{S_t} \Gamma^{\alpha,2} \gamma \Vert_{L^2 (\Sigma_t)} \le C E_2 (T)$,
    \item $\Vert (1 + t)^{-\delta} \chi_{S_t} r \Gamma \gamma \Vert_{L_t^\infty [2,T] L_r^\infty L_\omega^p} \le C E_2 (T)$.
\end{enumerate}
\end{lemma}
\begin{proof}
For the first part of this lemma, we use Sobolev embedding on spheres and commute to get that
\begin{align*}
\Vert (1 + t)^{-\delta} r \partial \Gamma \gamma \Vert_{L_t^\infty [2,T] L_r^2 L_\omega^p}&\le \Vert (1 + t)^{-\delta} r \partial \Gamma \gamma \Vert_{L_t^\infty [2,T] L_r^2 L_\omega^2}+\sum_{i,j}\Vert \Omega_{ij} (1 + t)^{-\delta} r \partial \Gamma \gamma \Vert_{L_t^\infty [2,T] L_r^2 L_\omega^2}\\
&\le \Vert (1 + t)^{-\delta} r \partial \Gamma \gamma \Vert_{L_t^\infty [2,T] L_r^2 L_\omega^2}+\sum_{i,j}\Vert (1 + t)^{-\delta} r \partial \Omega_{ij} \Gamma \gamma \Vert_{L_t^\infty [2,T] L_r^2 L_\omega^2}\\ 
&\le E_2(T).
\end{align*}

The second part of this lemma is proved entirely analogously to the third part of Lemma~\ref{lem:AngularEst}. The third part of the Lemma follows from applying the fundamental theorem of calculus in the $\partial_r$ direction and then using part  1 of Lemma~\ref{lem:bootstrappartialgammabound} (and the fact that $S_t$ has bounded width in the $\partial_r$ direction. The fourth statement is proven analogously to parts 1 and 2 of Lemma~\ref{lem:AngularEst}, except without Sobolev embedding. More specifically, it follows from converting to the $dr d\omega$ measure, applying a Hardy inequality in $r$ to $\Gamma^{\alpha,2} \gamma$ and using once again that the width of $S_t$ in $r$ is bounded.

For the fifth part of this lemma, we use Proposition~\ref{prop:Hardy} to reduce to the first part of the lemma. Indeed, for each $t$, this gives us that
\[
\Vert (1 + t)^{-\delta} \chi_{S_t} r \Gamma \gamma \Vert_{L_r^\infty L_\omega^p}\le \Vert (1 + t)^{-\delta} t \Gamma \gamma \Vert_{L_r^\infty [t - 2,t + 2] L_\omega^p}\le C \Vert (1 + t)^{-\delta} t \partial \Gamma \gamma \Vert_{L_r^2 L_\omega^p} \le C \Vert (1 + t)^{-\delta} r \partial \Gamma \gamma \Vert_{L_r^2 L_\omega^p}.
\]
Taking the $\sup$ for $2 \le t \le T$ then reduces this to the first part of the lemma.
\end{proof}

Now, all that remains is to establish the proof of Proposition \ref{prop:EnergyMain}.

\section{Control of nonlinear error terms} \label{sec:ClosingEnergy}
We now turn to the proof of Proposition \ref{prop:EnergyMain}. We recall the transformed equation \eqref{eq:TransformedEq}. 

First, for convenience, we wish to start at time $t=2$ instead of $t=0$. This way, we guarantee that on $S_t$, we have that $r$ and $t$ are comparable and we can freely interchange them by adding a constant in front. The bounds on all relevant quantities up to time $t=2$ follow easily from standard methods for obtaining well-posedness of semilinear wave equations in $H^3$.

We must commute the equation with appropriate vector fields in order to control $E_1 (T)$ and $E_2 (T)$. We commute the equation with $\Gamma^{\alpha,1}$ in order to control $E_1 (T)$ and with $\Gamma^{\alpha,2}$ in order to control $E_2 (T)$. Now, for $Q$ any quadratic form satisfying the null condition, we note that we have that
\begin{equation} \label{NullFormEst1}
    \begin{aligned}
    |Q(d (A\inv \gamma),d (A\inv \gamma))| &\le C |\partial (A\inv \gamma)| |\overline{\partial} (A\inv \gamma)|
    \\ &\le C |\partial \gamma| |\overline{\partial} \gamma| + C {1 \over \sqrt{1 + t}} \chi_{S_t} |\gamma| |\partial \gamma| + C \chi_{S_t} |\gamma| |\overline{\partial} \gamma| + \frac{C}{\sqrt{1+t}} \chi_{S_t} |\gamma| |\gamma|,
    \end{aligned}
\end{equation}
where we have used Lemma~\ref{lem:frameconversion}. We have used the fact that null forms can be bounded by good derivatives times bad derivatives, see also \cite{LukNotes}. Moreover, we note that (see also \cite{LukNotes}) that any quadratic form satisfying the null condition $Q$ has that
\begin{equation} \label{NullFormCommutation}
    \begin{aligned}
    \Gamma Q(d h_1,d h_2) = Q(d \Gamma h_1,d h_2) + Q(d h_1,d \Gamma h_2) + \tilde{Q} (d h_1,d h_2),
    \end{aligned}
\end{equation}

where $\Gamma$ is any commutation field. Iterating \eqref{NullFormCommutation} and using \eqref{NullFormEst1} gives us that
\begin{equation}\label{eq:NullFormEstimate}
    \begin{aligned}
    \left |\Gamma^\alpha (A Q(d (A\inv \gamma),d (A\inv \gamma))) \right | = \left |\sum_{|\beta_1| + |\beta_2| + \beta_3 \le |\alpha|} \Gamma^{\beta_3} A Q_{\beta_1 \beta_2 \beta_3} (d \Gamma^{\beta_1} (A\inv \gamma),d \Gamma^{\beta_2} (A\inv \gamma)) \right |
    \\ \le C \sum_{|\beta_1| + |\beta_2| + |\beta_3| + |\beta_4| + |\beta_5| \le |\alpha|} |\Gamma^{\beta_3} A| |\Gamma^{\beta_4} A\inv| |\Gamma^{\beta_5} A\inv| 
    \\ \times \left [|\partial \Gamma^{\beta_1} \gamma| |\overline{\partial} \Gamma^{\beta_2} \gamma| + {1 \over \sqrt{1 + t}} \chi_{S_t} |\Gamma^{\beta_1} \gamma| |\partial \Gamma^{\beta_2} \gamma| + \chi_{S_t} |\Gamma^{\beta_1} \gamma| |\overline{\partial} \Gamma^{\beta_2} \gamma| + {1 \over \sqrt{1 + t}} \chi_{S_t} |\Gamma^{\beta_1} \gamma| |\Gamma^{\beta_2} \gamma| \right ]
    \end{aligned}
\end{equation}
for some null forms $Q_{\beta_1 \beta_2 \beta_3}$, and where $\Gamma^\alpha$ is an arbitrary string of commutation fields.

We define the index $\sigma=\beta_3+\beta_4+\beta_5$ which is equal to the number of weighted commutation fields that fall on $A$ or $A^{-1}$ after commuting. By Lemma~\ref{lem:roottfactors}, each weighted vector field can introduce a weight of size at most $\sqrt{1+t}$ in the support of $f$ intersected with the support of $\gamma$. Because we only ever commute with two commutation fields, we have that terms all have $\sigma = 0$, $\sigma = 1$, or $\sigma = 2$. More precisely, when commuting with $\Gamma^{\alpha,2}$, the we get at most $\sigma = 2$, while we get at most $\sigma = 1$ when commuting with $\Gamma^{\alpha,1}$. Now, by the energy estimate (see Proposition~\ref{prop:EnEst}), we have that
\[
E_1^2 (T) \le \epsilon + C \sum_{\Gamma^{\alpha,1}} \sup_{0 \le s \le T} \int_2^s \int_{\Sigma_t} |\Gamma^{\alpha,1} (A m (\nabla (A^{-1} \gamma),\nabla (A^{-1} \gamma)))| |\partial_t \Gamma^{\alpha,1} \gamma| d x d t.
\]
Similarly, we have that
\[
E_2^2 (T) \le \epsilon + C\sup_{0 \le s \le T} (1 + s)^{-2 \delta}  \sum_{\Gamma^{\alpha,1}} \int_2^s \int_{\Sigma_t} |\Gamma^{\alpha,2} (A m (\nabla (A^{-1} \gamma),\nabla (A^{-1} \gamma)))| |\partial_t \Gamma^{\alpha,2} \gamma| d x d t.
\]
We shall set
\[
F_1 (t,x) = |\Gamma^{\alpha,1} (A m (\nabla (A^{-1} \gamma),\nabla (A^{-1} \gamma)))|,
\]
and
\[
F_2 (t,x) = |\Gamma^{\alpha,2} (A m (\nabla (A^{-1} \gamma),\nabla (A^{-1} \gamma)))|.
\]
These depend on the string of commutation vector fields $\Gamma^{\alpha,1}$, $\Gamma^{\alpha,1}$, but for the sake of convenience, we will not write those as parameters.

Now, let $s \le T$ be arbitrary. In order to control $E_1 (T)$ and $E_2 (T)$, it suffices to show that the bootstrap assumptions in Section~\ref{BtstrpA} imply that
\begin{equation}\label{F1error}
\int_2^s \int_{\Sigma_t} F_1 (t,x) |\partial_t \Gamma^{\alpha,1} \gamma| d x d t \le C \epsilon^{{9 \over 4}},
\end{equation}
for $1 \le i \le N$, and that
\begin{equation}\label{F2error}
\int_2^s \int_{\Sigma_t} F_2 (t,x) |\partial_t \Gamma^{\alpha,2} \gamma| d x d t \le C (1 + s)^{2 \delta} \epsilon^{{9 \over 4}}
\end{equation}
for $1 \le i \le N$. We shall now turn to controlling these spacetime integrals using \eqref{eq:NullFormEstimate}. For convenience, we shall also drop the interval $[0,s]$ in the $L_t$ norms, although it should be understood that all norms in $t$ are taken in this interval. This is particularly important when considering the terms with signature $\sigma = 2$, as these are the terms where the spacetime integrals will grow in $s$.

We shall begin with the case where $\Gamma^\alpha$ is $\Gamma^{\alpha,2}$ (i.e., we shall control $E_2 (T)$). In fact, the result for $E_1 (T)$ will follow because the only terms in the nonlinear errors which will grow have signature $\sigma = 2$. Because of this, and because all terms with signature $\sigma \le 1$ that arise in $E_1 (T)$ will also arise in $E_2 (T)$, the result for $E_1 (T)$ will follow. Thus, we turn to controlling the integrals
\[
\int_2^s \int_{\Sigma_t} F_2 (t,x) |\partial_t \Gamma^\alpha \gamma| d x d t.
\]

Because we are commuting with two weighted commutation fields, the signature $\sigma$ may now be $0$, $1$, or $2$. We first consider terms having $\sigma = 2$. These terms are controlled by
\begin{equation}
    \begin{aligned}
    C (1 + t) \chi_{S_t} \left [|\partial \gamma| |\overline{\partial} \gamma| + {1 \over \sqrt{1 + t}} |\gamma| |\partial \gamma| + |\gamma| |\overline{\partial} \gamma| + {1 \over \sqrt{1 + t}} |\gamma| |\gamma| \right ].
    \end{aligned}
\end{equation}

These terms will all be borderline. These are the terms that force this energy to grow. Moreover, we cannot use $E_2 (T)$ in order to control the error terms. We must use the fact that these terms only appear when we are commuting with two weighted commutation fields and neither one falls on the solution $\gamma$ in the nonlinearity, but rather, they both fall on the plane wave solution. Thus, we are free to commute both factors of $\gamma$ in the nonlinear error terms with a single weighted commutation field and a single translation field. This quantity is controlled by $E_1 (T)$, which does not grow. The fact that we can commute both terms with weighted commutation fields also allows us to take advantage of the fact that the volume of $\chi_{S_t}$ in $\Sigma_t$ grows like $t$. This is an improvement over the volume of the spheres, which grow like $t^2$ in $3 + 1$ dimensions (see Section~\ref{sec:Geometry}).

We shall first examine the term $\sqrt{1+t} \chi_{S_t} |\gamma| |\gamma|$. The error term we must control is of the form
\begin{equation}
    \begin{aligned}
    \int_2^s \int_{\Sigma_t} \sqrt{1+t} \chi_{S_t} |\gamma| |\gamma| |\partial_t \Gamma^\alpha \gamma| d x d t.
    \end{aligned}
\end{equation}

We have that
\begin{equation}
    \begin{aligned}
    \int_2^s \int_{\Sigma_t} \sqrt{1+t} \chi_{S_t} |\gamma| |\gamma| |\partial_t \Gamma^\alpha \gamma| d x d t = \int_2^s \int_0^\infty \int_{S^2} \sqrt{1+t} \chi_{S_t} |\gamma| |\gamma| |\partial_t \Gamma^\alpha \gamma| r^2 d \omega d r d t
    \\ \le \Vert (1 + t)^{{1 \over q}} \chi_{S_t} \Vert_{L_t^\infty L_r^2 L_\omega^q} \Vert \chi_{S_t} r \gamma \Vert_{L_t^\infty L_r^\infty L_\omega^p} \Vert \chi_{S_t} r \gamma \Vert_{L_t^\infty L_r^\infty L_\omega^p}
    \\ \times \Vert (1 + t)^{-\delta} r \partial_t \Gamma^\alpha \gamma \Vert_{L_t^\infty L_r^2 L_\omega^2} \Vert (1 + t)^{-{1 \over 2} - {1 \over q} + \delta} \Vert_{L_t^1 L_r^\infty L_\omega^\infty},
    \end{aligned}
\end{equation}
where ${1 \over q} + {2 \over p} = {1 \over 2}$ and  we have used the fact that $t$ and $r$ are comparable in the support of $\chi_{S_t}$. Now, by our choice of $p$ (see \eqref{pdef}), we have that ${1 \over q} \ge {1 \over 2} - \delta$. Thus, we have that
\begin{equation}\label{firsttnormbound}
    \begin{aligned}
    \Vert (1 + t)^{-{1 \over 2} - {1 \over q} + \delta} \Vert_{L_t^1 L_r^\infty L_\omega^\infty} \le \int_2^s (1 + t)^{-1 + 2 \delta} d t \le C (1 + s)^{2 \delta}.
    \end{aligned}
\end{equation}
Thus, using Lemma~\ref{lem:VolumeEst} and the bootstrap assumptions~\eqref{Btstrp10},\eqref{Btstrp2}, we have that these terms of the error integral are controlled by 
\[
C \epsilon^{{9 \over 4}} (1 + s)^{2 \delta},
\]
as desired.

We now consider the term $\sqrt{1+t} \chi_{S_t} |\gamma| |\partial \gamma|$. The error term we must control is of the form
\begin{equation}
    \begin{aligned}
    \int_2^s \int_{\Sigma_t} \sqrt{1+t} \chi_{S_t} |\gamma| |\partial \gamma| |\partial_t \Gamma^\alpha \gamma| d x d t.
    \end{aligned}
\end{equation}

We have that
\begin{equation}
    \begin{aligned}
    \int_2^s \int_{\Sigma_t} \sqrt{1+t} \chi_{S_t} |\gamma| |\partial \gamma| |\partial_t \Gamma^\alpha \gamma| d x d t = \int_0^t \int_0^\infty \int_{S^2} \sqrt{1+t} \chi_{S_t} |\gamma| |\partial \gamma| |\partial_t \Gamma^\alpha \gamma| r^2 d \omega d r d t
    \\ \le C\Vert (1 + t)^{{1 \over q}} \chi_{S_t} \Vert_{L_t^\infty L_r^\infty L_\omega^q} \Vert \chi_{S_t} r \gamma \Vert_{L_t^\infty L_r^\infty L_\omega^p} \Vert \chi_{S_t}r \partial \gamma \Vert_{L_t^\infty L_r^2 L_\omega^p}
    \\ \times \Vert (1 + t)^{-\delta} r \partial_t \Gamma^\alpha \gamma \Vert_{L_t^\infty L_r^2 L_\omega^2} \Vert (1 + t)^{-{1 \over 2} - {1 \over q} + \delta} \Vert_{L_t^1 L_r^\infty L_\omega^\infty},
    \end{aligned}
\end{equation}
where ${1 \over q} + {2 \over p} = {1 \over 2}$ and we have used the fact that $t$ and $r$ are comparable in the support of $\chi_{S_t}$. Using Lemma~\ref{lem:VolumeEst}, the bootstrap assumptions~\eqref{Btstrp10},\eqref{Btstrp11},\eqref{Btstrp2}, and \eqref{firsttnormbound}, we have that these terms of the error integral are controlled by 
\[
C \epsilon^{{9 \over 4}} (1 + s)^{2 \delta},
\]
as desired.

We now consider the term $(1 + t) \chi_{S_t} |\partial \gamma| |\overline{\partial} \gamma|$. The error integral we must control is of the form
\begin{equation}
    \begin{aligned}
    \int_2^s \int_{\Sigma_t} (1 + t) \chi_{S_t} |\partial \gamma| |\overline{\partial} \gamma| |\partial_t \Gamma^\alpha \gamma| d x d t.
    \end{aligned}
\end{equation}
We have that
\begin{equation}
    \begin{aligned}
    \int_2^s \int_{\Sigma_t} (1 + t) \chi_{S_t} |\partial \gamma| |\overline{\partial} \gamma| |\partial_t \Gamma^\alpha \gamma| d x d t = \int_2^s \int_0^\infty \int_{S^2} (1 + t) \chi_{S_t} |\partial \gamma| |\overline{\partial} \gamma| |\partial_t \Gamma^\alpha \gamma| r^2 d \omega d r d t
    \\ \le C\Vert (1 + |u|)^{{1 \over 2} + {\delta \over 2}} (1 + t)^{{1 \over q}} \chi_{S_t} \Vert_{L_t^\infty L_r^\infty L_\omega^q} \Vert \chi_{S_t} r \partial \gamma \Vert_{L_t^\infty L_r^\infty L_\omega^p} \Vert (1 + |u|)^{-{1 \over 2} - {\delta \over 2}} r \overline{\partial} \gamma \Vert_{L_t^2 L_r^2 L_\omega^p}
    \\ \times \Vert (1 + t)^{-\delta} r \partial \Gamma^\alpha \gamma \Vert_{L_t^\infty L_r^2 L_\omega^2} \Vert (1 + t)^{-{1 \over q} + \delta} \Vert_{L_t^2 L_r^\infty, L_\omega^\infty}.
    \end{aligned}
\end{equation}
where ${1 \over q} + {2 \over p} = {1 \over 2}$ and we have used the fact that $t$ and $r$ are comparable in the support of $\chi_{S_t}$. Now, using our choice of $p$, we have that ${1 \over q} \ge {1 \over 2} - \delta$. Thus, we have that
\begin{equation}\label{secondtnormbound}
    \begin{aligned}
    \Vert (1 + t)^{-{1 \over q} + \delta} \Vert_{L_t^2 L_r^\infty L_\omega^\infty} \le \left (\int_2^s (1 + t)^{-1 + 4 \delta} d t \right )^{{1 \over 2}} \le C (1 + s)^{2 \delta}.
    \end{aligned}
\end{equation}
Thus, using Lemma~\ref{lem:VolumeEst} and the bootstrap assumptions~\eqref{Btstrp11},\eqref{Btstrp12},\eqref{Btstrp2}, we have that these terms of the error integral are controlled by 
\[
C \epsilon^{{9 \over 4}} (1 + s)^{2 \delta},
\]
as desired.

We finally consider the term $(1 + t) \chi_{S_t} |\gamma| |\overline{\partial} \gamma|$. The error integral we have to control is of the form
\begin{equation}
    \begin{aligned}
    \int_2^s \int_{\Sigma_t} (1 + t) \chi_{S_t} |\gamma| |\overline{\partial} \gamma| d x d t.
    \end{aligned}
\end{equation}

We have that
\begin{equation}
    \begin{aligned}
    \int_2^s \int_{\Sigma_t} (1 + t) \chi_{S_t} |\gamma| |\overline{\partial} \gamma| d x d t = \int_2^s \int_0^\infty \int_{S^2} (1 + t) \chi_{S_t} |\gamma| |\overline{\partial} \gamma| r^2 d \omega d r d t
    \\ \le C\Vert (1 + |u|)^{{1 \over 2} + {\delta \over 2}} (1 + t)^{{1 \over q}} \chi_{S_t} \Vert_{L_t^\infty L_r^\infty L_\omega^q} \Vert \chi_{S_t} r \gamma \Vert_{L_t^\infty L_r^\infty L_\omega^p} \Vert (1 + |u|)^{-{1 \over 2} - {\delta \over 2}} r \overline{\partial} \gamma \Vert_{L_t^2 L_r^2 L_\omega^p}
    \\ \times \Vert (1 + t)^{-\delta} r \partial \Gamma^\alpha \gamma \Vert_{L_t^\infty L_r^2 L_\omega^2} \Vert (1 + t)^{-{1 \over q} + \delta} \Vert_{L_t^2 L_r^\infty L_\omega^\infty},
    \end{aligned}
\end{equation}
where ${1 \over q} + {2 \over p} = {1 \over 2}$ and we have used the fact that $t$ and $r$ are comparable in the support of $\chi_{S_t}$. Using Lemma~\ref{lem:VolumeEst}, the bootstrap assumptions~\eqref{Btstrp10},\eqref{Btstrp12},\eqref{Btstrp2}, and \eqref{secondtnormbound}, we have that these terms of the error integral are controlled by 
\[
C \epsilon^{{9 \over 4}} (1 + s)^{2 \delta},
\]
as desired.

Using the bootstrap assumptions, we have shown that the error integrals arising from terms with $\sigma = 2$ when controlling $E_2 (T)$ are of size $C \epsilon^{{9 \over 4}} (1 + s)^{2 \delta}$. We have thus shown that the terms with $\sigma = 2$ are controlled in a way that is consistent with $E_2 (T) \le C \epsilon$. We now turn to controlling the remaining terms, which are those with $\sigma = 1$ and $\sigma = 0$. We shall in fact show that the error integrals for all of these terms are bounded by $C \epsilon^{{9 \over 4}}$. This will give us that $E_1 (T) \le C \epsilon$ and that $E_2 (T) \le C \epsilon$, as desired.

We now consider terms having $\sigma = 1$. These terms consist of
\begin{equation}
    \begin{aligned}
    \sum_{|\beta_1| + |\beta_2| \le 1} C \sqrt{1+t} \chi_{S_t}
    \\ \times \left [|\partial \Gamma^{\beta_1} \gamma| |\overline{\partial} \Gamma^{\beta_2} \gamma| + {1 \over \sqrt{1 + t}} |\Gamma^{\beta_1} \gamma| |\partial \Gamma^{\beta_2} \gamma| + |\Gamma^{\beta_1} \gamma| |\overline{\partial} \Gamma^{\beta_2} \gamma| + {1 \over \sqrt{1 + t}} |\Gamma^{\beta_1} \gamma| |\Gamma^{\beta_2} \gamma| \right ].
    \end{aligned}
\end{equation}
Since $|\beta_1|+|\beta_2|\le 2-\sigma=1$, one of $|\beta_1|,|\beta_2|$ must be 0. We can commute that factor with two weighted commutation fields and the other factor with one weighted commutation field. This will allow us to once again take advantage of the small volume of the support of $\chi_{S_t}$.

We first consider the term $\chi_{S_t} |\Gamma^{\beta_1} \gamma| |\Gamma^{\beta_2} \gamma|$. The error integral we must control is of the form
\begin{equation}
    \begin{aligned}
    \int_2^s \int_{\Sigma_t} \chi_{S_t} |\Gamma^{\beta_1} \gamma| |\Gamma^{\beta_2} \gamma| |\partial_t \Gamma^\alpha \gamma| d x d t.
    \end{aligned}
\end{equation}

We have that either $|\beta_1| = 0$ or $|\beta_2| = 0$. We shall consider the case that $|\beta_1| = 0$, as the other case follows in the same way. Note that it is possible that we also have $|\beta_2|=0$; this causes us to use one of two different bootstrap estimates to bound certain factors. We have that
\begin{equation}\label{sigmais1beta1is0}
    \begin{aligned}
    \int_2^s \int_{\Sigma_t} \chi_{S_t} |\Gamma^{\beta_1} \gamma| |\Gamma^{\beta_2} \gamma| |\partial_t \Gamma^\alpha \gamma| d x d t = \int_2^s \int_0^\infty \int_{S^2} \chi_{S_t} |\gamma| |\Gamma^{\beta_2} \gamma| |\partial_t \Gamma^\alpha \gamma| r^2 d \omega d r d t
    \\ \le \Vert (1 + t)^{{1 \over b}} \chi_{S_t} \Vert_{L_t^\infty L_r^2 L_\omega^{b}} \Vert (1 + t)^{1 - \delta} \chi_{S_t} \gamma \Vert_{L_t^\infty L_r^\infty L_\omega^\infty} \Vert \chi_{S_t}(1 + t)^{-\delta} r \Gamma^{\beta_2} \gamma \Vert_{L_t^\infty L_r^\infty L_\omega^p}
    \\ \times \Vert (1 + t)^{-\delta} r \partial \Gamma^\alpha \gamma \Vert_{L_t^\infty L_r^2 L_\omega^2} \Vert (1 + t)^{-1 - {1 \over b} + 3 \delta} \Vert_{L_t^1 L_r^\infty L_\omega^\infty},
    \end{aligned}
\end{equation}
where ${1 \over b} + {1 \over p} = {1 \over 2}$. Now, by our choice of $p$, we have that ${1 \over b} \ge {1 \over 2} - {\delta \over 2}$. Thus, we have that
\begin{equation}\label{thirdtnormbound}
    \begin{aligned}
    \Vert (1 + t)^{-1 - {1 \over b} + 3 \delta} \Vert_{L_t^1 L_r^\infty L_\omega^\infty} \le \int_2^s (1 + t)^{-{3 \over 2} + {7 \over 2} \delta} d t \le C.
    \end{aligned}
\end{equation}
We bound the first factor of~\eqref{sigmais1beta1is0} using Lemma~\ref{lem:VolumeEst}. We bound the second factor using the bootstrap assumption~\eqref{Btstrp16}, the third factor using bootstrap assumption \eqref{Btstrp13} or \eqref{Btstrp11}, the third factor using \eqref{Btstrp2}, and the fourth factor using \eqref{thirdtnormbound}. We then have that these terms of the error integral are controlled by 
\[
C \epsilon^{{9 \over 4}},
\]
as desired.

We now consider the term $\chi_{S_t} |\Gamma^{\beta_1} \gamma| |\partial \Gamma^{\beta_2} \gamma|$. The error integral we must control is of the form
\begin{equation}
    \begin{aligned}
    \int_2^s \int_{\Sigma_t} \chi_{S_t} |\Gamma^{\beta_1} \gamma| |\partial \Gamma^{\beta_2} \gamma| |\partial_t \Gamma^\alpha \gamma| d x d t.
    \end{aligned}
\end{equation}

We consider two cases depending on whether $|\beta_2| = 0$ or $|\beta_2| = 1$. We first consider the case where $|\beta_1| = 0$ and $|\beta_2|=1$. We have that
\begin{equation}
    \begin{aligned}
    \int_2^s \int_{\Sigma_t} \chi_{S_t} |\Gamma^{\beta_1}\gamma| |\partial \Gamma^{\beta_2} \gamma| |\partial_t \Gamma^\alpha \gamma| d x d t = \int_2^s \int_0^\infty \int_{S^2} \chi_{S_t} |\gamma| |\partial \Gamma^{\beta_2} \gamma| |\partial_t \Gamma^\alpha \gamma| r^2 d \omega d r d t
    \\ \le \Vert (1 + t)^{{1 \over b}} \chi_{S_t} \Vert_{L_t^\infty L_r^\infty L_\omega^b} \Vert (1 + t)^{1 - \delta} \gamma \Vert_{L_t^\infty L_r^\infty L_\omega^\infty} \Vert (1 + t)^{-\delta} r \partial \Gamma^{\beta_2} \gamma \Vert_{L_t^\infty L_r^2 L_\omega^p}
    \\ \times \Vert (1 + t)^{-\delta} r\partial \Gamma^{\alpha} \gamma \Vert_{L_t^\infty L_r^2 L_\omega^2} \Vert (1 + t)^{-1 - {1 \over b} + 3 \delta} \Vert_{L_t^1 L_r^\infty L_\omega^\infty},
    \end{aligned}
\end{equation}
where ${1 \over b} + {1 \over p} = {1 \over 2}$. Using Lemma~\ref{lem:VolumeEst}, the bootstrap assumptions~\eqref{Btstrp16},\eqref{Btstrp14},\eqref{Btstrp2}, and \eqref{thirdtnormbound}, we have that these terms of the error integral are controlled by 
\[
C \epsilon^{{9 \over 4}},
\]
as desired.

We now consider the case where $|\beta_2| = 0$. Note that it is possible that we also have $|\beta_1|=0$; this causes us to use one of two different bootstrap estimates to bound certain factors. We have that
\begin{equation}\label{sigmais1beta2is0}
    \begin{aligned}
    \int_2^s \int_{\Sigma_t} \chi_{S_t} |\Gamma^{\beta_1} \gamma| |\partial \Gamma^{\beta_2}\gamma| |\partial_t \Gamma^\alpha \gamma| d x d t = \int_2^s \int_0^r \int_{S^2} \chi_{S_t} |\Gamma^{\beta_1} \gamma| |\partial \gamma| |\partial_t \Gamma^\alpha \gamma| r^2 d \omega d r d t
    \\ \le \Vert (1 + t)^{{1 \over b}} \chi_{S_t} \Vert_{L_t^\infty L_r^2 L_\omega^b} \Vert (1 + t)^{-\delta} \chi_{S_t} r \Gamma^{\beta_1} \gamma \Vert_{L_t^\infty L_r^\infty L_\omega^p} \Vert (1 + t)^{1 - \delta} \partial \gamma \Vert_{L_t^\infty L_r^\infty L_\omega^\infty}
    \\ \times \Vert (1 + t)^{-\delta} r \partial \Gamma^{\alpha} \gamma \Vert_{L_t^\infty L_r^2 L_\omega^2} \Vert (1 + t)^{-1 - {1 \over b} + 3 \delta} \Vert_{L_t^1 L_r^\infty L_\omega^\infty},
    \end{aligned}
\end{equation}
where ${1 \over b} + {1 \over p} = {1 \over 2}$. We bound the first factor in \eqref{sigmais1beta2is0} by using Lemma~\ref{lem:VolumeEst}, bound the second factor using bootstrap assumption~\eqref{Btstrp13} or \eqref{Btstrp11}, the third factor using bootstrap assumption~\eqref{Btstrp3}, the fourth factor using \eqref{Btstrp2}, and the fifth factor using \eqref{thirdtnormbound}. All together, we get that these terms of the error integral are controlled by 
\[
C \epsilon^{{9 \over 4}},
\]
as desired.

We now consider the term $\sqrt{1+t}\chi_{S_t} |\overline{\partial}\Gamma^{\beta_1} \gamma| |\partial \Gamma^{\beta_2} \gamma|$. The error integral we must control is of the form
\begin{equation}
    \begin{aligned}
    \int_2^s \int_{\Sigma_t}\sqrt{1+t}\chi_{S_t} |\overline{\partial}\Gamma^{\beta_1} \gamma| |\partial \Gamma^{\beta_2} \gamma| |\partial_t \Gamma^\alpha \gamma| d x d t.
    \end{aligned}
\end{equation}
We consider two cases depending on whether $|\beta_2| = 0$ or $|\beta_2| = 1$. We first consider the case where $|\beta_1| = 0$ and $|\beta_2|=1$. We have that
\begin{equation}
    \begin{aligned}
    \int_2^s \int_{\Sigma_t}\sqrt{1+t}\chi_{S_t} |\overline{\partial}\Gamma^{\beta_1}\gamma| |\partial \Gamma^{\beta_2} \gamma| |\partial_t \Gamma^\alpha \gamma| d x d t = \int_2^s \int_0^\infty \int_{S^2} \sqrt{1+t}\chi_{S_t} |\overline{\partial}\gamma| |\partial \Gamma^{\beta_2} \gamma| |\partial_t \Gamma^\alpha \gamma| r^2 d \omega d r d t
    \\ \le \Vert (1 + t)^{{1 \over b}} \chi_{S_t} \Vert_{L_t^\infty L_r^\infty L_\omega^b} \Vert (1 + t)^{{3 \over 2} - \delta} \overline{\partial}\gamma \Vert_{L_t^\infty L_r^\infty L_\omega^\infty} \Vert (1 + t)^{-\delta} r \partial \Gamma^{\beta_2} \gamma \Vert_{L_t^\infty L_r^2 L_\omega^p}
    \\ \times \Vert (1 + t)^{-\delta} r\partial \Gamma^{\alpha} \gamma \Vert_{L_t^\infty L_r^2 L_\omega^2} \Vert (1 + t)^{-1 - {1 \over b} + 3 \delta} \Vert_{L_t^1 L_r^\infty L_\omega^\infty},
    \end{aligned}
\end{equation}
where ${1 \over b} + {1 \over p} = {1 \over 2}$. Using Lemma~\ref{lem:VolumeEst}, the bootstrap assumptions~\eqref{Btstrp4},\eqref{Btstrp14},\eqref{Btstrp2}, and \eqref{thirdtnormbound}, we have that these terms of the error integral are controlled by 
\[
C \epsilon^{{9 \over 4}},
\]
as desired.

We now consider the case where $|\beta_2| = 0$. Note that it is possible that we also have $|\beta_1|=0$; this causes us to use one of two different bootstrap estimates to bound certain factors. We have that
\begin{equation}\label{sigmais1beta2is0take2}
    \begin{aligned}
    \int_2^s \int_{\Sigma_t} \sqrt{1+t}\chi_{S_t} |\overline{\partial}\Gamma^{\beta_1} \gamma| |\partial \Gamma^{\beta_2}\gamma| |\partial_t \Gamma^\alpha \gamma| d x d t = \int_2^s \int_0^r \sqrt{1+t}\int_{S^2} \chi_{S_t} |\overline{\partial}\Gamma^{\beta_1} \gamma| |\partial \gamma| |\partial_t \Gamma^\alpha \gamma| r^2 d \omega d r d t
    \\ \le \Vert (1 + t)^{{1 \over b}} \chi_{S_t} \Vert_{L_t^\infty L_r^\infty L_\omega^b} \Vert (1 + t)^{-2\delta} \chi_{S_t} r \overline{\partial}\Gamma^{\beta_1} \gamma \Vert_{L_t^2 L_r^2 L_\omega^p} \Vert (1 + t)^{1 - \delta} \partial \gamma \Vert_{L_t^\infty L_r^\infty L_\omega^\infty}
    \\ \times \Vert (1 + t)^{-\delta} r \partial \Gamma^{\alpha} \gamma \Vert_{L_t^\infty L_r^2 L_\omega^2} \Vert (1 + t)^{-{1 \over 2} - {1 \over b} + 4 \delta} \Vert_{L_t^2 L_r^\infty L_\omega^\infty},
    \end{aligned}
\end{equation}
where ${1 \over b} + {1 \over p} = {1 \over 2}$. Now, by our choice of $p$, we have that ${1 \over b} \ge {1 \over 2} - {\delta \over 2}$. Thus, we have that
\begin{equation}\label{fourthtnormbound}
    \begin{aligned}
    \Vert (1 + t)^{-{1 \over 2} - {1 \over b} + 4 \delta} \Vert_{L_t^2 L_r^\infty L_\omega^\infty}\le\left(\int_0^\infty (1+t)^{-2+9\delta}dt\right)^{1/2}\le C.
    \end{aligned}
\end{equation}
We bound the first factor in \eqref{sigmais1beta2is0take2} by using Lemma~\ref{lem:VolumeEst}, bound the second factor using bootstrap assumption~\eqref{Btstrp15} or \eqref{Btstrp12}, the third factor using bootstrap assumption~\eqref{Btstrp3}, the fourth factor using \eqref{Btstrp2}, and the fifth factor using \eqref{fourthtnormbound}. All together, we get that these terms of the error integral are controlled by 
\[
C \epsilon^{{9 \over 4}},
\]
as desired.

We now consider the term $\sqrt{1+t}\chi_{S_t} |\overline{\partial}\Gamma^{\beta_1} \gamma| |\Gamma^{\beta_2} \gamma|$. The error integral we must control is of the form
\begin{equation}
    \begin{aligned}
    \int_2^s \int_{\Sigma_t}\sqrt{1+t}\chi_{S_t} |\overline{\partial}\Gamma^{\beta_1} \gamma| |\Gamma^{\beta_2} \gamma| |\partial_t \Gamma^\alpha \gamma| d x d t.
    \end{aligned}
\end{equation}
We consider two cases depending on whether $|\beta_2| = 0$ or $|\beta_2| = 1$. We first consider the case where $|\beta_1| = 0$ and $|\beta_2|=1$. We have that
\begin{equation}
    \begin{aligned}
    \int_2^s \int_{\Sigma_t}\sqrt{1+t}\chi_{S_t} |\overline{\partial}\Gamma^{\beta_1}\gamma| |\Gamma^{\beta_2} \gamma| |\partial_t \Gamma^\alpha \gamma| d x d t = \int_2^s \int_0^\infty \int_{S^2} \sqrt{1+t}\chi_{S_t} |\overline{\partial}\gamma| |\Gamma^{\beta_2} \gamma| |\partial_t \Gamma^\alpha \gamma| r^2 d \omega d r d t
    \\ \le \Vert (1 + t)^{{1 \over b}} \chi_{S_t} \Vert_{L_t^\infty L_r^2 L_\omega^b} \Vert (1 + t)^{{3 \over 2} - \delta} \overline{\partial}\gamma \Vert_{L_t^\infty L_r^\infty L_\omega^\infty} \Vert (1 + t)^{-\delta} \chi_{S_t}r \Gamma^{\beta_2} \gamma \Vert_{L_t^\infty L_r^\infty L_\omega^p}
    \\ \times \Vert (1 + t)^{-\delta} r\partial \Gamma^{\alpha} \gamma \Vert_{L_t^\infty L_r^2 L_\omega^2} \Vert (1 + t)^{-1 - {1 \over b} + 3 \delta} \Vert_{L_t^1 L_r^\infty L_\omega^\infty},
    \end{aligned}
\end{equation}
where ${1 \over b} + {1 \over p} = {1 \over 2}$. Using Lemma~\ref{lem:VolumeEst}, the bootstrap assumptions~\eqref{Btstrp4},\eqref{Btstrp13},\eqref{Btstrp2}, and \eqref{thirdtnormbound}, we have that these terms of the error integral are controlled by 
\[
C \epsilon^{{9 \over 4}},
\]
as desired.

We now consider the case where $|\beta_2| = 0$. Note that it is possible that we also have $|\beta_1|=0$; this causes us to use one of two different bootstrap estimates to bound certain factors. We have that
\begin{equation}\label{sigmais1beta2is0take3}
    \begin{aligned}
    \int_2^s \int_{\Sigma_t} \sqrt{1+t}\chi_{S_t} |\overline{\partial}\Gamma^{\beta_1} \gamma| |\Gamma^{\beta_2}\gamma| |\partial_t \Gamma^\alpha \gamma| d x d t = \int_2^s \int_0^r \sqrt{1+t}\int_{S^2} \chi_{S_t} |\overline{\partial}\Gamma^{\beta_1} \gamma| |\gamma| |\partial_t \Gamma^\alpha \gamma| r^2 d \omega d r d t
    \\ \le \Vert (1 + t)^{{1 \over b}} \chi_{S_t} \Vert_{L_t^\infty L_r^\infty L_\omega^b} \Vert (1 + t)^{-2\delta} \chi_{S_t} r \overline{\partial}\Gamma^{\beta_1} \gamma \Vert_{L_t^2 L_r^2 L_\omega^p} \Vert (1 + t)^{1 - \delta} \chi_{S_t}\gamma \Vert_{L_t^\infty L_r^\infty L_\omega^\infty}
    \\ \times \Vert (1 + t)^{-\delta} r \partial \Gamma^{\alpha} \gamma \Vert_{L_t^\infty L_r^2 L_\omega^2} \Vert (1 + t)^{-{1 \over 2} - {1 \over b} + 4 \delta} \Vert_{L_t^2 L_r^\infty L_\omega^\infty},
    \end{aligned}
\end{equation}
where ${1 \over b} + {1 \over p} = {1 \over 2}$. We bound the first factor in \eqref{sigmais1beta2is0take2} by using Lemma~\ref{lem:VolumeEst}, bound the second factor using bootstrap assumption~\eqref{Btstrp15} or \eqref{Btstrp12}, the third factor using bootstrap assumption~\eqref{Btstrp16}, the fourth factor using \eqref{Btstrp2}, and the fifth factor using \eqref{fourthtnormbound}. All together, we get that these terms of the error integral are controlled by 
\[
C \epsilon^{{9 \over 4}},
\]
as desired.

We now consider terms having $\sigma = 0$. These terms consist of
\begin{equation}
    \begin{aligned}
    \sum_{|\beta_1| + |\beta_2| \le 2} |\partial \Gamma^{\beta_1} \gamma| |\overline{\partial} \Gamma^{\beta_2} \gamma| + \chi_{S_t} \left [ {1 \over \sqrt{1 + t}} |\Gamma^{\beta_1} \gamma| |\partial \Gamma^{\beta_2} \gamma| + |\Gamma^{\beta_1} \gamma| |\overline{\partial} \Gamma^{\beta_2} \gamma| + {1 \over (1 + t)^{{1 \over 2}}} |\Gamma^{\beta_1} \gamma| |\Gamma^{\beta_2} \gamma| \right ].
    \end{aligned}
\end{equation}
Note that we will be allowed to assume that $|\beta_1|+|\beta_2|=2$. For the terms where $|\beta_1|+|\beta_2|\le 1$, we note that they are simply a factor of $\sqrt{1+t}$ smaller than one of the terms with $\sigma=0$, so the bounds we obtained there suffice.

We first consider the term $|\partial \Gamma^{\beta_1} \gamma| |\overline{\partial} \Gamma^{\beta_2} \gamma|$. The error integral we must control is of the form
\begin{equation}
    \begin{aligned}
    \int_2^s \int_{\Sigma_t} |\partial \Gamma^{\beta_1} \gamma| |\overline{\partial} \Gamma^{\beta_2} \gamma| |\partial_t \Gamma^\alpha \gamma| d x d t.
    \end{aligned}
\end{equation}
We consider $3$ cases depending on the distribution of the weighted commutation fields in $\Gamma^{\beta_1}$ and $\Gamma^{\beta_2}$.

We first consider the case where $\Gamma^{\beta_1}$ contains all of the weighted commutation fields. This corresponds to $|\beta_2| = 0$. We have that
\begin{equation}
    \begin{aligned}
    \int_2^s \int_{\Sigma_t} |\partial \Gamma^{\beta_1} \gamma| |\overline{\partial} \gamma| |\partial_t \Gamma^\alpha \gamma| d x d t
    \\ \le  \Vert (1 + t)^{{3 \over 2} - \delta} \overline{\partial} \gamma \Vert_{L_t^\infty L_x^\infty} \Vert (1 + t)^{-\delta} \partial \Gamma^{\beta_1} \gamma \Vert_{L_t^\infty L_x^2} \Vert (1 + t)^{-\delta} \partial \Gamma^\alpha \gamma \Vert_{L_t^\infty L_x^2} \Vert (1 + t)^{-{3 \over 2} + 3 \delta} \Vert_{L_t^1 L_x^\infty}.
    \end{aligned}
\end{equation}
Now, using the bootstrap assumptions~\eqref{Btstrp4} and \eqref{Btstrp2}, we have that these terms of the error integral are controlled by 
\[
C \epsilon^{{9 \over 4}},
\]
as desired.

We now consider the case where $\Gamma^{\beta_2}$ contains all of the weighted commutation fields. This corresponds to $|\beta_1| = 0$. We use the fact that $1+|u|\le C(1+t)$ in the region under consideration to get that
\begin{equation}
    \begin{aligned}
    \int_2^s \int_{\Sigma_t} |\partial \gamma| |\overline{\partial} \Gamma^{\beta_2} \gamma| |\partial_t \Gamma^\alpha \gamma| d x d t
    \\ \le C\Vert (1 + t)^{1 - \delta} (1 + |u|)^{{1 \over 2}} \partial \gamma \Vert_{L_t^\infty L_x^\infty} \Vert (1 + t)^{-2 \delta} (1 + |u|)^{-{1 \over 2} - {\delta \over 2}} \overline{\partial} \Gamma^{\beta_2} \gamma \Vert_{L_t^2 L_x^2}
    \\ \times \Vert (1 + t)^{-\delta} \partial \Gamma^\alpha \gamma \Vert_{L_t^\infty L_x^2} \Vert (1 + t)^{-1 + {9 \delta \over 2}} \Vert_{L_t^2 L_x^\infty}.
    \end{aligned}
\end{equation}
Now, using the bootstrap assumptions~\eqref{Btstrp3},\eqref{Btstrp7},\eqref{Btstrp2}, we have that these terms of the error integral are controlled by 
\[
C \epsilon^{{9 \over 4}},
\]
as desired.

Now, we shall consider the case where $\Gamma^{\beta_1}$ and $\Gamma^{\beta_2}$ each contain one commutation field. We use the fact that $1+|u|\le C(1+t)$ in the region under consideration to get that
\begin{equation}
    \begin{aligned}
    \int_2^s \int_{\Sigma_t} |\partial \Gamma^{\beta_1} \gamma| |\overline{\partial} \Gamma^{\beta_2} \gamma| |\partial_t \Gamma^\alpha \gamma| d x d t
    \\ \le C\Vert (1 + t + r)^{{1 \over 2} - {3 \delta \over 4}} (1 + |u|)^{{1 \over 4}} \partial \Gamma^{\beta_1} \gamma \Vert_{L_t^\infty L_x^4} \Vert (1 + |u|)^{-{1 \over 4} - {\delta \over 2}} (1 + t + r)^{{1 \over 2} - {3 \delta \over 2}} \overline{\partial} \Gamma^{\beta_2} \gamma \gamma \Vert_{L_t^2 L_x^4}
    \\ \times \Vert (1 + t)^{-\delta} \partial \Gamma^\alpha \gamma \Vert_{L_t^\infty L_x^2} \Vert (1 + t)^{-1 + {15 \delta \over 4}} \Vert_{L_t^2 L_x^\infty}.
    \end{aligned}
\end{equation}
Now, using the bootstrap assumptions~\eqref{Btstrp8},\eqref{Btstrp9},\eqref{Btstrp2}, we have that these terms of the error integral are controlled by 
\[
C \epsilon^{{9 \over 4}},
\]
as desired.

We shall now consider the term $\chi_{S_t} {1 \over \sqrt{1 + t}} |\Gamma^{\beta_1} \gamma| |\Gamma^{\beta_2} \gamma|$. The integral we must control is of the form
\begin{equation}
    \begin{aligned}
    \int_2^s \int_{\Sigma_t} \chi_{S_t} {1 \over \sqrt{1 + t}} |\Gamma^{\beta_1} \gamma| |\Gamma^{\beta_2} \gamma| |\partial_t \Gamma^\alpha \gamma| d x d t.
    \end{aligned}
\end{equation}
We consider $2$ main cases depending on the distribution of commutation fields in $\Gamma^{\beta_1}$ and $\Gamma^{\beta_2}$. We note that we may always put an additional unit derivative on both $\Gamma^{\beta_1} \gamma$ and $\Gamma^{\beta_2} \gamma$.

We first consider the case where either $\Gamma^{\beta_1}$ or $\Gamma^{\beta_2}$ contains both commutation fields. We consider the case where $\Gamma^{\beta_1}$ has both commutation fields, as the case where $\Gamma^{\beta_2}$ has both commutation fields follows in the same way.

We have that
\begin{equation}
    \begin{aligned}
    \int_2^s \int_{\Sigma_t} \chi_{S_t} {1 \over \sqrt{1 + t}} |\Gamma^{\beta_1} \gamma| |\gamma| |\partial_t \Gamma^\alpha \gamma| d x d t
    \\ \le \Vert (1 + t)^{-\delta} \chi_{S_t} \Gamma^{\beta_1} \gamma \Vert_{L_t^\infty L_x^2} \Vert (1 + t + r)^{1 - \delta} \chi_{S_t} \gamma \Vert_{L_t^\infty L_x^\infty}
    \\ \times \Vert (1 + t)^{-\delta} \partial \Gamma^\alpha \gamma \Vert_{L_t^\infty L_x^2} \Vert (1 + t)^{-{3 \over 2} + 3 \delta} \Vert_{L_t^1 L_x^\infty}.
    \end{aligned}
\end{equation}
Now, using the bootstrap assumptions~\eqref{Btstrp17},\eqref{Btstrp16},\eqref{Btstrp2}, we have that these terms of the error integral are controlled by 
\[
C \epsilon^{{9 \over 4}},
\]
as desired.

We now consider the case where $\Gamma^{\beta_1}$ and $\Gamma^{\beta_2}$ each contain one commutation field. We have that
\begin{equation}
    \begin{aligned}
    \int_2^s \int_{\Sigma_t} \chi_{S_t} {1 \over (1 + t)^{{1 \over 2}}} |\Gamma^{\beta_1} \gamma| |\Gamma^{\beta_2} \gamma| |\partial_t \Gamma^\alpha \gamma| d x d t = \int_2^s \int_0^\infty \int_{S^2} \chi_{S_t} {1 \over (1 + t)^{{1 \over 2}}} |\Gamma^{\beta_1} \gamma| |\Gamma^{\beta_2} \gamma| |\partial_t \Gamma^\alpha \gamma| r^2 d \omega d r d t
    \\ \le C\Vert (1 + t)^{{1 \over q}} \chi_{S_t} \Vert_{L_t^\infty L_r^2 L_\omega^q} \Vert (1 + t)^{-\delta} r \chi_{S_t}\Gamma^{\beta_1} \gamma \Vert_{L_t^\infty L_r^\infty L_\omega^p} \Vert (1 + t)^{-\delta} r \chi_{S_t}\Gamma^{\beta_2} \gamma \Vert_{L_t^\infty L_r^\infty L_\omega^p}
    \\ \times \Vert (1 + t)^{-\delta} r \partial \Gamma^\alpha \gamma \Vert_{L_t^\infty L_r^2 L_\omega^2} \Vert (1 + t)^{-{3 \over 2} - {1 \over q} + 3 \delta} \Vert_{L_t^1 L_r^\infty L_\omega^\infty},
    \end{aligned}
\end{equation}
where here ${1 \over q} + {2 \over p} = {1 \over 2}$ and we have used the fact that $t$ and $r$ are comparable in the support of $\chi_{S_t}$. Using Lemma~\ref{lem:VolumeEst} and the bootstrap assumptions~\eqref{Btstrp13},\eqref{Btstrp2}, we have that these terms of the error integral are controlled by 
\[
C \epsilon^{{9 \over 4}},
\]
as desired.

We now consider the term $\chi_{S_t} {1 \over \sqrt{1 + t}} |\Gamma^{\beta_1} \gamma| |\partial \Gamma^{\beta_2} \gamma|$. The error integral we must control is of the form
\begin{equation}
    \begin{aligned}
    \int_2^s \int_{\Sigma_t} \chi_{S_t} {1 \over \sqrt{1 + t}} |\Gamma^{\beta_1} \gamma| |\partial \Gamma^{\beta_2} \gamma| |\partial_t \Gamma^\alpha \gamma| d x d t.
    \end{aligned}
\end{equation}

We now consider $3$ cases depending on the distribution of commutation fields in $\Gamma^{\beta_1}$ and $\Gamma^{\beta_2}$.

We first consider the case where $\Gamma^{\beta_1}$ contains both commutation fields. Then, we have that $|\beta_2| = 0$. We have that

\begin{equation}
    \begin{aligned}
    \int_2^s \int_{\Sigma_t} \chi_{S_t} {1 \over \sqrt{1 + t}} |\Gamma^{\beta_1} \gamma| |\partial \gamma| |\partial_t \Gamma^\alpha \gamma| d x d t
    \\ \le \Vert (1 + t)^{-\delta} \chi_{S_t} \Gamma^{\beta_1} \gamma \Vert_{L_t^\infty L_x^2} \Vert (1 + t)^{1 - \delta} \partial \gamma \Vert_{L_t^\infty L_x^\infty} \Vert (1 + t)^{-\delta} \partial \Gamma^\alpha \gamma \Vert_{L_t^\infty L_x^2} \Vert (1 + t)^{-{3 \over 2} + 3 \delta} \Vert_{L_t^1 L_x^\infty}.
    \end{aligned}
\end{equation}
Using the bootstrap assumptions~\eqref{Btstrp17},\eqref{Btstrp3},\eqref{Btstrp2}, we have that these terms of the error integral are controlled by 
\[
C \epsilon^{{9 \over 4}},
\]
as desired.

We now consider the case where $\Gamma^{\beta_2}$ contains both commutation fields, meaning that $|\beta_1| = 0$. We have that

\begin{equation}
    \begin{aligned}
    \int_2^s \int_{\Sigma_t} \chi_{S_t} {1 \over \sqrt{1 + t}} | \gamma| |\partial \Gamma^{\beta_2}\gamma| |\partial_t \Gamma^\alpha \gamma| d x d t
    \\ \le \Vert (1 + t+r)^{1-\delta}\chi_{S_t} \gamma \Vert_{L_t^\infty L_x^\infty}\Vert (1 + t)^{ - \delta} \partial\Gamma^{\beta_2} \gamma \Vert_{L_t^\infty L_x^2}  \Vert (1 + t)^{-\delta} \partial \Gamma^\alpha \gamma \Vert_{L_t^\infty L_x^2} \Vert (1 + t)^{-{3 \over 2} + 3 \delta} \Vert_{L_t^1 L_x^\infty}.
    \end{aligned}
\end{equation}
Using the bootstrap assumptions~\eqref{Btstrp16},\eqref{Btstrp2}, we have that these terms of the error integral are controlled by 
\[
C \epsilon^{{9 \over 4}},
\]
as desired.

We finally consider the case where both $\Gamma^{\beta_1}$ and $\Gamma^{\beta_2}$ contain one commutation field. We have that

\begin{equation}
    \begin{aligned}
    \int_2^s \int_{\Sigma_t} \chi_{S_t} {1 \over \sqrt{1 + t}} |\Gamma^{\beta_1} \gamma| |\partial \Gamma^{\beta_2} \gamma| |\partial_t \Gamma^\alpha \gamma| d x d t = \int_2^s \int_0^\infty \int_{S^2} {1 \over \sqrt{1 + t}} \chi_{S_t} |\Gamma^{\beta_1} \gamma| |\partial \Gamma^{\beta_2} \gamma| |\partial_t \Gamma^\alpha \gamma| r^2 d \omega d r d t
    \\ \le \Vert (1 + t)^{{1 \over q}} \chi_{S_t} \Vert_{L_t^\infty L_r^\infty L_\omega^q} \Vert (1 + t)^{-\delta} \chi_{S_t} r \Gamma^{\beta_1} \gamma \Vert_{L_t^\infty L_r^\infty L_\omega^p} \Vert (1 + t)^{-\delta} r \partial \Gamma^{\beta_2} \gamma \Vert_{L_t^\infty L_r^2 L_\omega^p}
    \\ \times \Vert (1 + t)^{-\delta} r\partial \Gamma^\alpha \gamma \Vert_{L_t^\infty L_r^2 L_\omega^2} \Vert (1 + t)^{-{3 \over 2} - {1 \over q} + 3 \delta} \Vert_{L_t^1 L_r^\infty L_\omega^\infty},
    \end{aligned}
\end{equation}
where ${1 \over q} + {2 \over p} = {1 \over 2}$ and we have used the fact that $t$ and $r$ are comparable in the support of $\chi_{S_t}$. Now, by our choice of $p$, we have that ${1 \over q} \ge {1 \over 2} - \delta$, meaning that we have that
\begin{equation}
    \begin{aligned}
    \Vert (1 + t)^{-{3 \over 2} - {1 \over q} + 4 \delta} \Vert_{L_t^1 L_r^\infty L_\omega^\infty} \le \int_2^s (1 + t)^{-2 + 4 \delta} d t \le C.
    \end{aligned}
\end{equation}
Then, using Lemma~\ref{lem:VolumeEst} and the bootstrap assumptions~\eqref{Btstrp13},\eqref{Btstrp2}, we have that these terms of the error integral are controlled by 
\[
C \epsilon^{{9 \over 4}},
\]
as desired.

We finally consider the term $\chi_{S_t} |\Gamma^{\beta_1} \gamma| |\overline{\partial} \Gamma^{\beta_2} \gamma|$. The error integral we must control is of the form
\begin{equation}
    \begin{aligned}
    \int_2^s \int_{\Sigma_t} \chi_{S_t} |\Gamma^{\beta_1} \gamma| |\overline{\partial} \Gamma^{\beta_2} \gamma| |\partial_t \Gamma^\alpha \gamma| d x d t.
    \end{aligned}
\end{equation}

We consider $3$ cases depending on the distribution of commutation fields in $\Gamma^{\beta_1}$ and $\Gamma^{\beta_2}$.

We first consider the case where $\Gamma^{\beta_1}$ contains both commutation fields, meaning that $|\beta_2| = 0$. We have that
\begin{equation}
    \begin{aligned}
    \int_2^s \int_{\Sigma_t} \chi_{S_t} |\Gamma^{\beta_1} \gamma| |\overline{\partial} \gamma| |\partial_t \Gamma^\alpha \gamma| d x d t
    \\ \le \Vert (1 + t)^{-\delta} \chi_{S_t} \Gamma^{\beta_1} \gamma \Vert_{L_t^\infty L_x^2} \Vert (1 + t)^{{3 \over 2} - \delta} \overline{\partial} \gamma \Vert_{L_t^\infty L_x^\infty}
    \\ \times \Vert (1 + t)^{-\delta} \partial \Gamma^\alpha \gamma \Vert_{L_t^\infty L_x^2} \Vert (1 + t)^{-{3 \over 2} + 3 \delta} \Vert_{L_t^1 L_x^\infty}.
    \end{aligned}
\end{equation}
Using the bootstrap assumptions~\eqref{Btstrp17},\eqref{Btstrp4},\eqref{Btstrp2}, we have that these terms of the error integral are controlled by 
\[
C \epsilon^{{9 \over 4}},
\]
as desired.

We now consider the case where $\Gamma^{\beta_2}$ contains both commutation fields, meaning that $|\beta_1| = 0$. We have that

\begin{equation}
    \begin{aligned}
    \int_2^s \int_{\Sigma_t} \chi_{S_t} |\gamma| |\overline{\partial} \Gamma^{\beta_1} \gamma| |\partial_t \Gamma^\alpha \gamma| d x d t
    \\ \le \Vert (1 + t)^{1 - \delta} (1 + |u|)^{{1 \over 2}} \chi_{S_t} \gamma \Vert_{L_t^\infty L_x^\infty} \Vert (1 + t)^{-2 \delta} (1 + |u|)^{-{1 \over 2} - {\delta \over 2}} \overline{\partial} \Gamma^{\beta_2} \gamma \Vert_{L_t^2 L_x^2}
    \\ \times \Vert (1 + t)^{-\delta} \partial \Gamma^\alpha \gamma \Vert_{L_t^\infty L_x^2} \Vert (1 + t)^{-1 + {9 \delta \over 2}} \Vert_{L_t^2 L_x^\infty},
    \end{aligned}
\end{equation}
where we have used the fact that $|u| \le C t$.

Using the bootstrap assumptions~\eqref{Btstrp3},\eqref{Btstrp7},\eqref{Btstrp2}, we have that these terms of the error integral are controlled by 
\[
C \epsilon^{{9 \over 4}},
\]
as desired.

Finally, we consider the case where $\Gamma^{\beta_1}$ and $\Gamma^{\beta_2}$ each contain one commutation field. We have that
\begin{equation}
    \begin{aligned}
    \int_2^s \int_{\Sigma_s} \chi_{S_t} |\Gamma^{\beta_1} \gamma| |\overline{\partial} \Gamma^{\beta_2} \gamma| |\partial_t \Gamma^\alpha \gamma| d x d t = \int_2^s \int_0^\infty \int_{S^2} \chi_{S_t} |\Gamma^{\beta_1} \gamma| |\overline{\partial} \Gamma^{\beta_2} \gamma| |\partial_t \Gamma^\alpha \gamma| r^2 d \omega d r d t
    \\ \le \Vert (1 + t)^{{1 \over q}} \chi_{S_t} \Vert_{L_t^\infty L_r^\infty L_\omega^q} \Vert (1 + t)^{-\delta} \chi_{S_t}r \Gamma^{\beta_1} \gamma \Vert_{L_t^\infty L_r^\infty L_\omega^p}
    \\ \times \Vert (1 + t)^{-2\delta} \chi_{S_t}r \overline{\partial} \Gamma^{\beta_2} \gamma \Vert_{L_t^2 L_r^2 L_\omega^p} \Vert (1 + t)^{-\delta} r \partial \Gamma^\alpha \gamma \Vert_{L_t^\infty L_r^2 L_\omega^2} \Vert (1 + t)^{-1 - {1 \over q} + 4 \delta} \Vert_{L_t^2 L_r^\infty L_\omega^\infty},
    \end{aligned}
\end{equation}
where ${1 \over q} + {2 \over p} = {1 \over 2}$ and we have used the fact that $t$ and $r$ are comparable in the support of $\chi_{S_t}$. Then, using Lemma~\ref{lem:VolumeEst} and the bootstrap assumptions~\eqref{Btstrp13},\eqref{Btstrp15},\eqref{Btstrp2}, we have that these terms of the error integral are controlled by 
\[
C \epsilon^{{9 \over 4}},
\]
as desired.

Since we have gone through all the terms, we get $\eqref{F1error}$ and $\eqref{F2error}$, as desired.

\section{Instability}\label{sec:unstable}
In this section, we assume that the linearized equation is of the form
\begin{equation}\label{eq:exprootequation}
\square\eta=B_y(u')\partial_y\eta+B_z(u')\partial_z\eta+F
\end{equation}
where $B_y,B_z:[-1,1]\to\R^{k\times k}$. The main result of this section is the proof of Theorem \ref{thm:generalinstab}. We first discuss the genericity assumptions in the theorem statement.

One may desire a statement of the form ``whenever Condition~\ref{cond:unexciting} is not satisfied, we have that for a generic traveling wave profile, applying the transformation from Section~\ref{sec:transformation} and then linearizing gives an equation of the form \eqref{eq:exprootnoforcing} for which the conditions of Theorem~\ref{thm:generalinstab} are satisfied". However, that is false. In fact given any system of equations of the form \eqref{eq:zerothversion} where $m_{ij\ell}$ is antisymmetric in all three indices and has only $\partial_y\phi_j\partial_{u'}\phi_\ell-\partial_{u'}\phi_j\partial_{y}\phi_\ell$ terms, we have that $B_y$ and $B_z$ end up being antisymmetric, so any linear combination has purely imaginary eigenvalues, thus not satisfying Condition~\ref{cond:M}. More generally, we can take $m_{ij\ell}$ to have terms antisymmetric in all three indices, plus standard null form terms. Then after applying the transformation from Section~\ref{sec:transformation}, we get that $B_y$ is similar to an antisymmetric matrix, so it still has imaginary eigenvalues.

What is true is that the conditions of Theorem~\ref{thm:generalinstab} are satisfied generically when we vary both the wave profile and the coefficients $m_{ij\ell}$ on the right-hand side of  \eqref{eq:zerothversion}. To see this, we take an arbitrary nonzero traveling wave profile $f$ and an arbitrary point $u_0$ where $f'(u_0)\ne 0$. By changing coordinates in $\R^k$, we can assume that $f_1'(u_0)\ne 0$ and $f_i'(u_0)=0$ for $i\ne 1$. Recall that $\gamma$ satisfies \eqref{eq:transformed}. Then varying the coefficients in $m_{ij\ell}$ (specifically the coefficients in front of $\partial_y\phi_i\partial_{u'}\phi_1$; we want to keep the other coefficients constant in order to not modify the matrix $A$) allows us to vary $B_y$ arbitrarily (since the matrix $A$ is invertible), so all we have to do is note that Condition~\ref{cond:M} is a generic condition on matrices.

Finally, based on naive dimension-counting, we believe the following is true:
\begin{conjecture}
For generic coefficients in the right-hand side of \eqref{eq:zerothversion}, all nonzero traveling wave profiles will satisfy the conditions of Theorem~\ref{thm:generalinstab}.
\end{conjecture}

We will now prove an instability and blowup result for a specific equation of the form \eqref{eq:exprootequation}, where we can get more exact asymptotics and develop some intuition.
\begin{proposition}\label{prop:Bessel}
If $\eta_1$ is a solution to
\begin{equation}\label{eq:specialcaselinear}
\square\eta_1=B(u')\partial_y\eta_1
\end{equation}
then there is initial data supported on the unit ball that grows faster than $\exp((K-\epsilon)\sqrt{t})$ for any $\epsilon>0$ where
\[
K=\sup_{u_0\le u_1}\frac{1}{\sqrt{2}\sqrt{u_1-u_0}}\int_{u_0}^{u_1}B(\alpha)d\alpha.
\]
If $\eta_1$ is a solution to
\begin{equation}\label{eq:specialcasenonlinear}
\square\eta_1=B(u')\partial_y\eta_1+\partial_a\eta_1\partial^a\eta_1
\end{equation}
then, for any $k\in \N$ and any $\epsilon>0$, there is some $C>0$ so that one can find an open ball of initial data  whose $H^k$ norm is $\sim\delta$, and which blows up within time $(\log(C\delta)/(K-\epsilon))^2$ for arbitrarily small $\delta>0$.
\end{proposition}
\begin{proof}
We will work in the region $\{-1\le u'\le 1, 1\le v'\}$, we will have the initial data along the $\{u'=-1\}$ characteristic surface be 0, and we will specify the initial data along the $\{v'=1\}$ characteristic surface.

Due to the region we care about and domain of dependence, once we have decided on the initial data we want on $\{-1\le u'\le 1, v'=1\}$, we can run the equation backward in time on a finite domain to get appropriate initial data on $\{t=0\}$. Thus converting between the $\{u'=-1\}$ characteristic surface and the initial data on $\{t=0\}$ will introduce at most constant factors.

We take the Fourier transform (appropriately normalized) of \eqref{eq:specialcaselinear} in $y,z$ to get that $q=\hat\eta_1(\xi_y,\xi_z,u',v')$ satisfies
\[
4\partial_{u'}\partial_{v'}q=-(|\xi^2|+iB(u')\xi_y)q
\]
We will find the fundamental solution for every $\xi$, that is solve
\[
4\partial_{u'}\partial_{v'}q=-(|\xi|^2+iB(u')\xi_y)q+ 4 \delta(v')\delta(u'-u_0).
\]
The $4$ in front of $\delta(v')\delta(u'-u_0)$ is chosen just to make the constants work out more nicely. Getting bounds on this equation will translate immediately to bounds on the actual fundamental solution.

We will solve this equation by iteration, working in the region $\{u'\ge u_0\}$. Let $a(u')=-\frac{|\xi^2|+iB(u')\xi_y}{4}$. The first iteration yields $q_0=1$ (since we are restricting ourselves to the region $\{u'\ge u_0\}$. In general, we will get
\[
q_k(u',v')=1+\int_{u_0}^{u'}\int_{1}^{v'}  a(\alpha)q_{k-1}(\alpha,\beta)d\beta d\alpha
\]
which, when expanded, gives
\begin{align*}
q_k&=\sum_{j=0}^k \frac{(v'-1)^k}{k!}\int_{u_0\le u_1\le\cdots\le u_k}a(u_1)\cdots a(u_k)du_1\cdots du_k\\
&=\sum_{j=0}^k \frac{1}{(k!)^2}\left((v'-1)\int_{u_0}^{u'} a(\alpha)d\alpha\right)^k\xrightarrow[k\to\infty]{}I_0\left(2\sqrt{(v'-1)\int_{u_0}^{u'} a(\alpha)d\alpha}\right)
\end{align*}
where $I_0$ is the modified Bessel function of the first kind. From this we can hypothetically extract whatever information we want; we will just note that for $|\xi|$ fairly large, we have that
\[
2\sqrt{(v'-1)\int_{u_0}^{u'} a(\alpha)d\alpha}=\sqrt{v'-1}\sqrt{-(u'-u_0)|\xi|^2-i\xi_y\int_{u_0}^{u'} B(\alpha)d\alpha}
\]
and that when we take $\xi_y=|\xi|$, then
\[
2\sqrt{(v'-1)\int_{u_0}^{u'} a(\alpha)d\alpha}=\pm\sqrt{v'-1}\left(-i\sqrt{u'-u_0}|\xi|+\frac{1}{2\sqrt{u'-u_0}}\int_{u_0}^{u'} B(\alpha)d\alpha+O(1/|\xi|)\right).
\]
Since $I_0(z)\sim\frac{\exp(z)}{\sqrt{2\pi z}}(1+O(1/z))$ (see, e.g. \cite{WatsonBessel}), we have that (when $|\xi|>v'/\epsilon, v'>2$)
\begin{equation}
    \begin{aligned}
    |q(u',v')| &=\lim_{k\to\infty}|q_k(u',v')|=\left|I_0\left(2\sqrt{(v'-1)\int_{u_0}^{u'} a(\alpha)d\alpha}\right)\right|
    \\ &=\frac{1}{\sqrt{2\pi z}}\exp\left(\frac{\sqrt{v'-1}}{2\sqrt{u'-u_0}}\int_{u_0}^{u'} B(\alpha)d\alpha\right)(1+O(\sqrt{v'}/|\xi|)).
    \end{aligned}
\end{equation}
This gives the asymptotics in the theorem statement (keeping in mind that $v=2t+O(1)$ in the relevant region).

To have initial data small in $H^k$ that still has the desired growth rate at time $T$ (up to the an $\epsilon$), we choose $N$ sufficiently large and take the initial data for $\{v'=0\}$ to be supported in $|\xi|$ on a ball of radius $T^{-N^2}$ centered at $(T^N, 0)$ and in $u'$ on a strip of with $T^{-N}$ centered at the optimal $u_0$ (actually, this is what we want the $\partial_{u'}$ derivative of the initial data to satisfy. We can bring the value of $\eta_1$ back down to 0 once we get past the desired $u'$). The initial size will have to be polynomially small in $T$ since we have some $H^k$ norm, and we will have a loss of $\epsilon$ due to that and due to spreading from the optimal $u_0$.

To get the blowup result, we need to find $\eta_1$ satisfying $\eqref{eq:specialcasenonlinear}$. We use the Nirenberg trick, letting $\phi=\exp(-\eta_1)-1$ and noting that then $\phi$ satisfies the linear equation \eqref{eq:specialcaselinear} precisely when $\eta_1$ satisfies \eqref{eq:specialcasenonlinear}. The linear instability result allows us to get $\phi=-1$ at some point before the desired time, at which point $\eta_1$ blows up.
\end{proof}
We will also prove the generic linear instability result given in Theorem~\ref{thm:generalinstab}. First, we prove an upper bound on the growth rate:
\begin{lemma}\label{lem:exproottmultiplier}
There exists some some $K_0>0$ depending on $B$ so that for any $K>K_0$, there is some $C>0$  such that whenever
$\eta$ is a solution to $\eqref{eq:exprootequation}$ that is supported on a unit ball at time 0 and that satisfies
\begin{align}
\|\eta(0,\cdot)\|_{H^1}&<\delta\\
\|\partial_t\eta(0,\cdot)\|_{L^2}&<\delta\\
\|F(t,\cdot)\|_{L^2}&<\delta\exp(K\sqrt{t})\label{cond:Fbound},
\end{align}
then
\begin{equation}\label{bound:fastbootstrap}
||D\eta||_{L^2}\le C \delta\exp(10 K \sqrt{t})
\end{equation}
%and
%\[
%||\partial_v\eta||_{L^2(\{x\ge t-1\})}\lesssim \frac{1}{\sqrt{t}}\exp(C\sqrt{t})
%\]
\end{lemma}
\begin{proof}
For now, we suppose that $F=0$ and obtain the correct estimates for the homogeneous equation. Later, we will deal with the inhomogeneity by using Duhamel.

We will prove the bound at time $T$. We use the multiplier
\begin{equation}\label{expr:multipinstability}
\exp\left(-g(u',v')\right)\partial_{v'}\eta
\end{equation}
where $\partial_{v'}g\ge 0,\partial_{u'}g\ge 0$, and \begin{equation}\label{cond:gcondition}
4\partial_{v'}g\partial_{u'}g\ge |B_y|^2+|B_z|^2.
\end{equation}
where $|B_y|$ denotes the operator norm of the matrix $B_y$ and similarly for $B_z$. 
This is, for instance, satisfied by $g$ of the form 
\begin{equation}\label{eq:gexample}
g=\sqrt{Q(u')}\sqrt{v'+1}
\end{equation}
where $Q'=|B_y|^2+|B_z|^2$.

Multiplying by \eqref{expr:multipinstability} and integrating over the strip $W=\{-1\le u'\le 1, 0\le t\le T\}$, we get
\begin{align*}
0&=\iint_W e^{-g}\left(4\partial_{v'}\eta\cdot\partial_{u'}\partial_{v'}\eta-\partial_{v'}\eta\cdot\partial_{y}\partial_{y}\eta-\partial_{v'}\eta\cdot\partial_{z}\partial_{z}\eta-\partial_{v'}\eta\cdot B_y(u')\partial_y\eta-\partial_{v'}\eta\cdot B_z(u')\partial_z\eta\right)\\
&=\int_{\partial W\cap \{u'=1\}}2\exp(-g)|\partial_{v'}\eta|^2+\int_{\partial W\cap \{t=T\}}\frac{\exp(-g)}{2}(|\partial_{y}\eta|^2+|\partial_{z}\eta|^2+4|\partial_{v'}\eta|^2)\\
&\qquad+\int_W 2e^{-g}\partial_{u'}g|\partial_{v'}\eta|^2+\frac 12 e^{-g}\partial_{v'}g(|\partial_{y}\eta|^2+|\partial_{z}\eta|^2)-e^{-g}\partial_{v'}\eta\cdot B_y(u')\partial_y\eta-e^{-g}\partial_{v'}\eta\cdot B_z(u')\partial_z\eta\\
&\qquad-\int_{\partial W\cap \{u'=-1\}}2\exp(-g)|\partial_{v'}\eta|^2-\int_{\partial W\cap \{t=0\}}\frac{\exp(-g)}{2}(|\partial_{y}\eta|^2+|\partial_{z}\eta|^2+4|\partial_{v'}\eta|^2)
% \frac{e^{-g}}{2}|\partial_{v'}\eta|^2+\int_{\partial W\cap \{u=1\}} \frac{e^{-g}}{4}(|\partial_{v'}\eta|^2+|\partial_{v'}\eta|^2)+\int_{\partial W\cap \{u=-1\}}-\int_{}+\iint_W
\end{align*}
The first four terms in the bulk integral together are nonnegative because of Cauchy-Schwarz and condition \eqref{cond:gcondition}. Also, the term along $\{u'=-1\}$ is 0 by domain of dependence because $\eta$ is supported on the unit ball at time 0. Thus
\[
\int_{\partial W\cap \{t=T\}}\frac{\exp(-g)}{2}(|\partial_{y}\eta|^2+|\partial_{z}\eta|^2+4|\partial_{v'}\eta|^2)\le \int_{\{t=0\}}\frac{\exp(-g)}{2}(|\partial_{y}\eta|^2+|\partial_{z}\eta|^2+4|\partial_{v'}\eta|^2).
\]
By looking at the function $g$ as defined in \eqref{eq:gexample}, we then get that for some $C>0$ and some $K_0$, if $F=0$ and the assumptions of the lemma are satisfied, then we get for some constant $C$ that
\[
\int_{\partial W\cap \{t=T\}}|\partial_{y}\eta|^2+|\partial_{z}\eta|^2+|\partial_{v'}\eta|^2\le C\delta\exp(K_0\sqrt{t}).
\]
We now us Duhamel to put in an inhomogeneity satisfying \eqref{cond:Fbound}, and we obtain
\[
\int_{\partial W\cap \{t=T\}}|\partial_{y}\eta|^2+|\partial_{z}\eta|^2+|\partial_{v'}\eta|^2\le C\delta\exp(K_0\sqrt{t})+\int_{0}^T C\delta\exp(K_0\sqrt{t-s})\exp(K\sqrt{s})   ds\le C_1\delta\exp(5 K \sqrt{T})   ds
\]
If we now treat \eqref{eq:exprootequation} as a wave equation (putting both $F$ and the first order term on the right hand side), we get
\[
||D\eta(t,\cdot)||_{L^2}\le \int_0^t C_2\delta\exp(5 K \sqrt{s})ds+C_3\delta\le C_4\delta\exp(10 K \sqrt{t})
\]
\end{proof}
This upper bound allows us to prove Theorem~\ref{thm:generalinstab}.
\begin{proof}
We begin with some useful reductions for the problem. By applying rotations and reflections in the $y-z$ plane, we can assume that $B_y (u_0)$ has at least one eigenvalue with positive real part. We use $\lambda(B_y (u))$ to denote the largest real part among all eigenvalues of the matrix $B_y (u)$. We thus have $\lambda(B_y (u_0)) > 0$. Because of this, there exists some closed interval $[a,b]$ containing $u_0$ such that $\lambda(B_y (u))$ is uniformly bounded from below away from $0$ in this interval. Now, given any point $u \in [a,b]$, we note that the dimension of the space of generalized eigenvectors with corresponding eigenvalues having real part equal to $\lambda(B_y (u))$ can only decrease in a sufficiently small neighborhood (depending only on $B_y$) of $u$. We take a point $u_0 \in (a,b)$ for which this dimension is minimized. This gives us an interval of the form $[u_1,u_2] \subset [a,b]$ in which the dimension is equal to this minimum.

We may thus assume we are in the following setting. We are given an interval $[u_1,u_2]$. We have that $\lambda(B_y (u))$ is positive and bounded uniformly away from $0$ on this interval. If the eigenvalues of $B_y (u)$ are denoted by $\lambda_1 (u), \dots, \lambda_k (u)$, we have that that $|\Re(\lambda_i (u))| \le \lambda(B_y (u))$ for all $u \in [u_1,u_2]$. We have that the dimension of the vector space of generalized eigenvectors whose eigenvalues have real part equal to $\lambda(B_y (u))$ is constant on this interval.
Finally, we may assume that $\epsilon$ is so small such that $\Re(\lambda_i (u)) \le \lambda(B_y(u)) - 10 \epsilon$ for all $u \in [u_1,u_2]$ whenever $\lambda_i$ is an eigenvalue that does not correspond to the vector space of generalized eigenvectors having real part equal to $\lambda(B_y (u))$. These reductions come from picking an appropriate interval $[u_1,u_2]$ after applying appropriate rotations and reflections in the $y-z$ plane. All of these operations and the size of the interval $[u_1,u_2]$ depend only on $B_y$.

We will actually prove that we can take $K$ in the theorem statement to be given by
\begin{equation} \label{eq:GrowthRate}
    \begin{aligned}
    K=\frac{1}{\sqrt{2}\sqrt{u_2-u_1}}\int_{u_1}^{u_2}\lambda(B_y (u')) du'-\epsilon
    \end{aligned}
\end{equation}
for arbitrarily small $\epsilon>0$ (note that we can decrease the size of $\epsilon$ without violating the condition above that $\Re(\lambda_i (u)) \le \lambda(B_y(u)) - 10 \epsilon$).

The $N=1$ part of the of the theorem statement is essentially a special case of this (except, when $N=1$, it is fine if $B$ is negative on some part of the interval $[u_1,u_2]$). To get the theorem statement for general $N$, we note that after rotations in the $y-z$ plane and reflections, we can assume that the relevant linear combination is just $B_y(u_0)$, and then we can use continuity to take some small interval $[u_1,u_2]$ around $u_0$ in which Condition $2$ will be satisfied uniformly. Also, throughout the proof below, we will assume that $T>T_0(B_y,B_z,K,m,\epsilon)$ is sufficiently large. For small $T$, we will make the theorem statement true by taking an arbitrary solution and making $c$ sufficiently small.

We will use the geometric optics ansatz. We recall that this is motivated by considering the equation
\[
\Box \eta + B(t - x) (\partial_t - \partial_x) \eta = 0.
\]
In this equation, it is relatively straightforward to construct data which exhibits exponential growth using the geometric optics ansatz adapted to the null generators of the $t - x = c$ null hyperplanes. We wish to use a similar construction. However, because the first order terms have $\partial_y$ derivatives instead, the null direction for the geometric optics ansatz must have nontrivial $y$ component for growth to occur. Meanwhile, taking too large of a $y$ component makes the null geodesic exit the support of $B$ faster, meaning that growth can be sustained for a shorter period of time.

We balance these issues by choosing the null vector
\[
L=(L_t,L_x,L_y,L_z)=\left(1,1-\frac{u_2-u_1}{T},-\sqrt{2\frac{u_2-u_1}{T}-\frac{(u_2-u_1)^2}{T^2}},0\right)
\]
for some sufficiently large parameter $T$. Any sufficiently large choice of this parameter will allow the construction to work. We shall show that the solution has grown the desired amount at time $T$.

We also take some $\delta>0$ small (the relevant condition is that $m\delta<\epsilon/2$) and will use frequency $\mu=\exp(\delta\sqrt{T})$. We need such a high frequency in order to be able to bound the error terms.

If we wish to be roughly transported along the null vector $L$, it is natural to take something that is high frequency in the null direction
\[
\bar{L} = (L_t,-L_x,-L_y,-L_z) = \left (1,\frac{u_2-u_1}{T} - 1,\sqrt{2\frac{u_2-u_1}{T}-\frac{(u_2-u_1)^2}{T^2}},0\right).
\]
Using $\zeta$ to denote the spacetime coordinates, we take the ansatz
\begin{equation}\label{eq:geomoptics}
\eta=\exp(i\mu \bar{L} \cdot\zeta)\sum_{j=0}^M\frac{\varphi_j}{(i\mu)^j}+\Pi
\end{equation}
where $M=M(B_y,B_z,K,m,\epsilon)$ is sufficiently large (and we will set $T_0$ to be sufficiently large after we can fixed $M$). We take initial data $\Pi(0,\cdot)=\varphi_j(0,\cdot)=0$ for $j>0$ and $||\varphi_1(0,\cdot)||_{L^\infty}\sim ||\varphi_1(0,\cdot)||_{H^{2M+m+2}}\sim 1$.

Then, matching powers of $\mu$, we get the following system of equations
\begin{equation}\label{eq:phijODEs}
\begin{aligned}
2\partial_L \varphi_0+B_yL_y\varphi_0+B_zL_z\varphi_0&=0\\
2\partial_L \varphi_j+B_yL_y\varphi_j+B_zL_z\varphi_0&=-\square\varphi_{j-1}+B_y\partial_y\varphi_{j-1}+B_z\partial_z\varphi_{j-1}\\
\square\Pi+B_y\partial_y\Pi+B_z\partial_z\Pi&=\frac{\exp(i\mu \bar L\cdot\zeta)}{(i\mu)^M}\left(\square\varphi_M-B_y\partial_y\varphi_M-B_z\partial_z\varphi_M\right).
\end{aligned}
\end{equation}
where $\partial_L=\partial_t+L_x\partial_x+L_y\partial_y+L_z\partial_z$. 
Note that each $\phi_j$ satisfies a transport equation with forcing terms. We shall prove bounds on these solutions by examining the ODEs along the integral curves of $L$ that arise from integrating the transport equation. These ODEs will be controlled by comparing them with constant coefficient ODEs on sufficiently small intervals. Here we deal with general $k$, but note that the $k=1$ case is simpler, since we can easily explicitly solve the ODE, as opposed to solving approximately and bounding errors.

We take a fixed integral curve of $L$ starting at $t = 0$ and $x = -u_1$. This guarantees that the integral curve initially has $u'$ coordinate equal to $u_1$. We then have that $u'(t) = u_1 + {u_2 - u_1 \over T} t = u_2 {t \over T} + \left (1 - {t \over T} \right ) u_1$. Along this curve, the equations for the $\varphi_j$ become ODEs in the parameter $t$. The ODE for $\varphi_0$ is given by
\[
2 \varphi_0' (t) + B_y (u'(t))  L_y \varphi_0 (t) = 2 \varphi_0' (t) +B_y \left (u_2 {t \over T} + \left (1 - {t \over T} \right ) u_1 \right )  L_y \varphi_0 (t) = 0.
\]
In general, the equation becomes
\[
2 \varphi_j' (t) + B_y \left (u_2 {t \over T} + \left (1 - {t \over T} \right ) u_1 \right ) L_y \varphi_j (t) = F_j (t),
\]
where $F_j (t)$ depends on $\varphi_{j'}$ with $j' \le j - 1$. The following ODE estimates are used to estimate the solutions $\varphi_j$ of these equations. They will be used to show that the leading order term $\varphi_0$ experiences growth and that the other terms $\varphi_j$ with $j \ge 1$ in the expansion are controlled. Using Lemma \ref{lem:exproottmultiplier} to control the error $\Pi$ will then give us the desired result.

\begin{claim}
Let $P:[0,1]\to \R^{k\times k}$ be given by $P(a)=B_y(au_2+(1-a)u_1)$, and let $F$ be an arbitrary continuous function.
For any $\epsilon>0$ sufficiently small and if $R$ is a solution to the ODE
\begin{equation}\label{rescaledODE}
R'= {1 \over 2} P(t/T)(-L_y)R + F(t)
\end{equation}
then for all sufficiently large $T$ and with $0 \le s \le t \le T$, we have that
\begin{equation}\label{bound:ODEgrowth}
\begin{aligned}
|R(t)|\le C(B_y,\epsilon) \exp\left(\sqrt{\frac{(u_2-u_1)T}{2}}\left(\int_{s/T}^{t/T}\lambda(P(\tau))d\tau+\epsilon\right)\right)|R(s)|
\\ + C(B_y,\epsilon) \int_s^t \exp \left (\sqrt{{(u_2 - u_1) T \over 2}} \left(\int_{\hat t / T}^{t / T} \lambda(P(\tau)) d \tau + \epsilon \right)\right ) |F(\hat t)| d \hat t
\end{aligned}
\end{equation}

Also, when $F = 0$, there is initial data $R(0)=R_0$ for which
\begin{equation}\label{constr:ODEgrowth}
|R(T)|\ge \exp\left(\sqrt{\frac{(u_2-u_1)T}{2}}\left(\int_{0}^1\lambda(P(\tau))d\tau-\epsilon\right)\right)|R_0|
\end{equation}
when $T$ is sufficiently large as a function of $\epsilon$ and $B_y$.
\end{claim}
We note that
\[
\int_{s / T}^{t / T} \lambda(P(\tau)) d \tau = \int_{s / T}^{t / T} \lambda(B_y(\tau u_2 + (1 - \tau) u_1)) d \tau = {1 \over u_2 - u_1} \int_a^b \lambda(B_y(u)) d u,
\]
where $a = {s \over T} u_2 + \left (1 - {s \over T} \right ) u_1$ and $b = {t \over T} u_2 + \left (1 - {t \over T} \right ) u_1$. Thus, the bounds we are proving are consistent with \eqref{eq:GrowthRate}.

\begin{proof}
We note that, because $-L_y$ is of size roughly $1 / \sqrt{T}$, we see that $R$ will schematically experience exponential growth at a rate of ${1 \over \sqrt{T}}$. It is natural to rescale units in order to make the exponential growth rate comparable to $1$ instead.
Let $W > 0$ be some fixed scale. We shall treat the ODE as a perturbation of a constant coefficient ODE on intervals of length $W \sqrt{T}$ and we shall patch estimates on these intervals together.

In order to prove the upper bound, we note that by linearity it suffices to consider separately the case of $F = 0$ and nonzero initial data and the case of nonzero $F$ and vanishing initial data. We shall prove estimates in both cases and add them together.

We first consider the case of $F = 0$ and nonzero data. By rescaling and translating the domain, solving the ODE \eqref{rescaledODE} over $[s,s+W\sqrt{T}]$ is equivalent to solving the ODE
\begin{align*}
Q(0)&=R(s)\\
Q'(\tau)&= {1 \over 2} \sqrt{T}P(s/T+\tau/\sqrt{T})(-L_y)Q = P_T (\tau) Q
\end{align*}
for $\tau\in [0,W]$ and where $P_T(\tau) = {1 \over 2} \sqrt{T} P(s / T + \tau / \sqrt{T}) (-L_y)$. Now note that, holding $s/T$ constant and taking $u'=(s/T)u_1+(1-s/T)u_2$, we have that
\[
P_T (\tau) = {1 \over 2} \sqrt{T}P(s/T+\tau/\sqrt{T})(-L_y)\to {1 \over 2} P(s / T) \left(\lim_{T\to\infty}- \sqrt{T}L_y\right) = \sqrt{{u_2 - u_1 \over 2}} P(s / T) = \sqrt{{u_2-u_1 \over 2}} B_y(u')
\]
uniformly as $T \rightarrow \infty$.

%Now, for the upper bound, we pick $T$ so large such that
%\[
%|\sqrt{T} P(s / T + \tau / \sqrt{T}) L_y - \sqrt{2 (u_2 - u_1)} B_y (u')| \le \epsilon / 2.
%\]

We now compare the solution $Q$ of this ODE with the constant coefficient ODE
\begin{align*}
    Q_1 (0) &= R(s)
    \\ Q_1' (\tau) &= \sqrt{{{u_2 - u_1} \over 2}} P(s / T)Q_1.
\end{align*}
Denoting by $E(\tau)$ the error $E(\tau) = Q(\tau) - Q_1 (\tau)$, we get that $E$ satisfies the ODE
\begin{align*}
    E(0) &= 0
    \\ E' (\tau) &= P_T (\tau) Q(\tau) - \sqrt{{{u_2 - u_1}  \over 2}} P(s / T) Q_1 (\tau)
    \\ &= \sqrt{{(u_2 - u_1) \over 2}} P(s / T) E(\tau) + \left (P_T (\tau) - \sqrt{{(u_2 - u_1) \over 2}} P(s / T) \right ) Q (\tau).
\end{align*}
We know that
\begin{equation} \label{eq:ConstantCoeffODEBound}
    \begin{aligned}
    |Q_1 (\tau)| &\le C_1 (B_y) (1 + \tau^k) |R(s)| \exp \left (\sqrt{{u_2 - u_1}  \over 2} \lambda (P(s / T)) \tau \right ),
    \end{aligned}
\end{equation}
where $C_1 (B_y)$ is some constant depending on $B_y$. By applying the Duhamel principle, the solution $E$ is given by
\[
E(\tau) = \int_0^\tau \exp \left ( \sqrt{{{u_2 - u_1}  \over 2}} P(s / T) (\tau - t) \right ) \left (P_T (t) - \sqrt{{{u_2 - u_1}  \over 2}} P(s / T) \right ) Q(t) d t.
\]
Thus, we have that
\begin{equation}
    \begin{aligned}
    |E(t')| \le C_1 (B_y) \tau (1 + \tau^k) \exp \left (\sqrt{{u_2 - u_1}  \over 2} \lambda (P(s / T)) \tau \right ) \sup_{0 \le t \le \tau} \left |P_T (t) - \sqrt{{u_2 - u_1 \over 2}} P(s / T) \right | \sup_{0 \le t \le \tau} |Q(t)|
    \end{aligned}
\end{equation}
for all $0 \le t' \le \tau$. Now, we note that $\sup_{0 \le t \le \tau} |Q(t)| \le \sup_{0 \le t \le \tau} |Q_1 (t)| + \sup_{0 \le t \le \tau} |E(t)|$. Thus, taking the supremum in $t'$ between $0$ and $\tau$ gives us that
\begin{equation} \label{eq:ODEErrorBound}
    \begin{aligned}
    \sup_{0 \le t \le \tau} |E(t)| &\le C_1 (B_y) \tau (1 + \tau^k) \exp \left (\sqrt{{u_2 - u_1}  \over 2} \lambda (P(s / T)) \tau \right )
    \\ &\times \sup_{0 \le t \le \tau} \left |P_T (t) - \sqrt{{u_2 - u_1 \over 2}} P(s / T) \right |  \left (\sup_{0 \le t \le \tau} |Q_1 (t)| + \sup_{0 \le t \le \tau} |E(t)| \right ).
    \end{aligned}
\end{equation}
Let $\tilde{\epsilon} > 0$ be arbitrary. We now pick $T$ so large such that
\[
\left |P_T (t) - \sqrt{{(u_2 - u_1) \over 2}} P(s / T) \right | \le \epsilon_1,
\]
where $\epsilon_1 > 0$ is chosen such that $1 - \epsilon_1 C_1 (B_y) W (1+W^k) \exp \left (\sqrt{{u_2 - u_1 \over 2}} \lambda(u') W \right ) > 0$, and such that
\[
{W (1+W^k) \over 1 - \epsilon_1 C_1 (B_y) W (1+W^k)^2 \exp \left (\sqrt{{u_2 - u_1 \over 2}} \lambda(P(s / T)) W \right )} C_1^2 (B_y) \exp \left (\sqrt{{u_2 - u_1 \over 2}} \lambda(P(s / T)) W \right )^2 \epsilon_1 \le \tilde{\epsilon}.
\]
Then, using \eqref{eq:ConstantCoeffODEBound} and \eqref{eq:ODEErrorBound}, this gives us that
\begin{equation} \label{eq:ConstantCoeffComparison}
    \begin{aligned}
    \left |R(s+ \tau \sqrt{T})-\exp\left(\sqrt{\frac{u_2-u_1}{2}}P(s / T) \tau \right)R(s) \right | \le \sup_{0 \le t \le \tau} |E(t)| \le \tilde{\epsilon} |R(s)|
    \end{aligned}
\end{equation}
for all $\tau \le W$.
Now, by choosing $W$ sufficiently large and $\tilde\epsilon$ sufficiently small in terms of $\epsilon$ and $B_y$, we get that
\begin{equation} \label{eq:SubintEst}
    \begin{aligned}
    |R(s+W\sqrt{T})|\le \exp\left(\left(\sqrt{\frac{u_2-u_1}{2}}\lambda(P(s / T))+\epsilon/2\right)W\right)|R(s)|.
    \end{aligned}
\end{equation}
We note that this will, in general, force us to take $T$ even larger. We also note that $T$ and $W$ can be chosen uniformly in $u' \in [u_1,u_2]$ in the above.

We now decompose $[s,t]$ into the intervals $[s,s + W \sqrt{T}], [s + W \sqrt{T},s + 2 W \sqrt{T}], \dots,$ where the final interval may have to have length up to $2 W \sqrt{T}$ instead of $W \sqrt{T}$ because $t - s$ may not be an integer multiple of $W$. Let $N_I$ denote the number of such intervals. We then iteratively apply the above estimate on each interval starting with $[s,s + W \sqrt{T}]$ in order to bound the solution on the whole interval. This gives us the estimate
\[
|R(t)| \le C(B_y,\epsilon)\exp \left (\sqrt{{(u_2 - u_1) \over 2}} W \left ( \sum_{i = 1}^{N_I} \lambda \left (P \left ({s + (i - 1) W \sqrt{T} \over T} \right ) \right ) + {\epsilon \over 2} \right ) \right ) |R(s)|.
\]
where the constant factor comes from using \eqref{eq:ConstantCoeffComparison} and \eqref{eq:ConstantCoeffODEBound} for the extra length of the final interval.

Now, we note that
\[\sqrt{{(u_2 - u_1) \over 2}} W \sum_{i = 1}^I \lambda \left (P \left ({s + (i - 1) W \sqrt{T} \over T} \right ) \right ) = \sqrt{{(u_2 - u_1) T \over 2}} {W \over \sqrt{T}} \sum_{i = 1}^{N_I} \lambda \left ( P \left ({s + (i - 1) W \sqrt{T} \over T} \right ) \right ),
\]
and the sum
\[
{W \over \sqrt{T}} \sum_{i = 1}^{N_I} \lambda \left ( P \left ({s + (i - 1) W \sqrt{T} \over T} \right ) \right )
\]
converges to the integral of $\lambda(P)$ between $s / T$ and $t / T$ uniformly in $s,t$ as $T \rightarrow \infty$. Thus, after possibly picking $T$ larger, we get that
\begin{equation} \label{eq:FinalEst}
    \begin{aligned}
    |R(t)| \le C(B_y,\epsilon)\exp \left (\sqrt{{(u_2 - u_1) T \over 2}} \left ( \int_{s / T}^{t / T} \lambda(P(\tau)) d \tau + \epsilon \right ) \right ) |R(s)|,
    \end{aligned}
\end{equation}
which is the desired upper bound.

We now consider the case of nonzero $F$ and vanishing initial data at $t = s$. Let $L(t,s)$ denote the solution operator sending data at time $s$ to the solution at time $t$. Using the Duhamel principle, we note that
\[
R(t) = \int_s^t L(t,\hat t) F(\hat t) d \hat t.
\]
Using the bound on the solution operator obtained above, this means that
\[
|R(t)| \le C(B_y,\epsilon) \int_s^t \exp \left (\sqrt{{(u_2 - u_1) T \over 2}} \left(\int_{\hat t / T}^{t / T} \lambda(P(\tau)) d \tau + \epsilon \right)\right ) |F(\hat t)| d \hat t,
\]
finishing the proof of~\eqref{bound:ODEgrowth}.

We now proceed to prove the lower bound \eqref{constr:ODEgrowth} when $F = 0$. We shall follow the same strategy as for the upper bound, propagating estimates on intervals by comparing with constant coefficient ODEs. We decompose the interval $[0,T]$ into the intervals $[0,W \sqrt{T}], [W \sqrt{T},2 W \sqrt{T}], \dots$, where the last interval may have length up to $2 W \sqrt{T}$. We shall now propagate an estimate from the left endpoint of each interval to the right endpoint of each interval. The desired result will then follow by iterating over the intervals. Indeed, it suffices to prove the estimate
\[
|R((k + 1) W \sqrt{T})| \ge \exp \left (W \sqrt{{u_2 - u_1 \over 2}} \left (\lambda \left ( P \left ({k W \over \sqrt{T}} \right ) \right ) - {\epsilon \over 2} \right ) \right )|R(k W \sqrt{T})|,
\]
as the resulting exponential bound \eqref{constr:ODEgrowth} will follow in the same way that \eqref{eq:FinalEst} followed from \eqref{eq:SubintEst}. We shall now prove this bound.

For convenience, we will let $\lambda_0=\lambda \left ( P \left ({k W \over \sqrt{T}} \right )\right )$. Let $\Lambda(t)$ denote the vector space consisting of the span of all generalized eigenvectors of $P(t)$ with real part equal to $\lambda(P(t))$, and let $\Lambda^c (t)$ denote the complementary vector space consisting of the span of all of the remaining generalized eigenvectors. Moreover, let $Pr_t$ denote the projection associated with $\Lambda(t)$ with respect to the splitting given by $\Lambda(t)$ and $\Lambda^c (t)$, and let $Pr_t^c$ denote $I - Pr_t$ where $I$ is the identity. We note that $Pr_t$ can be written as a sum of eigenprojections associated to the matrix $P(t)$. We shall take data $R(0)$ having length $1$ and lying in $\Lambda(0)$.

Proceeding by induction on the intervals, we assume that we are given some interval $[k W \sqrt{T},(k + 1) W \sqrt{T}]$ and that we have that
\[
R(k W \sqrt{T}) = A v + \hat{\epsilon} B w,
\]
where $v \in \Lambda(k W \sqrt{T})$ and $w \in \Lambda^c (k W \sqrt{T})$ both have unit length. The numbers $A > 0$ and $B > 0$ are the amplitudes, and the key point is that $R$ essentially lies in the eigenspace of the eigenvalues having largest real part because we assume that $B \le A$ and that $\hat{\epsilon}$ is small. We shall show that this structure propagates with growth in the amplitude $A$, meaning that
\[
R ((k + 1) W \sqrt{T}) = A' v' + \hat{\epsilon} B' w',
\]
where $v' \in \Lambda((k + 1) W \sqrt{T})$ and $w' \in \Lambda^c ((k + 1) W \sqrt{T})$ both have unit length, where
\[
A' \ge A \exp \left (W \sqrt{{u_2 - u_1 \over 2}} \left (\lambda_0 - {\epsilon \over 2} \right ) \right )
\]
is the new amplitude, and where $B' \le A'$.

Let $s = k W \sqrt{T}$. Now, because $R(s) = A v + \hat{\epsilon} B w$, we have that
\begin{equation}
    \begin{aligned}
    \exp \left (W\sqrt{{u_2 - u_1 \over 2}} P(s / T)  \right ) R(s) = \exp \left (W\sqrt{ {u_2 - u_1 \over 2}} P(s / T) \right ) (A v + \hat{\epsilon} B w)
    \\ = A \exp \left (W \sqrt{{u_2 - u_1 \over 2}} P(s / T) \right ) v + \hat{\epsilon} B \exp \left (W\sqrt{{u_2 - u_1 \over 2}} P(s / T)  \right ) w.
    \end{aligned}
\end{equation}
We shall use \eqref{eq:ConstantCoeffComparison} to compare the solution with the solution to the constant coefficient ODE. We shall first analyze the solution of the constant coefficient ODE more closely.

Let $\tilde{v} = A\exp \left (W\sqrt{{(u_2 - u_1) \over 2}} P(s / T) \right ) v$. Now, we can decompose $\tilde{v}$ into vectors $v_1 = Pr_{(k + 1) W \sqrt{T}} (\tilde{v}) \in \Lambda((k + 1) W \sqrt{T})$ and $v_2 = \tilde{v} - v_1$. We now recall that eigenprojection is continuous as a function of the matrix entries (see \cite{Kato95}), meaning that $Pr_t$ is continuous as a function of $t$. By taking $T$ sufficiently large, we can thus make $|Pr_{(k + 1) W \sqrt{T}} (\tilde{v}) - Pr_{k W \sqrt{T}} (\tilde{v})| \le \epsilon' |\tilde{v}|$ for $\epsilon' > 0$ arbitrary.
Because $Pr_{k W \sqrt{T}} (\tilde{v}) = \tilde{v}$, this means that
\[
|v_2| \le \epsilon' |\tilde{v}|
\]
in the above decomposition. Now, we have that
\[
|\tilde{v}| \le |v_1| + |v_2| \le |v_1| + \epsilon' |\tilde{v}|,
\]
meaning that we have that
\[
|\tilde{v}| \le {1 \over 1 - \epsilon} |v_1|.
\]
Thus, we have that
\[
|v_2| \le {\epsilon' \over 1 - \epsilon'} |v_1|.
\]

Because $w \in \Lambda^c (k W \sqrt{T})$, we know that $w$ can be written as a linear combination of generalized eigenvectors whose eigenvalues have real part at most $\lambda_0 - 10 \epsilon$. Thus, for $W$ sufficiently large, we have that
\begin{equation}
    \begin{aligned}
    \left |\exp \left (W\sqrt{{(u_2 - u_1) \over 2}} P(s / T)  \right )  Pr_s^C (R(s))  \right | &= \left |\hat{\epsilon} B \exp \left (W\sqrt{{(u_2 - u_1) \over 2}} P(s / T)  \right ) w \right |
    \\ &\le \hat{\epsilon} B \exp \left (W \sqrt{{u_2 - u_1 \over 2}} (\lambda_0- 5 \epsilon) \right ).
    \end{aligned}
\end{equation}
where we are using the facts that $W$ is sufficiently large and that we can provide uniform lower bounds on angles between elements of the basis for some basis of generalized eigenvectors. 
Let
\[
\tilde{w} = \hat{\epsilon} B \exp \left (W\sqrt{{(u_2 - u_1) \over 2}} P(s / T) \right ) w.
\]
We now set $w_1 = Pr_{(k + 1) W \sqrt{T}} (\tilde{w})$ and $w_2 = \tilde{w} - w_1$. We note that $w_1 \in \Lambda((k + 1) W \sqrt{T})$ and $w_2 \in \Lambda^c ((k + 1) W \sqrt{T})$. We also note that $Pr_{\Lambda(k W \sqrt{T})} (\tilde{w}) = 0$. For $T$ sufficiently large and using the continuity of the eigenprojections as was done for $\tilde{v}$, we have that
\[
|w_1| \le \epsilon' |\tilde{w}|
\]
for $\epsilon' > 0$ arbitrary.

We have now decomposed
\[
\exp \left (W\sqrt{{(u_2 - u_1) \over 2}} P(s / T) \right ) R(s) = \exp \left (W\sqrt{{(u_2 - u_1) \over 2}} P(s / T)  \right ) (A v + \hat{\epsilon} B w) = v_1 + v_2 + w_1 + w_2,
\]
where $v_1, w_1 \in \Lambda((k + 1) W \sqrt{T})$ and $v_2, w_2 \in \Lambda^c ((k + 1) W \sqrt{T})$. Moreover, we recall the estimate \eqref{eq:ConstantCoeffComparison}, which says that
\begin{equation}
    \begin{aligned}
    |E((k + 1) W \sqrt{T})| &= |E(s + W \sqrt{T})|
    \\ &= \left |R(s + W \sqrt{T}) - \exp \left (W\sqrt{{(u_2 - u_1) \over 2}} P(s / T)  \right ) R(s) \right |\\
    &\le \tilde{\epsilon} |R(s)|.
    \end{aligned}
\end{equation}
If we set $E_1 = Pr_{(k + 1) W \sqrt{T}} (E((k + 1) W \sqrt{T}))$ and $E_2 = E((k + 1) W \sqrt{T}) - E_1$, this means that, for $T$ sufficiently large, we can take $\tilde\epsilon$ sufficiently small to get that $|E_1| \le \epsilon' A$ and $|E_2| \le \epsilon' A$. We have now decomposed
\[
R((k + 1) W \sqrt{T}) = v_1 + w_1 + E_1 + v_2 + w_2 + E_2,
\]
where $v_1, w_1, E_1 \in \Lambda((k + 1) W \sqrt{T})$ and $v_2, w_2, E_2 \in \Lambda^c ((k + 1) W \sqrt{T})$. We thus have that
\[
v' = {v_1 + w_1 + E_1 \over |v_1 + w_1 + E_1|},
\]
and that
\[
w' = {v_2 + w_2 + E_2 \over |v_2 + w_2 + E_2|}.
\]
We must now show that $A' = |v_1 + w_1 + E_1| \ge A \exp \left (W\sqrt{{u_2 - u_1 \over 2}}\left (\lambda_0 - {\epsilon \over 2} \right ) \right )$, and that $|v_2 + w_2 + E_2| \le \hat{\epsilon} A'$. These estimates follow immediately from the above considerations. Indeed, we have that $|v_1 + w_1 + E_1| \ge |v_1| - |w_1| - |E_1|$. Then, for $W$ sufficiently large, we have that
\[
|v_1| = \left |Pr_{(k + 1) W \sqrt{T}} \left (\exp \left (W\sqrt{{u_2 - u_1 \over 2}} P(s / T) \right ) A v \right ) \right | \ge C(B_y) (1 - 10 \epsilon') \exp \left (W\sqrt{{u_2 - u_1 \over 2}}\left(\lambda_0-\frac{\epsilon}{4}\right) \right ) A.
\]
Moreover, we have that
\[
|w_1| \le \epsilon' \exp \left (W\sqrt{{u_2 - u_1 \over 2}}  (\lambda_0 - 5 \epsilon) \right ) A.
\]
Combining these facts with \eqref{eq:ConstantCoeffComparison} gives us that
\[
A' = |v_1 + w_1 + E_1| \ge C(B_y) (1 - 10 \epsilon') A \exp \left (W\sqrt{{u_2 - u_1 \over 2}} \left(\lambda_0-\frac{\epsilon}{4}\right)\right ) - 2 \epsilon' A \exp \left (W\sqrt{{u_2 - u_1 \over 2}}  (\lambda_0 - 5 \epsilon) \right ).
\]
Taking $W$ sufficiently large gives us that
\[
A' \ge \exp\left (W\sqrt{u_2 - u_1 \over 2} \left (\lambda_0 - {\epsilon \over 2} \right ) \right ) A,
\]
giving us the desired growth of the amplitude. We note that this also tells us that
\[
|v_1 + w_1 + E_2| \ge {1 \over 2} |v_1|
\]
as long as $\epsilon'$ is sufficiently small.

Similarly, we have that
\[
|v_2| \le {\epsilon' \over 1 - \epsilon'} |v_1|,
\]
and we have that
\[
|w_2| \le \hat{\epsilon} B \exp \left (W\sqrt{u_2 - u_1 \over 2} (\lambda_0- 5 \epsilon) \right ).
\]
Combining this with \eqref{eq:ConstantCoeffComparison} and taking $W$ sufficiently large gives us that
\[
|v_2 + w_2 + E_2| \le {\epsilon' \over 1 - \epsilon'} |v_1| + \frac{\hat{\epsilon}}{4} B \exp \left (W\sqrt{u_2 - u_1 \over 2} \left(\lambda_0- {\epsilon \over 2}\right) \right )  + 2 \epsilon' A.
\]
Now, because $\epsilon'$ can be made arbitrarily small by taking $T$ sufficiently large and by the lower bounds on $|v_1|$, we have that
\[
|v_2 + w_2 + E_2| \le \frac{\hat{\epsilon}}{10} |v_1|.
\]
Thus,  we have that
\[
|v_2 + w_2 + E_2| \le \hat{\epsilon} |v_1 + w_1 + E_1|
\]
for $\epsilon'$ sufficiently small. This means that
\[
R((k + 1) W \sqrt{T}) = A' v' + \hat{\epsilon} B' w'
\]
with $B' \le A'$ when $T$ is sufficiently large, giving us the desired result.
\end{proof}

As a direct consequence of \eqref{constr:ODEgrowth}, we know that $\varphi_0$ will grow appropriately along integral curves of the vector field $L$. We must now control the other terms $\varphi_j$ in the expansion. We have the following estimates on these terms, which are proven using induction along with \eqref{bound:ODEgrowth} on the ODEs along the integral curves of $L$.

\begin{claim}\label{varphigrowth}
For any $j\le M$ and any $b\le 2M+m-2j+6$ and any $\epsilon > 0$, there is some constant $C_{j,b,\epsilon,B_y}$ so that 
\begin{equation}\label{upperboundinducted}
||D^b\varphi_j||_\infty(t)\le A C_{j,b,\epsilon,m,M} \exp \left(\sqrt{\frac{u_2-u_1}{2}T}\left(\int_{0}^{t/T}\lambda(P(\tau))d\tau + \epsilon \right) \right ).
\end{equation}
for all $t\in [0,T]$ with $T$ sufficiently large where $A$ is the size of the $C^{2 M + m - 2 j}+6$ norm of the initial data for $\varphi_0$.
\end{claim}
\begin{proof}
We prove this by a double induction on $j$ and $b$. Suppose that we obtained $\eqref{upperboundinducted}$ for every pair $(j,b)$ where either $j$ is smaller or $j$ is the same and $b$ is smaller. The base case of $j = b = 0$ follows immediately from using \eqref{bound:ODEgrowth} on the equation for $\varphi_0$.

Now, applying $D^{\beta}$ to the equation \eqref{eq:phijODEs} for $\varphi_j$ for some multiindex $\beta$ with $|\beta|\le b$, we get that
\begin{equation}\label{eq:differentiated}
\partial_{L} D^{\beta}\varphi_j+B_yL_yD^{\beta}\varphi_j=-L_y[D^\beta,B_y]\varphi_j+D^\beta\left(-\square\varphi_{j-1}-B_y\partial_y\varphi_{j-1}-B_z\partial_z\varphi_{j-1}\right)
\end{equation}
where we are defining $\varphi_{j-1}=0$ for convenience of notation. 
Now, by the inductive hypothesis, there is some constant $C_{j,b,\epsilon}$ such that the right-hand side of $\eqref{eq:differentiated}$ is bounded by
\[
C_{j,b,\epsilon}A \exp\left(\sqrt{\frac{u_2-u_1}{2}T}\left(\int_{0}^{t/T}\lambda(P(\tau))d\tau + \epsilon / 10 \right)\right).
\]
Now, using the estimate \eqref{bound:ODEgrowth} and using that $T$ is sufficiently large, we have that
\begin{equation}
    \begin{aligned}
    |D^\beta \phi_j| (t) &\le C_{j,b,\epsilon}A \exp\left(\sqrt{\frac{u_2-u_1}{2}T}\left(\int_{0}^{t/T}\lambda(P(\tau))d\tau + \epsilon / 10 \right)\right)\\
    &\qquad+\hat C_{j,b,\epsilon} \int_0^t \exp \left (\sqrt{{(u_2 - u_1) T \over 2}} \left ( \int_{\hat t/T}^{t / T} \lambda(P(\tau)) d \tau + \epsilon / 10 \right ) \right )
    \\ &\qquad\times \exp\left(\sqrt{\frac{u_2-u_1}{2}T}\left(\int_{0}^{\hat t / T}\lambda(P(\tau))d \tau + \epsilon / 10 \right)\right) d \hat t
    \\ &\le C_{j,b,\epsilon} t \exp \left(\sqrt{\frac{u_2-u_1}{2}T}\left(\int_{0}^{t/T}\lambda(P(\tau))d\tau + \epsilon / 5 \right) \right )
    \\ &\le C_{j,b,\epsilon} \exp \left(\sqrt{\frac{u_2-u_1}{2}T}\left(\int_{0}^{t/T}\lambda(P(\tau))d\tau + \epsilon \right) \right ),
    \end{aligned}
\end{equation}
giving us the desired result.
\end{proof}
We now pick $\epsilon$ very small and $\delta$ such that $\delta \le {\epsilon \over 2 m}$. We use the geometric optics ansatz with $\mu = \delta \sqrt{T}$. Let $\chi$ be a smooth bump function with $|\chi| \le 1$, with $\chi = 1$ in a neighborhood of some point $p$ where $-x_1 = u_1$, and with $\chi$ compactly supported in the unit ball. Such a function will have $H^m$ and $C^{2 M + m + 6}$ norm comparable to some constant depending on $B_y$, $m$, and $M$ (this dependence means that we may have to pick $T$ even larger in order to appropriately apply the above results). We now recall $R_0$ from \eqref{constr:ODEgrowth} and take data equal to
\[
e^{-m \delta \sqrt{T}} ||\chi||_{H^m}^{-1} ||\chi||_{C^{2 M + m + 6}}^{-1} \chi R_0
\]
for $\varphi_0$. The initial data for $\varphi_0$ then has $C^{2 M + m + 6}$ norm at most $1$. This guarantees that appropriate traces at $t = 0$ of the solution %of the wave equation
\[
\sum_{j = 0}^M {\phi_j \over (i \mu)^j} + \Pi
\]
constructed using the geometric optics ansatz will have $H^m$ norm at most $1$. This corresponds to the solution
%of the wave equation
we construct having initial data with $H^m$ norm at most $1$. We note that the data for $\varphi_0$ has amplitude comparable to $\exp(-m \delta \sqrt{T})$ near $p$. Now, we pick some  $\bar\epsilon$ sufficiently small compared to $\delta$ and we have shown that
\[
|\varphi_0| (T) \ge C_{B_y,\delta} \exp (-m \delta \sqrt{T}) \exp \left (\sqrt{(u_2 - u_1) T \over 2} \left ( \int_0^{t / T} \lambda(P(\tau)) d \tau - \bar\epsilon \right ) \right ),
\]
where we have followed an integral curve of $L$ starting at the point $p$. Moreover, by \eqref{upperboundinducted}, we have that
\begin{equation}
    \begin{aligned}
    {|\varphi_j| (T) \over |\mu|^j} &\le C_{j,B_y,\delta} \exp (-(m + j) \delta \sqrt{T}) \exp \left (\sqrt{(u_2 - u_1) T \over 2} \left ( \int_0^{t / T} \lambda(P(\tau)) d \tau + \bar\epsilon \right ) \right )
    \\ &\le C_{j,B_y,\delta} \exp \left (-{\delta \sqrt{T} \over 10} \right ) |\varphi_0| (T),
    \end{aligned}
\end{equation}
where we have followed the integral curve of $L$ starting at $p$. Thus, for $T$ sufficiently large, we get that the dominant term among the $\phi_j$ in the geometric optics ansatz is $\phi_0$.

We now simply estimate the remainder $\Pi$. By commuting the equation for $\Pi$ with unit derivatives and using Lemma~\ref{lem:exproottmultiplier} with $K$ appropriately chosen, we get that
\[
\Vert \partial \Pi \Vert_{H^1 (\Sigma_t)} \le C_{M,K} {1 \over |\mu|^M} \exp (10 K \sqrt{t}) \le C \exp (10 K \sqrt{t} - M \delta \sqrt{t}).
\]
Taking $M$ sufficiently large and using the Sobolev embedding theorem will give us that $||\Pi||_{L^\infty (\Sigma_T)} \le C_{M,\delta,K} \exp \left (-{\delta \sqrt{T} \over 10} \right ) |\varphi_0| (T)$. Thus, for $T$ sufficiently large as a function of $M$, $\delta$, $K$, and $\epsilon$, we have that
\[
|\eta| (T) \ge |\varphi_0| (T) - \sum_{j = 1}^M {|\phi_j| (T) \over |\mu|^j} - |\Pi| (T) \ge \exp \left (\sqrt{(u_2 - u_1) T \over 2} \left ( \int_0^{t / T} \lambda(P(\tau)) d \tau - \epsilon \right ) \right ),
\]
as desired.
\end{proof}

\section{Related Directions} \label{sec:RelatedDirections}
In this section, we shall describe problems related to the one studied in this paper. We shall discuss if and how the strategies followed above can be applied in these cases, and some of the additional difficulties that arise.

We first note that the decay rates we get are not sharp. They can be improved by applying, for example, the $r^p$ method of Dafermos and Rodnianski (see \cite{DafRod10}). Using these methods along with modified weighted Sobolev inequalities like those in Section~\ref{sec:KlaiSobInequalities} would allow the decay rate $(1 + t)^{-2 + \delta}$ for good derivatives. This would also allow us to show that the radiation field of the renormalized perturbation grows like $(t + r)^\delta$. One idea which we believe could be used to get even sharper results would be to use fractional angular Laplacians. Indeed, for getting pointwise decay, all that is really necessary is commuting with $1 + \epsilon$ angular derivatives. This is because the sphere is two dimensional, and the embedding from $H^1$ into $L^\infty$ just barely fails. Such an operator is better behaved when hitting the background traveling wave because it has lower weights (from Section~\ref{sec:Geometry}, we know that every derivative gives a power of $\sqrt{t}$, so this should schematically grow like $t^{{1 \over 2} + \epsilon}$). However, commuting with fractional derivatives also introduces other difficulties, such as nonlocality. We also have not used the observation that certain commutation fields (scaling and the boost in the $x$ direction) do not introduce bad weights. Indeed, in the setting of plane symmetric solutions, the $\partial_v$ derivative is better than the angular derivatives in terms of decay. In fact commuting with $r^2 \partial_v$ or with an appropriate modification of the conformal Morawetz vector field $K_0 = (t^2 + r^2) \partial_t + 2 t r \partial_r$ should introduce only $\sqrt{t}$ weights even though they have quadratic weights.\footnote{We note that the standard null form is better from the viewpoint of this analysis as well, as in that null form, the worst derivative $\partial_u$ is always hit with the best derivative $\partial_v$, while the angular derivatives which are better than $\partial_u$ but worse than $\partial_v$ are always squared.} The fact that the boost in the $x$ direction introduces better weights was already used in \cite{AbbresciaWong19}. These observations could also be useful establishing better decay estimates.

It is also natural to try to remove the assumptions of compact support for both the traveling wave and the perturbation. We believe that both of these assumptions in the paper are technical, and we conjecture that global nonlinear stability continues to be true as long as Condition~\ref{cond:unexciting} is satisfied and the traveling wave solution and perturbation each have data which decays sufficiently quickly (for the traveling wave, this means that the profile functions $f(t - x)$ have sufficiently fast decay rates away from $0$, while for the perturbation, this means that the data have sufficiently fast decay rates away from the origin $r = 0$ in $\Sigma_0$). In fact, we note that in the work of Liu-Zhou in \cite{LiuZhou19}, they do not require compact support in the traveling wave $f$, and instead, they only require sufficiently fast decay.

Another natural problem is to try to extend these results to other dimensions. Because proving global stability for the trivial solution of nonlinear wave equations satisfying the null condition is much harder in $2 + 1$ dimensions, we shall first restrict ourselves to $n + 1$ dimensions with $n \ge 4$. We shall discuss a bit about $2 + 1$ dimensions after. We also note that there is a section describing the Lorentzian minimal surface equation in $2 + 1$ dimensions later in this section, but the nonlinearities in the minimal surface equation are better behaved than general nonlinearities satisfying the null condition (the $1 + 1$ dimensional case is discussed there, but this case is special, as there is no dispersive mechanism).

We believe that the analogous stability problem when Condition~\ref{cond:unexciting} is satisfied can be solved in $n + 1$ with $n \ge 4$ using the same strategy as in this paper. There would, however, be required modifications. One helpful fact in this case is that solutions to the wave equation decay faster in higher dimensions (in fact, in dimensions $n + 1$ with $n \ge 4$, we recall from \cite{Kl85} that global stability holds for the trivial solution of general wave equations with quadratic nonlinearities). Another difference which works to your benefit is that the gain in the volume of the quantity analogous to $S_t$ is larger. In $3 + 1$ dimensions, the volume goes from $t^2$ to $t$, which is a gain of $t$. In $n + 1$ dimensions, the volume goes from $t^{n - 1}$ to $t^{{n - 1 \over 2}}$, which is a gain of $t^{{n - 1} \over 2}$. The difference which requires more care is that Sobolev embedding is worse in higher dimensions. This means that proving pointwise decay will require commuting with more vector fields (see Section~\ref{sec:KlaiSobInequalities}), giving worse weights when the derivatives hit the traveling wave. From a rough calculation, we believe the gain in volume makes up for this. When Condition~\ref{cond:unexciting} is not satisfied but Condition~\ref{cond:M} is satisfied, the proof we gave can be directly adapted to show linear instability.

In $2 + 1$ dimensions, the problem is harder due to the weaker decay of solutions to the linear wave equations. Indeed, in two dimensions, both quadratic and cubic terms must have special structure in order for existing proofs of global stability to work (see \cite{Katayama17}) because general cubic nonlinearities just fail to result in sufficient decay. The first result for nonlinearities satisfying the null condition in $2 + 1$ dimensions was the work of Godin in \cite{Godin93} in which global existence was established for certain class of nonlinearities satisfying the null condition. Then, in \cite{Alinhac01} and \cite{Alinhac201}, Alinhac was able to prove a general almost global existence result and was able to prove global existence in the absence of quadratic semilinear terms satisfying the null condition. The full problem allowing for semilinear terms remained open a while longer, but has now been resolved in the works of Katayama in \cite{Katayama17} and Zha in \cite{Zha19}.\footnote{We thank Dongbing Zha for making us aware of these results as well as the result \cite{Godin93} by Godin.} In $2 + 1$ dimensions, there is an instability statement analogous to Theorem~\ref{thm:generalinstab} that can be proven in the same way as in this paper. For stability, we believe that the geometric observations from Section~\ref{sec:Geometry} as well as certain aspects of the scheme used in the proof of Theorem~\ref{thm:main} would be useful in studying stability of plane waves in $2 + 1$ dimensions. However, given the added difficulties in $2 + 1$ dimensions, there would have to be substantial modifications (see \cite{Katayama17} and \cite{Zha19}). We believe that this is an interesting problem to pursue.

Next, one may ask what happens when~\eqref{eq:zerothversion} and the traveling wave do not satisfy the conditions of any of the theorems. The most interesting such case is when $m_{ij\ell}$ in the right-hand side of \eqref{eq:zerothversion} is antisymmetric, or the sum of antisymmetric terms and standard null forms (see the discussion at the beginning of Section~\ref{sec:unstable}). In that case, the geometric optics ansatz does not give growth, and we are not aware of a good construction to show instability. It is plausible that such traveling wave solutions are in fact stable, but we do not know how to prove this.

We shall now describe the main parts of adapting the above proof to the Lorentzian minimal surface equation. This problem has already been studied by Abbrescia-Wong in \cite{AbbresciaWong19} and by Liu-Zhou in \cite{LiuZhou19}. To our knowledge, these, along with \cite{AbbresciaChen19} for the wave map equation, have been the only examples of proving stability for traveling wave solutions to nonlinear wave equations in higher dimensions. We note that in $1 + 1$ dimensions, every solution is some kind of traveling wave, and there has been a lot of work on proving the stability of the trivial solution (see \cite{Wong17}, \cite{LuliYangYu18}, and \cite{Zha20}) and on proving general global existence results for certain equations (see \cite{AbbresciaWong219}). We note that Shao and Zha prove global stability of suitable plane wave solutions in $1 + 1$ dimensions for any quasilinear wave equation satisfying the null condition in \cite{ShaZha20}. Finally, there is a long and rich history of studying conservation laws in $1 + 1$ dimensions, see \cite{Daf16}.

The Lorentzian minimal surface equation is formally a stationary point of the Lagrangian
\[
\iint_{\R^{n+1}} \sqrt{1+\partial_\alpha\varphi\partial^\alpha\varphi}
\]
The name comes from the fact that when the metric is Euclidean, rather than the Minkowski metric, minimal surfaces are stationary points.

The minimal surface equation can be written as
\[
\Box \varphi = -{m(d \varphi,d (m(d \varphi,d \varphi))) \over 2(1-m(d \varphi,d \varphi))}
\]
where $m$ is the standard null form (see \cite{Lind04}). Because only the standard null form appears, the equation is consistent with satisfying a quasilinear version of Condition~\ref{cond:unexciting}. In \cite{AbbresciaWong19}, Abbrescia-Wong showed global stability of traveling wave solutions to this equation in $n + 1$ dimensions with $n \ge 3$. Meanwhile, Liu-Zhou showed in \cite{LiuZhou19} the global stability of traveling wave solutions in $n + 1$ dimensions with $n \ge 2$. We shall restrict ourselves to $2 + 1$ dimensions in the following discussion.

We will use the coordinate system $(v',u',y)$. We shall now derive the equations for a perturbation of the minimal surface equation. After deriving these equations, we shall describe how the above scheme can be modified to deal with this problem. 

In these coordinates, the Lorentzian minimal surface equation takes the form 
\bal
&(4\partial_{u'}\partial_{v'}+\partial_y\partial_y)\varphi(1+\partial_y\varphi\partial_y\varphi+4\partial_{u'}\varphi\partial_{v'}\varphi)\\
&-\bigg(4\partial_{v'}\varphi\partial_y\varphi\partial_{u'}\partial_y\varphi+4\partial_{u'}\varphi\partial_y\varphi\partial_{v'}\partial_y\varphi+4(\partial_{v'}\varphi)^2\partial_{u'}^2\varphi+4(\partial_{u'}\varphi)^2\partial_{v'}^2\varphi+8\partial_{u'}\varphi\partial_{v'}\varphi\partial_{u'}\partial_{v'}\varphi+\partial_y\varphi\partial_y\varphi\partial_y\partial_y\varphi\bigg).
\eal
For any profile $f$, the function $f(u')$ is a solution to the above equation.

We shall begin by describing how a good gauge can be chosen (see \cite{AbbresciaWong19} where this gauge is described). We take $\varphi=f(u')+\eta$ and get that the linear part has an extra term of the form $-4(f')^2\partial_{v'}^2\eta$, which one wouldn't expect to decay. However, we can get rid of this term by choosing the appropriate gauge, namely setting $\overline{v}=v'+g(u')$ for $g'=(f')^2$ (this corresponds to turning each $\partial_{v'}$ into $\partial_{\overline v}$ and each $\partial_{u'}$ into $\partial_{u'}+(f')^2\partial_{\overline v}$), we get the linear part to be the wave equation. In fact, we get the usual Lorentzian minimal surface equation above plus the terms
\begin{equation} \label{eq:MinSurfEq}
    \begin{aligned}
    &+4f'\partial_{v'}\eta\partial_y^2\eta+8f'\partial_{v'}\eta\partial_{u'}\partial_{v'}\eta-8f'\partial_{u'}\eta\partial_{v'}^2\eta-4f''(\partial_{v'}\eta)^2-4f'\partial_y\eta\partial_{v'}\partial_y\eta\\
    &+4(f')^2(\partial_y\eta)^2\partial_{v'}^2\eta+4(f')^2(\partial_{v'}\eta)^2\partial_y^2\eta-8(f')^2\partial_{v'}\eta\partial_y\eta\partial_{v'}\partial_y\eta-8f'f''(\partial_{v'}\eta)^3.
    \end{aligned}
\end{equation}
The equation is now written in a way that the background traveling wave only influences the perturbation in nonlinear terms. We believe that the scheme used in this paper can now be adapted to understand this problem. To begin with, Section~\ref{sec:Geometry} can clearly be adapted to this problem. The Klainerman-Sobolev Inequalities in Section~\ref{sec:KlaiSobInequalities} can also be used. However, the scheme used in Sections \ref{sec:BootstrapAssumptions} and \ref{sec:ClosingEnergy} must be adapted to deal with the additional difficulties that the equations are quasilinear and are now in $2 + 1$ dimensions instead of $3 + 1$ dimensions. Because we are proving stability and the main stability mechanism is decay, the fact that the equations are now quasilinear perturbations of the flat wave equation are not expected to introduce serious new difficulties. The fact that the equations are in $2 + 1$ dimensions requires substantial modifications for several reasons. To begin with, solutions of the wave equation decay more slowly in two dimensions. Quadratic semilinear terms, even ones satisfying the null condition, are difficult to deal with. However, we note that there is additional structure in the Lorentzian minimal surface equation. The nonlinearities are schematically of the form $\partial \phi \partial \partial m(d \phi,d \phi)$ where $m$ is some null form, meaning that the nonlinearities should decay even faster than most nonlinearities satisfying the null condition. However, the linearization around the background traveling wave does introduce nonlinear terms which are only quadratic. Recalling Section~\ref{sec:Geometry}, we see that $\partial_y = \overline{\partial} + O(1 / \sqrt{t})$ in $S_t$ and $\partial_{v'} = \overline{\partial} + O(1 / t)$ in $S_t$. Thus, looking at the terms in \eqref{eq:MinSurfEq}, we see that they are better in terms of decay than most quadratic terms satisfying the null condition because they have either two $\partial_{v'}$ derivatives or one $\partial_{v'}$ derivative and two $\partial_y$ derivatives. This structure would have to be used.

In \cite{LiuZhou19}, Liu-Zhou are able to treat more general plane wave solutions than the ones considered here. They are able to consider solutions which, for the $2 + 1$ dimensional Lorentzian minimal surface equation, are schematically of the form $(a + b y) f(t - x)$. The transformation in Section~\ref{sec:transformation} does not seem to easily generalize when perturbing such solutions, so the strategy we have followed cannot be directly applied. In \cite{LiuZhou19}, they are still able to perform renormalized energy estimates in this setting. We believe that it would be interesting to study the stability of such solutions to systems of semilinear wave equations as well, although we note that there are equations which only admit solutions of the form $f(t - x)$ and not $y f(t - x)$ as solutions (for example, the standard case of $\Box \phi = m(d \phi,d \phi)$ with $m$ the standard null form only admits solutions of the form $f(t - x)$ and not $y f(t - x)$ except in the trivial case where $f = 0$). On way to proceed would be to either use a renormalization scheme like that in \cite{LiuZhou19} or finding a transformation analogous to the one in Section~\ref{sec:transformation}. Should this be possible, we believe that the observations in Section~\ref{sec:Geometry} and the use of low regularity Klainerman-Sobolev Inequalities in the scheme carried out in Section~\ref{sec:KlaiSobInequalities} could still be useful.

It is also natural to ask what happens to systems of equations satisfying the weak null condition. The weak null condition was originally introduced by Lindblad and Rodnianski in \cite{LindRod03} in the context of studying the stability of Minkowski space for the Einstein Equations in wave coordinates. We note that it is still an open problem to prove global stability for the trivial solution of general semilinear systems satisfying the weak null condition. However, assuming an additional structural condition on the system which he called the \emph{hierarchical weak null condition}, Keir was able to establish a very general result on the global stability of trivial solutions of nonlinear wave equations on nontrivial backgrounds in \cite{Keir18}. As the general problem is still open, we shall restrict ourselves to equations satisfying the hierarchical weak null condition.

Because our discussion on systems of equations satisfying the hierarchical weak null condition will only be schematic, we shall not give the precise definition and shall only introduce an example below. A very thorough description of both the weak null condition and the heirarchical weak null condition is given in \cite{Keir18}.

An example of a system satisfying the hierarchical null condition which admits traveling wave solutions is
\begin{equation}
    \begin{aligned}
    \Box \phi_1 &= m(d \phi_2,d \phi_2)
    \\ \Box \phi_2 &= (\partial_t \phi_1)^2,
    \end{aligned}
\end{equation}
as we can take $\phi_1 = 0$ and $\phi_2 = f(t - x)$. Such equations could potentially have even worse instabilities. Indeed, if we consider instead the system
\begin{equation}
    \begin{aligned}
    \Box \phi_1 &= m(d \phi_2,d \phi_2)
    \\ \Box \phi_2 &= (\partial_t \phi_2 - \partial_x \phi_2) (\partial_t \phi_1 - \partial_x \phi_1)
    \end{aligned}
\end{equation}
and we linearize around the solution $\phi_1 = 0$ and $\phi_2 = f(t - x)$, then we can see that the linearization will contain terms like those in \eqref{eq:ExpGrowth}. This means that solutions of the perturbation should be able to experience exponential growth, as can be seen by using geometric optics. In the case where Condition $1$ is met and also no terms like those in \eqref{eq:ExpGrowth} appear in the linearization, we conjecture that there is global stability in the cases where global stability is known to be true for the trivial solution (see \cite{Keir18}). This condition can be described as saying that the background traveling wave only excite standard null forms and not the antisymmetric null forms or the other quadratic terms. Proving this is more delicate and would require other ideas. For example, commuting with two weighted commutation fields now seems to be too much, and one possibility to improve already comes from commuting with fractional angular derivatives as was described earlier in this Section.

We finally mention the behavior of systems of semilinear wave equations admitting multiple traveling wave solutions which move in different directions. If we consider the system of equations
\begin{equation}
    \begin{aligned}
    \Box \phi_1 &= m(d \phi_2,d \phi_2)
    \\ \Box \phi_2 &= m(d \phi_1,d \phi_2) + m(d \phi_3,d \phi_2)
    \\ \Box \phi_3 &= m(d \phi_2,d \phi_2),
    \end{aligned}
\end{equation}
we note that $\phi_2 = 0$, $\phi_1 = f(t - x)$, and $\phi_3 = f(t - (x + y) / \sqrt{2})$ is a solution to this system of equations. Thus, this system admits vector valued solutions where different components of the vector are traveling waves moving in different directions. In fact, we note that we find another solution in $\phi_2 = 0$, $\phi_1 = f(t - x) + f(t - (x + y) / \sqrt{2})$, and $\phi_3 = 0$. In this example, the vector valued solution is such that the second component has traveling waves moving in different directions. We thus see that solutions with traveling waves moving in different directions can exist. As a followup to the main results in Theorem~\ref{thm:main} and Theorem~\ref{thm:generalinstab}, it is natural to study the stability and instability of these solutions.

The global stability of such solutions which also satisfy Condition~\ref{cond:unexciting} is actually a corollary of Theorem~\ref{thm:main}. The smallness $\epsilon$ of the data will now have to depend on the background traveling wave along with the angles between the directions that each traveling wave is propagated along.

To make the setting more precise, we assume that we are given a semilinear system such as the one in \eqref{eq:firstversion}, and we assume that we are given a collection of unit vectors $\omega_i^j \in S^2$ where $i \in \{ 1, \dots, k \}$. For each fixed $i$, $j$ is allowed to range from $1$ to some number $N_i$. The unit vectors $\omega_i^j$ represent the $N_i$ different directions that the $i$th component of the background solution will propagate. If the $i$th component of the background solution is 0, we take $N_i = 0$ and the collection of numbers $\omega_i^j$ for that $i$ is empty.
Then, we assume that we are given a vector valued function $f$ where the $i$th component $f^i$ is a sum of functions $f^{i,j}$, each a function of $t - \omega_i^j \cdot x$ supported where $|t - \omega_i^j \cdot x| \le 1$. We assume that $f$ is a solution to the equation \eqref{eq:firstversion}. We then have the following corollary.

\begin{corollary}
As long as Condition~\ref{cond:unexciting} is satisfied in an appropriate sense, the solution $f$ is globally nonlinearly stable under sufficiently small and smooth perturbations supported in the unit ball.
\end{corollary}
We note that the required smallness will depend on the $f^i$, the $\omega^i$, and the system of equations.

We shall first describe the appropriate version of Condition~\ref{cond:unexciting}. Because we now have traveling waves moving in several different directions given by $\omega_i^j$, we must assume that the null forms $m_{i j l}$ vanish appropriately for every $\omega_a^b$. More precisely, the appropriate version of Condition~\ref{cond:unexciting} is that $m_{i j l} (d t - d \overline{x}_a^b,d \overline{y}_a^b) = m_{i j l} (d t - d \overline{x}_a^b,d \overline{z}_a^b) = 0$ for every $a$ and $b$ where $(t,\overline{x}_a^b,\overline{y}_a^b,\overline{z}_a^b)$ is a flat coordinate system coming from a spatial rotation sending the $x$ axis to $\omega_a^b$.

We shall now sketch how this result follows from Theorem~\ref{thm:main}. For every $i$, we denote by $S_t^{i,j}$ the intersection of the set $u \ge -1$ and $|t - \omega_i^j \cdot x| \le 1$. This is the intersection of the support of the perturbation and the support of a traveling wave which makes the the $i$th component of the solution and moves in the direction $\omega_i^j$. Traveling waves moving in the direction $\omega_i^j$ can only interact with the perturbation in the set $S_t^{i,j}$. We define
\[
\tilde S_t=\bigcup_{\omega_{i_1}^{j_1} \ne \omega_{i_2}^{j_2}}S_t^{i_1,j_1}\cap S_t^{i_2,j_2},
\]
that is the set where we can simultaneously interact with several traveling waves moving in different directions. The observation which allows us to reduce this problem to the one we have already studied is that there is some $T$ sufficiently large depending on $\{\omega_i^j\}$ such that $\tilde S_t$ is empty for $t\ge T$. This is simply the fact that traveling waves moving in different directions eventually separate from each other in the support of the perturbation. By picking $\epsilon$ sufficiently small in terms of $T$, $f$, and the system of equations, we can construct a solution between $t = 0$ and $t = T$ using local well-posedness results, which do not require decay.

From here, the problem can be solved piece by piece, as we have reduced ourselves to a regime where the perturbation can only be influenced by traveling waves moving in a single direction at every point. Indeed, when $t \ge T$ and $|t - \omega_i^j \cdot x| \le 1$, the equations for the perturbation look exactly the same as the equations for the perturbation in \eqref{eq:PerturbationEquation} after a rotation sending $\omega_i^j$ to the unit vector pointing along the positive $x$ axis.
This is because the plane waves traveling in the $\omega_{i'}^{j'}$ direction with $\omega_{i'}^{j'} \ne \omega_i^j$ are $0$ in this region. Thus, we can solve in the region between $t - \omega_i^j \cdot x = -1$ and $t - \omega_i^j \cdot x = 1$ taking the data inherited at $t = T$ in the same way as in Theorem~\ref{thm:main}.
After this is done for each $\omega_i^j$, we note that we now have inherited data on a Lipschitz surface $\mathcal{C}$ which is a union of null planes $t - \omega_i^j \cdot x = 1$ and the null cone $t = r - 1$. In this region, the equations for the perturbation become
\[
\Box \eta_i = \sum m_{i j l} (\nabla \eta_j,\nabla \eta_l)
\]
because $f = 0$ in this region. The problem can be solved in this remaining region using the bounds inherited from the previous step in the same way as Theorem~\ref{thm:main} is established using the fact that the renormalized perturbations $\gamma^{i,j}$ coming from solving each problem in $|t - \omega_i^j \cdot x| \le 1$ are equal to some constant times the original perturbation on $t - \omega_i^j \cdot x = 1$. This means that we have control over the appropriate quantities on $\mathcal{C}$ to solve the equation in the remaining region.

When Condition~\ref{cond:unexciting} is not satisfied but Condition~\ref{cond:M} is satisfied, we are in a setting which is analogous to Theorem~\ref{thm:generalinstab}. For $t$ sufficiently large, the traveling waves separate, and we can repeat the proof of Theorem~\ref{thm:generalinstab} in order to prove linear instability with respect to perturbations arising from data at these later times. We conjecture that the solutions are still linearly unstable with respect to perturbations arising at $t = 0$.

\bibliographystyle{abbrv}
\bibliography{sources}
\end{document}